\documentclass[reqno]{amsart}
\usepackage[dvipsnames]{xcolor}
\usepackage[normalem]{ulem}
\usepackage{amssymb,latexsym,amsfonts,amsmath,mathtools}
\usepackage{mathrsfs}
\usepackage{graphicx}
\usepackage{tikz}
\usepackage{comment}
\usetikzlibrary{calc}
\usepackage{pgfplots}
\pgfplotsset{compat=newest}
\usepackage[nomessages]{fp}
\usepackage{dsfont}
\usetikzlibrary{calc,shapes,arrows}
\usepackage{enumitem}

\newtheorem{theorem}{Theorem}[section]
\newtheorem{lemma}[theorem]{Lemma}

\newtheorem{problem}[theorem]{Problem}
\newtheorem{prop}[theorem]{Proposition}
\newtheorem{corollary}[theorem]{Corollary}

\newtheorem{defn}[theorem]{Definition}
\newtheorem{example}[theorem]{Example}

\newtheorem{assumption}[theorem]{Assumption}
\theoremstyle{remark}
\newtheorem{remark}[theorem]{Remark}

\numberwithin{equation}{section}

\newcommand{\R}{\mathbb{R}}{}
\newcommand{\N}{\mathbb{N}}

\newcommand{\cA}{\mathcal{A}}
\newcommand{\cB}{\mathcal{B}}
\newcommand{\cC}{\mathcal{C}}
\newcommand{\cD}{\mathcal{D}}
\newcommand{\cH}{\mathcal{H}}

\newcommand{\cK}{\mathcal{K}}

\newcommand{\cT}{\mathcal{T}}
\newcommand{\cS}{\mathcal{S}}

\newcommand{\cU}{\mathcal{U}}

\newcommand{\cP}{\mathcal{P}}

\newcommand{\hb}{\widehat{b}}
\newcommand{\bt}{\widetilde{b}}
\newcommand{\hi}{\widehat{\imath}}
\newcommand{\hs}{\widehat{s}}
\newcommand{\hp}{\widehat{p}}
\newcommand{\hmu}{\widehat{\mu}}
\newcommand{\heta}{\widehat{\eta}}
\newcommand{\wt}{\widetilde}

\newcommand{\wi}{\bar {\imath}}
\newcommand{\wwi}{\widehat {\imath}}
\newcommand{\ws}{\bar {s}}
 
\newcommand{\toS}{\rightsquigarrow} 
\DeclarePairedDelimiterX{\inp}[2]{\langle}{\rangle}{#1, #2}

\newcommand{\stkout}[1]{\ifmmode\text{\sout{\ensuremath{#1}}}\else\sout{#1}\fi}

\begin{document}

\title[Viability and  control of a delayed SIR epidemic]{Viability and  control of a delayed SIR epidemic\\ with an ICU State Constraint}

\author{Dimitri Breda}
\author{Matteo Della Rossa}
\author{Lorenzo Freddi}
\address[Dimitri Breda, Matteo Della Rossa, Lorenzo Freddi]{Dipartimento di Scienze Matematiche, Informatiche e Fisiche,  Universit\`a di Udine, via delle Scienze 206, 33100 Udine, Italy}
\email{dimitri.breda@uniud.it, matteo.dellarossa@uniud.it, lorenzo.freddi@uniud.it}

%\author{Dan Goreac}
%\address[Dan Goreac]{School of Mathematics and Statistics, Shandong University, Weihai, Weihai 264209, PR China}
%\address[Dan Goreac]{LAMA, Univ Gustave Eiffel, UPEM, Univ Paris Est Creteil, CNRS, F-77447 Marne-la-Valle\'ee, France}
 %\email{dan.goreac@univ-eiffel.fr}

\begin{abstract}
This paper studies viability and control synthesis for a delayed SIR epidemic. The model integrates a constant delay representing an incubation/latency time. The control inputs model non-pharmaceutical interventions, while an intensive care unit (ICU) state-constraint is introduced to reflect the healthcare system's capacity. The arising delayed control system is analyzed via functional viability tools, providing insights into fulfilling the ICU constraint through feedback control maps. 
In particular, we consider two scenarios: first, we consider the case of general continuous initial conditions. Then, as a further refinement of our analysis, we assume that the  initial conditions satisfy a Lipschitz continuity property, consistent with the considered model.
The study compares the (in general, sub-optimal) obtained control policies with the optimal ones for the delay-free case, emphasizing the impact of the delay parameter.    
The obtained results are supported and illustrated, in a concluding section, by numerical examples.
\end{abstract} 

\maketitle

% \begin{abstract}
% This paper studies viability and control synthesis for a delayed SIR epidemic. The model integrates a constant delay representing an incubation/latency time. The control inputs model non-pharmaceutical interventions, while an intensive care unit (ICU) state-constraint is introduced to reflect the healthcare system's capacity. The arising delayed control system is analyzed via functional viability tools, providing insights into fulfilling the ICU constraint through feedback control maps. 
% In particular, we consider two scenarios: first, we consider the case of general continuous initial conditions. Then, as a further refinement of our analysis, we assume that the  initial conditions satisfy a Lipschitz continuity property, consistent with the considered model.
% The study compares the (in general, sub-optimal) obtained control policies with the optimal ones for the delay-free case, emphasizing the impact of the delay parameter.    
% The obtained results are supported and illustrated, in a concluding section, by numerical examples.
% \end{abstract} 



%\tableofcontents

\section{Introduction}

The  optimal control of epidemics (see, for instance~\cite{anderson1992infectious,Behncke,hansen2011optimal,Mart}) starting from the classical SIR model of Kermack and McKendrick (\cite{kermack1927contribution}), has been widely studied in recent years, also because of the COVID-19 pandemic, see~\cite{alvarez2020simple,AFGLL23,FGLX22,Ketch,Kruse,MR2022}.  In this paper, we focus  on the viability analysis and the synthesis of control strategies for the following  two-dimensional  SIR model  with delay:
%\begin{equation}\label{eq:System}
%\begin{cases}
%\dfrac{ds}{dt}(t)=-b(t-\tau_1)\, s(t-\tau_2)\, i(t-\tau_3)\\[2ex]
%\dfrac{di}{dt}(t)= b(t-\tau_1)\, s(t-\tau_2)\, i(t-\tau_3)-\gamma i(t-\tau_3)
%\end{cases}
%\end{equation}
\begin{equation}\label{eq:System}
\begin{cases}
\dfrac{ds}{dt}(t)=-b(t)\, s(t)\, i(t-h),%\quad t\in(0,t_f)
\\[2ex]
\dfrac{di}{dt}(t)= b(t)\, s(t)\, i(t-h)-\gamma i(t).%,\quad t\in(0,t_f)
\\[2ex]
\end{cases}
\end{equation}
%{\color{blue} Here $\tau_1,\tau_2\tau_3>0$ are given numbers representing constant delays.}  
 Here $h>0$ is a  given positive number representing a constant delay, modeling the presence of an incubation/latency time, i.e., assuming a temporal gap between the disease contraction and the development of infectivity of average individuals. 
The parameter $\gamma\in(0,1)$ represents the natural (and constant) recovery rate, while the control parameters are measurable functions $b(t)\in U:=[\beta_\star,\beta]$ for some $0<\beta_\star\leq\beta<1$. The value $\beta$ models the natural transmission rate of the disease, while $\beta_\star$ represents the minimal attainable transmission rate obtained via non-pharmaceutical policies (social distancing, lockdowns, etc.) 
 The initial conditions are imposed by prescribing a pair of continuous functions $\phi=(i_\phi,s_\phi):[-h,0]\to \R^2$, and by requiring that 
 \begin{equation}\label{icphi}
 \mbox{$i(t)=i_\phi(t)$ and $s(t)=s_\phi(t)$ for all $t\in [-h,0]$.}
 \end{equation}
 The derivatives in
 \eqref{eq:System} are meant in a distributional sense and the trajectories are constructed from {\em admissible controls} $b\in \cU:=L^\infty((0,+\infty);U)$.  
%Whenever the control $b$ and the initial conditions $s(0)=s_0,\ i(0)=i_0$ are fixed,  the unique solution to \eqref{eq:System} will also be denoted by $\pr{s^{s_0,i_0,b}, i^{s_0,i_0,b}}$.
%The reader is invited to note that the trajectory can be computed for positive and negative times $t\in\mathbb{R}$.
%

The delayed version \eqref{eq:System} of the classical Kermack and McKendrick SIR model \cite{kermack1927contribution}, with a constant trasmission rate coefficient $b\equiv\beta$,
was introduced in \cite{DiLiddo86,BT1995}. This model has been widely studied, from the stability point of view, in several papers, as for instance \cite{YH2007,McC2010}. Optimization and optimal control problems have been more recently considered in~\cite{ZKJ2009} but always in the case of a constant transmission rate.
In \cite{BT1995}, the nonnegative constant $h$ represents a time during which the infectious agents develop in a vector and it is only after time $h$ that the infected vector
can infect a susceptible individual.  On the other hand, $h$ can also be regarded as a latency time only after which
the infected individual becomes able to transmit the infection. 
%With this interpretation, the class $i$ represents an \emph{exposed}, rather than infected, population (and in this case, we talk about ``SE(I)R'' models). Exposed individuals becomes infected only after time $h$. 
 In this slightly different %formulation,
 interpretation,  system~\eqref{eq:System} has been considered in the recent Covid-19 studies \cite{LMSW2020} and \cite{YYLZ2021} but, in the latter, with a constant transmission rate $b\equiv\beta$. 
We mention  \cite{Huan_etal2022,PW2022} among the recent studies about the impact of a latency period on Covid-19 transmission.

We consider system~\eqref{eq:System} under an  {\em intensive care unit (ICU) constraint} on the number of infected (cft. \cite{Kantner,Miclo,AvrFre22}), i.e., we require that 
\begin{equation}\label{ICU}
i(t)\le i_M\;\; \forall\,t\geq 0,
\end{equation}
for some fixed $0<i_M\le1$.
In this setting, the prescribed level $i_M$ represents an upper bound to the capacity of the health-care system to treat infected patients.

From a theoretical point of view, the delay system~\eqref{eq:System} together with the constraint~\eqref{ICU} can be seen as a \emph{differential delayed inclusion}, as defined in~\cite{Haddad81,HuPap95,FMM2018}. Within this approach, the analysis of~\eqref{eq:System}-\eqref{ICU} can be tackled by using the \emph{functional viability tools} introduced in~\cite{H1984} and well summarized in~\cite[Chapter 12]{Aubin2009}. For recent developments in this area see, for example,~\cite{FraMar16,FraHai18}. Functional viability analysis provides geometric characterizations of the notions of uncontrolled forward invariance and \emph{viability} (i.e., the property of being forward invariant \emph{under a suitable control}). Under some technical restrictions, these ideas allow us to provide an explicit state-space description of subsets of initial conditions for which the constraint~\eqref{ICU} is satisfied by solutions to~\eqref{eq:System}, under arbitrary or suitable controls actions.

Viability analysis, as said, implicitly highlights feasible control actions which allow to keep the solutions of~\eqref{eq:System} in the feasible region (defined by the constraint~\eqref{ICU}). Consequently, it also provides selection mechanisms for implementing such control actions, in the form of \emph{feedback control maps}. This selection procedure can also be tuned according to the minimization of a given cost functional of the form 
\begin{equation}\label{eq:CostFUnctionalIntro}
\int_0^{+\infty} G(\beta-b(t))\,dt,
\end{equation}
where $G$ is a continuous and strictly increasing scalar function. It is worth noting that, in some cases, it could be more effective to take $u(t):=\beta-b(t)$ as a control function in place of $b$, and the running cost simply becomes $G(u(t))$ (actually, we make this choice in Section \ref{Sec:Simul} which is devoted to examples and numerical simulations). Thus, summarizing, the performed viability analysis allows us to provide a (in general, sub-optimal) feedback strategy, which depends only on the current and past position of the state, for the optimal control problem~\eqref{eq:System}-\eqref{ICU}-\eqref{eq:CostFUnctionalIntro}. We formally compare our viability results  with those obtained in the delay-free case (see~\cite{AvrFre22,FG2023}), which can be recovered by taking the limit as $h \to 0$.

We point out that different approaches to the study of delayed optimal control problems have been proposed in the literature. Notably, we underline that recent researches studied various formulations of Pontryagin principles for delayed optimal control problems with state constraints, see~\cite{V2018,V2019,BV2017} and references therein. The arising conditions are  necessary for optimality of feasible control polices and rely on an advance-differential equation governing the adjoint states. On the other hand, differently from the delay-free case studied in~\cite{Fre20,AvrFre22,FG2023}, these Pontryagin-principle-based conditions have, up to now and for the problem under consideration, been unable to provide exact closed-form expression for the optimal control policy.  On the other hand, viability analysis provides itself  a satisfactory alternative and a flexible tool in building feedback control polices that only depends on the present and past values of the solutions. 

The paper is organized as follows. In Section~\ref{sec:Preliminaries} we introduce the main notation and definitions. We also provide a self-contained summary of viability analysis tools for delayed differential inclusions, basically borrowed from~\cite{H1984,Aubin2009}, together with a novel characterization result for some particular class of (functional) subsets. In Section~\ref{sec:QualitativeAnalysis} we provide a preliminary qualitative analysis of system~\eqref{eq:System}, comparing and underlying the peculiarities of this model with respect to the delay-free case. In Section~\ref{Sec:ViabilityAndControl} we provide the applications of viability analysis for system~\eqref{eq:System} under the state-constraint~\eqref{ICU}; considering two cases.
\begin{enumerate}[leftmargin=*]
\item First, we consider initial conditions as general continuous functions $\phi:[-h,0]\to \R^2$, satisfying the ICU constraint~\eqref{ICU}, and provide a complete viability analysis and the corresponding feedback control policy.
\item We then strengthen our hypothesis, supposing that initial conditions also satisfy a Lipschitz property compatible with the considered model. In this setting, we prove that when the delay converges to $0$ we are able to recover the viable regions  of the delay-free case studied in~\cite{AvrFre22,FG2023}.
\end{enumerate}
In Section~\ref{Sec:Simul} we illustrate our results with the aid of numerical examples, and some concluding remarks are provided  in Section~\ref{Sec:Conclu}. Some technical proofs are postponed in a final Appendix, to avoid breaking the flow of the presentation.

\paragraph{\textbf{Notation:}} We denote by $\R_+=[0,+\infty)$ the set of non-negative reals. Given $A\subset \R^n$ and $B\subset\R^m$, we denote by $\cC(A;B)$ and $\cC^1(A;B)$ the set of  continuous and continuously differentiable functions from $A$ to $B$, respectively.
Given $x\in \R^n$, the scalar $|x|$ is its Euclidean norm. Given two sets $A,B$, the notation $g:A\toS B$ stands for the \emph{set-valued map} $g:A\to \cP(B)$ where $\cP(B)$ is the power set of $B$.  Given a norm $\|\cdot\|$ on $\R^n$,  and $R>0$, denote by $B_{\|\cdot\|}(0,R)$ the closed ball (w.r.t. $\|\cdot\|$) of radius $R$. Given $v,w\in \R^n$, $\inp{v}{w}\in \R$ denotes the \emph{Euclidean scalar product} of $v$ and $w$. Given $x\in \R^n$, we denote by  $|x|_{\max}:=\max\{|x_1|,|x_2|,\dots, |x_n|\}$ its \emph{max (or infinity) norm}.

\section{Notation and Preliminaries}\label{sec:Preliminaries}
\subsection{The delayed SIR model: notation and existence of solutions}\label{subsec:Model}

%In this section we formally introduce the epidemic model under consideration and the main tools used in the sequel.\\
%Given a delay $h>0$ and a recovery rate $\gamma>0$, 
Let us consider the delay system introduced in~\eqref{eq:System} 
%, where $b$ is an input signal, $b\in L^\infty(\R_+; U)=\cU$ with $U=[\beta_\star, \beta]\subset (0,+\infty)$.
%We further consider 
with state constraint \eqref{ICU}.
%\begin{equation}\label{eq:StateConstraaint}
%i(t)\leq i_M\;\;\;\forall t\in \R_+,
%\end{equation}
%for some $i_M\leq 1$ which can be considered as a safety upper bound related to the capacity of the health-care system to treat infected people. The input signal $b\in \cU$ is considered as a control. The case $b(t)=\beta$ represents the case of no-action, since $\beta$ models the natural transmission rate, and $\beta_\star<\beta$ is the maximal possible control action in reducing the transmission rate.
Besides those already given in the introduction, we use the following notation: $x:=(s,i)\in \R^2$,  $\cC:=\cC([-h,0];\R^2)$. Moreover, the components of a given $\phi\in \cC$ will be denoted by  $s_\phi$ and $i_\phi$, that is, $\phi(t)=(s_\phi(t),i_\phi(t))$ for all $t\in [-h,0]$.  Let us define $f:\cC\times U\to \R^2$ by
\begin{equation}\label{eq:DefnDelSistema}
f(\phi,b):=\left( -b s_\phi(0) i_\phi(-h),\ bs_\phi(0)i_\phi(-h)-\gamma i_\phi(0) \right ),
\end{equation}
and introduce the set-valued map $F_U:\cC\toS  \R^2$ defined by
\begin{equation}\label{eq:DefnDiffInclSIR}
F_U(\phi):=\{f(\phi,b)\ \vert\ b\in [\beta_\star, \beta]\}.
\end{equation}
It can be seen that $F_U$ has convex and compact values, it is locally bounded and upper semicontinuous  
%; where, for the formal definitions of these concepts for set-valued maps, 
(we refer to~\cite{RockWets09,Aubin2009} for these general concepts).

For any $t\in \R$, consider the map 
%$S(t):\cC(\R;\, \R^2)\to \cC$ defined by
\begin{equation}\label{S(t)}
\begin{aligned}
S(t):\cC(\R;\R^2)\to &\ \cC\\
x\mapsto &\ S(t)x\\
&\ [S(t)x](s):=x(t+s)\;\; \forall\,s\in [-h,0].
\end{aligned}
\end{equation}
In the delay-systems community, the simplified notation $S(t)x:=x_t$  is often used.
As the control changes in the space $\cU$, the class of Cauchy problems~\eqref{eq:System}-\eqref{icphi} can thus be rewritten as  a \emph{delayed differential inclusion} of the form
\begin{equation}\label{eq:InfiniteDimensionalSetting}
\begin{cases}
x'(t)\in F_U(S(t)x),\\
S(0)x=\phi\in \cC.
\end{cases}
\end{equation}
According to \cite{Haddad81}, a \emph{solution} to~\eqref{eq:InfiniteDimensionalSetting} is a continuous function  $x_\phi:[-h,\tau)\to \R^2$, with $\tau>0$,  satisfying~\eqref{eq:InfiniteDimensionalSetting} for almost all $t\in  [0,\tau)$ and  such that $x_\phi$ is absolutely continuous on every compact subinterval of $[0,\tau)$. As done for the initial condition, the two components of $x_\phi$ will be denoted by  $s_\phi$ and $i_\phi$, that is, $x_\phi(t)=(s_\phi(t),i_\phi(t))$ for all $t\in [-h,\tau)$. This is consistent with the fact that $x_\phi(t)=\phi(t)$ whenever $t\in[-h,0]$.  

As a preliminary result, we show that, for a remarkable class of initial conditions, solutions exist and are globally defined.

Let us define the triangle \[
T:=\{x\in \R^2\, \vert\, s\geq 0,\, i\geq 0,\, s+i\leq 1\}
\]
and introduce the set $\cT:=\cC([-h,0]; T)$. 

\begin{lemma}\label{lemma:Well-PosednessLemma}
For any initial condition $\phi \in \cT$ and for any control input  $b\in \cU$, there exists a unique solution $x_{\phi}:[-h,+\infty)\to \R^2$ to the Cauchy problem~\eqref{eq:System}-\eqref{icphi}. Moreover,  $x_{\phi}(t)\in T$ for every $t\in[-h,+\infty)$.
We will denote this solution also with $x_{\phi,b}$ (or $x^{\phi,b}$), when we want to stress the fact that it corresponds to the prescribed input $b$. 
\end{lemma}
\begin{proof}
Given any $\phi\in \cT$ and any $b\in \cU$, the local existence and uniqueness of the  solution follows by~\cite[Chapter 2, Subsection 2.6]{Hale}. Let us denote by $x_{\phi}:[-h,\tau)\to \R^2$, $\tau>0$, such solution and prove that  $x_{\phi}(t)=(s_{\phi}(t),i_{\phi}(t))\in T$ for all $t\in [0,\tau)$. 
First we note that, by considering $i_\phi:[-h,\tau)\to\R$ as a coefficient and integrating the first equation in~\eqref{eq:System}, we have
\[
s_\phi(t)=s_\phi(0)e^{\int_0^t b(r)i_\phi(r-h)\,dr}.
\]
Since $s_\phi(0)\geq 0$ by hypothesis, this implies that   $s_\phi(t)\geq 0$ for all $t\in [0,\tau)$. Using also the non-negativity of the initial condition, we get
\[
i'_\phi(t)=  b(t)s_\phi(t)i_\phi(t-h)-\gamma i_\phi(t)\geq -\gamma i_\phi(t)\;\;\forall\, t\in [0,h)\cap [0,\tau).
\]
By a comparison argument (see for example~\cite[Lemma 1.2]{Teschl12}), the previous inequality implies
\[
i_\phi(t)\geq i_\phi(0)e^{-\gamma t}\geq 0\;\;\forall \,t\in [0,h)\cap [0,\tau).
\]
By iterating the argument on any interval of the form $[kh,(k+1)h]$ for $k\in \N$, we conclude that $i_\phi(t)\geq 0$ for all $t\in [0,\tau)$.
Now, summing the equations in~\eqref{eq:System} we have
\[
(s_\phi(t)+i_\phi(t))'= -\gamma i_\phi(t)\;\;\forall t\in [0,\tau).
\]
Since, by assumption, $i_\phi(0)+s_\phi(0)\leq 1$, this implies that $s_\phi(t)+i_\phi(t)\leq 1$ for all $t\in [0,\tau)$. We have thus proved that if $\phi\in \cT$, then $x_\phi(t)\in T$, for all $t\in [0,\tau)$. Since $T$ is compact, the results in~\cite[Theorem 3.1 and  Subsection 2.6, Chapter 2,]{Hale} imply that the solution exists in $[0,+\infty)$, and the proof is concluded.
\end{proof}
\begin{remark}
We note that, for the  model under consideration, the delayed value of the susceptible component of the state (i.e., $s(t-h)$) does not appear in the equations. 
% definition~\eqref{eq:DefnDelSistema} of the vector field. 
This implies that the solutions to the delayed Cauchy problem~\eqref{eq:System}-\eqref{icphi} only depend on the value of the $i$-component of the initial condition $\phi\in \cT$ and on $s(0)$. In other words, the evolution of the susceptible population in the interval $[-h,0)$ does not play any role in the future evolution of the epidemic.
\end{remark}

Well-posedness results, like that of Lemma~\ref{lemma:Well-PosednessLemma}, can be found also in~\cite{DiLiddo86} for the constant control case and in ~\cite{YH2007,McC2010} for slightly different delayed SIR models.
We note that the proof substantially shows that \emph{any} triangle of the form $\{x\in \R_+^2\;\vert\;s+i\leq a\}$ for $a\geq 0$ is invariant. It is physically reasonable to consider initial conditions in $\cT$, as $s_\phi(t)$ and $i_\phi(t)$ are the fractions of susceptible (resp. infectious) population  at time $t$. Thus, the preliminary Lemma~\ref{lemma:Well-PosednessLemma} can be seen as a permanence result: if the initial condition is a ``physically feasible'' curve, the solution, forward in time, remains in the region of physical feasibility, no matter the external input $b\in \cU$.

\subsection{Viability for delay systems: general theory and first results.}
We recall here the definitions of viability/forward invariance for delay systems and related results (\cite[Chapter 12]{Aubin2009},~\cite{H1984}). Given any dimension $n\in \N$, consider a set valued map $F:\cC([-h,0];\R^n)\toS \R^n$ (in this subsection, we write $\cC=\cC([-h,0];\R^n)$) and a delayed differential inclusion of the form
\begin{equation}\label{eq:InfiniteDimensionalSetting1}
\begin{cases}
x'(t)\in F(S(t)x),\\
S(0)x=\phi\in \cC.
\end{cases}
\end{equation}

For the SIR model~\eqref{eq:System} we ``a priori" know that solutions exist and are globally defined, for suitable initial conditions and controls, as proved in Lemma~\ref{lemma:Well-PosednessLemma}. For this reason and for the sake of simplicity, 
%to simplify the notation, 
in this subsection we  make the following hypothesis.

%% that for any $\phi\in \cC$ the solution set of~\eqref{eq:InfiniteDimensionalSetting1} is non-empty  and maximal solutions are defined on~$\R_+$.
\begin{assumption}\label{assumpt:ExistenceofSolution}
    For any $\phi\in \cC$  there exist at least a solution to~\eqref{eq:InfiniteDimensionalSetting1}, and maximal solutions are defined on~$\R_+$.
\end{assumption}

\begin{defn}%[Forward Invariance and Viability]
\label{defn:GeneralForwardInvatiantViability}
Given $\cK\subseteq \cC$, we say that
\begin{enumerate}%[leftmargin=*]
	\item
$\cK$ is \emph{forward invariant} for~\eqref{eq:InfiniteDimensionalSetting1} if for any $\phi\in \cK$ and \emph{any} solution $x_\phi$ to~\eqref{eq:InfiniteDimensionalSetting1},  starting at $\phi$, we have $S(t)x_\phi\in \cK$ for all $t\in \R_+$;
\item $\cK$ is \emph{viable} for~\eqref{eq:InfiniteDimensionalSetting1} if for any $\phi\in \cK$ there exists a solution $x_\phi$ to~\eqref{eq:InfiniteDimensionalSetting1},  starting at $\phi$, for which it holds that $S(t)x_\phi\in \cK$ for all $t\in \R_+$.
\end{enumerate}
\end{defn} 

The following concept of feasible directions (see~\cite[Definition 12.2.1]{Aubin2009}) is useful in providing a geometric characterization of forward invariance and viability.
 
% \begin{defn}[Set of feasible directions]\label{defn:COneFeasibleDirecitions}
% Given a subset $\cK\subseteq \cC$ and $\phi \in \cK$, we define $D_\cK(\phi)\subset \R^n$ the \emph{set of feasible directions} to $\cK$ at $\phi$ by:
% \begin{subequations}\label{eq:ConditionsFeasibleDirection}
% \begin{align}\notag
% v\in D_\cK(\phi)\;\text{ iff }\; &\forall \varepsilon>0,\; \exists\;\tau\in (0,\varepsilon)\;\text{and}\;\phi_\tau\in \cC(\R, \R^n) \;\text{such that:}\\
% & S(0)\phi_\tau=\phi,\;\;\;S(\tau)\phi_\tau\in \cK,\\
% &\frac{\phi_\tau(\tau)-\phi_\tau(0)}{\tau}\in B(v,\varepsilon),\label{eq:ConditionsFeasibleDirection22}
% \end{align}
% \end{subequations}
% where $B(v,\varepsilon)$ denotes the closed ball centered in $v$ of radius $\varepsilon$ with respect to an arbitrary norm of $\R^n$.
% \end{defn}

\begin{defn}%[Set of feasible directions]
\label{defn:COneFeasibleDirecitions}
Given a subset $\cK\subseteq \cC$ and $\phi \in \cK$, the 
\emph{set of feasible directions} to $\cK$ at $\phi$ is defined by
% \begin{subequations}\label{eq:ConditionsFeasibleDirection}
% \begin{align}\notag
%  D_\cK(\phi):=\big\{v\in\R^n\, :\ & \forall\, \varepsilon>0,\  \exists\,\tau\in (0,\varepsilon)\ \text{and}\ \phi_\tau\in \cC(\R, \R^n)\ \text{such that}\\
% & S(0)\phi_\tau=\phi,\ S(\tau)\phi_\tau\in \cK,\\
% &\frac{\phi_\tau(\tau)-\phi_\tau(0)}{\tau}\in B(v,\varepsilon)\hspace{26ex}\big\},\label{eq:ConditionsFeasibleDirection22}
% \end{align}
% \end{subequations}
% \begin{equation}\label{eq:ConditionsFeasibleDirection}
% \begin{eqalign}\notag
%  D_\cK(\phi):=\big\{v\in\R^n\, :\ & \forall\, \varepsilon>0,\  \exists\,\tau\in (0,\varepsilon)\ \text{and}\ \phi_\tau\in \cC(\R, \R^n)\ \text{such that}\\
% & S(0)\phi_\tau=\phi,\ S(\tau)\phi_\tau\in \cK,\ \frac{\phi_\tau(\tau)-\phi_\tau(0)}{\tau}\in B(v,\varepsilon)%\hspace{26ex}
% \big\},%\label{eq:ConditionsFeasibleDirection22}
% \end{eqalign}
% \end{equation}
\begin{eqnarray}
 D_\cK(\phi):=\big\{v\in\R^n\!\!\!\!\!&:\!\!\!\!& \forall\, \varepsilon>0,\  \exists\,\tau\in (0,\varepsilon)\ \text{and}\ \phi_\tau\in \cC(\R; \R^n)\ \text{such that}\nonumber\\
&& S(0)\phi_\tau=\phi,\ S(\tau)\phi_\tau\in \cK,\ \frac{\phi_\tau(\tau)-\phi_\tau(0)}{\tau}\in B(v,\varepsilon)%\hspace{26ex}
\big\},\label{eq:ConditionsFeasibleDirection}
%\label{eq:ConditionsFeasibleDirection22}
\end{eqnarray}
where $B(v,\varepsilon)$ denotes the closed ball of $\R^n$ with radius $\varepsilon$ and centered in $v$.  
%with respect to an arbitrary norm 
\end{defn}

In our setting, a more concise representation of the set of feasible directions can be given in terms of the classical notion of  Bouligand contingent cone, recalled below.
\begin{defn}%[Bouligand contingent cone]
\label{DefnBoul}
Consider $K\subset\R^n$ and $x\in K$, the \emph{Bouligand contingent cone to $K$ at $x$} is defined by
\[
T_K(x):=\{v\in \R^n\ \vert\ \exists\, t_k\to 0, t_k>0,\ \exists\, v_k\to v\text{ such that } x+t_kv_k\in K \ \forall\, k\in \N\}.
\]
\end{defn}
It is well-known that, if the set $K$ is convex, the Bouligand contigent cone coincides with the classical tangent cone of convex analysis, i.e., 
\begin{equation}\label{eq:CaracConeConvex}
T_K(x)=\text{cl}\left (\{v\in \R^n\ \vert\ \exists \,\lambda \geq 0 \text{ and } \exists\, y \in K\text{ such that }v=\lambda (y-x)\}\right),
\end{equation}
see~\cite[Theorem 6.9]{RockWets09}. If $x\in \text{int}(K)$, then $T_K(x)=\R^n$. Moreover, if the set $K\subset \R^n$ is defined by
\begin{equation}\label{Kgi}
K=\{x\in \R^n\ \vert\  g_i(x)\leq 0,\, i\in \{1,\dots, M\}\},
\end{equation}
for some $g_1,\dots, g_M\in \cC^1(\R^n;\R)$, then, for any  $x\in\partial K$ such that $g_i(x)=0$ for all $i\in I\subseteq \{1,\dots, M\}$ (and $g_j(x)\neq 0$ for $j\notin I$) and,  under the {\em constraint qualification assumption}
$$
\exists\, v_0\in\R^n\mbox{ such that } \langle \nabla g_i(x),v_0\rangle >0 \ \forall\,i\in I,
$$
we have
\begin{equation}\label{eq:COneGradient}
T_K(x)=\{v\in \R^n\ \vert\ \inp{\nabla g_i(x)}{v}\leq 0\  \forall\, i\in I\},
\end{equation}
see~\cite[Section 4.1.1]{AF2009book}.
%[Theorem 6.31]{RockWets09}. 

In the sequel, the subset $\cK\subset \cC$ will represent the set of initial conditions of the epidemic model~\eqref{eq:System}.  As anticipated in the Introduction, we are going to consider two different kinds of initial conditions, sharing the same set $K\subseteq S\subseteq \R^n$ of traces in $0$. First, we take the subset $\cK_S\subset \cC$ defined by
\begin{equation}\label{eq:SetContinuous}
\cK_S=\{\phi \in \cC\ \vert\ \phi(0)\in K,\,\phi(t)\in S \ \forall \,t\in [-h,0]\}.
\end{equation}
Besides this large set, we consider the following smaller sets of Lipschitz continuous initial conditions, i.e., given a constant $L>0$, we take
\begin{equation}\label{eq:LocLipschitzSet}
\cK_{S,L}=\{\phi\in \cC_{L,\|\cdot\|}\ \vert\ \phi(0)\in K,\,\phi(t)\in S\ \forall\,t\in [-h,0]\},
\end{equation}
where $\cC_{L,\|\cdot\|}$ denotes the set of functions $\phi\in \cC$ that are $L$-Lipschitz in $[-h,0]$ w.r.t.\ the norm $\|\cdot\|$ in $\R^n$, i.e., $\|\phi(t_1)-\phi(t_2)\|\leq L|t_1-t_2|$, for all $t_1,t_2\in [-h,0]$.
 In the following statement we characterize the set of feasible directions of Definition~\ref{defn:COneFeasibleDirecitions} for these different choices of initial conditions.

\begin{lemma}\label{lemma:ConesCharacterization}
Consider a set $K\subset \R^n$ and a convex set $S\supseteq K$. The following characterization of the set of feasible directions to $\cK_S$ in~\eqref{eq:SetContinuous} holds:
\begin{equation}\label{eq:EquivalenceSets1}
D_{\cK_S}(\phi)=T_K(\phi(0))\;\;\,\forall \phi\in \cK_S.
\end{equation}

 Moreover, given $L>0$, we have
\begin{equation}\label{eq:EquivalenceSets2}
D_{\cK_{S,L}}(\phi)=T_K(\phi(0))\cap B_{\|\cdot\|}(0,L)\;\;\,\forall \,\phi\in \cK_{S,L},
\end{equation}
where $B_{\|\cdot\|}(0,L)$ denotes the closed ball of $\R^n$ with the norm $\|\cdot\|$, having radius $L$ and center $0$.
\end{lemma}
\begin{proof}
First of all, we note that given $x\in K$ we have 
\begin{equation}\label{Eq:TangentConeCHar}
T_K(x)=\{v\in \R^n\ \vert\ \exists\, t_k\to 0,\,t_k>0,\, v_k\to v\text{ such that } \|v_k\|=\|v\|\text{ and } x+t_kv_k\in K\ \forall\, k\in \N\},
\end{equation}
i.e., we can add, without loss of generality, the constraint $\|v_k\|=\|v\|$ to the sequences of vectors in Definition~\ref{DefnBoul}. Indeed, this is clear if $v=0\in T_K(x)$. When, instead, $0\neq v\in T_K(x)$, starting from some sequences $t_k>0$ and $v_k\in \R^n\setminus\{0\}$ such that  $t_k\to 0$, $v_k\to v$ and  $x+t_kv_k\in K$, $\forall k\in \N$, the claim is proved by taking the modified sequences $\wt t_k=\frac{\|v_k\|}{\|v\|}t_k$ and $\wt v_k=\frac{\|v\|}{\|v_k\|}v_k$.
\\
Let us note, moreover, that it is enough to prove~\eqref{eq:EquivalenceSets2}, since the proof of~\eqref{eq:EquivalenceSets1} can be easily obtained from the latter, by sending $L$ to $+\infty$.
% We first prove~\eqref{eq:EquivalenceSets1}. We thus take a $\cK_S$ as in~\eqref{eq:SetContinuous} and $\phi\in \cK_S$.
%  Consider any $v\in T_K(\phi(0))$. For any $\varepsilon>0$ let $k\in \N$ be large enough such that $t_k<\varepsilon$ and $v_k\in B(v,\varepsilon)$. Define $\tau=t_k$ and, since $S$ is connected, consider any $\psi\in \cC([0,\tau],\R^n)$ such that $\psi(0)=\phi(0)$, $\psi(\tau)=\phi(0)+\tau v_k$ and $\psi(s)\in S$ for any $s\in [0,\tau]$. Then, define $\psi_\tau\in \cC([-h,\tau],\R^n)$ by $\phi \ast \psi$, where $\ast$ denotes the concatenation in time. It is easy to see that $S(0)\psi_\tau=\phi$ and $S(\tau)\psi_\tau\in \cK_S$ and, we have
% \[
% \frac{\psi_\tau(\tau)-\psi(0)}{\tau}=\frac{\psi(0)+t_kv_k-\psi(0)}{t_k}
% =v_k\in B(v,\varepsilon),
% \]
%  proving $T_K(\phi(0))\subseteq D_{\cK_S}(\phi)$.
% \\
% For $D_{\cK_S}(\phi)\subseteq T_K(\phi(0))$, take any $v\in D_{\cK_S}(\phi)$. For any $k\in \N$, consider $\varepsilon=\frac{1}{k}$  and take $\tau\in (0,\varepsilon)$ and $\psi_\tau$ as in Definition~\ref{defn:COneFeasibleDirecitions}. Define $t_k=\tau$ and $v_k=\frac{\psi_\tau(\tau)-\psi_\tau(0)}{\tau}$. Then we have $t_k\leq \frac{1}{k}\to 0$ and $v_k\to v$. Hence $v\in T_K(\phi(0)$, and the proof of~\eqref{eq:EquivalenceSets1} is concluded. 

Given $\phi\in \cK_{S,L}$, let us start by proving the inclusion $T_K(\phi(0))\cap B_{\|\cdot\|}(0,L)\subseteq D_{\cK_{S,L}}(\phi)$.
For $v\in T_K(\phi(0))\cap B_{\|\cdot\|}(0,L)$, let $t_k$ and $v_k$ be two sequences as in~\eqref{Eq:TangentConeCHar}. Given any $\varepsilon>0$ consider a $k\in \N$ large enough such that $t_k< \varepsilon$ and $\|v_k-v\|< \varepsilon$. Define $\tau=t_k$ and consider $\psi\in \cC([0,\tau];\R^n)$ defined by $\psi(s)=\phi(0)+sv_k$ for any $s\in [0,\tau]$. By convexity of $S$, $\psi(s)\in S$ for all $s\in [0,\tau]$. Moreover, $\psi$ is $L$-Lipschitz, since $\|\psi'(s)\|=\|v_k\|=\|v\|\leq L$, for any $s\in [0,\tau]$. We then define $\phi_\tau=\phi \ast \psi\in \cC([-h,\tau];\R^n)$, where $\ast$ denotes the concatenation in time. To prove that $v\in \cD_{\cK_{L,S}}(\phi)$ we have to verify that $\phi_\tau$ satisfies the conditions~\eqref{eq:ConditionsFeasibleDirection} of Definition~\ref{defn:COneFeasibleDirecitions}, with $\cK=\cK_{S,L}$. We trivially have 
$S(0)\phi_\tau=\phi$ and, moreover, $S(\tau)\phi_\tau(0)=\phi_\tau(\tau)=\phi(0)+t_k v_k\in K$. We also note that $S(\tau)\phi_\tau(s)\in S$ for all $s\in [-h,0]$ and that $S(\tau)\phi_\tau$ is $L$-Lipschitz.
We have thus proved that $S(\tau)\phi_\tau\in \cK_{S,L}$. Also the remaining condition in ~\eqref{eq:ConditionsFeasibleDirection} holds since
\[
\frac{\phi_\tau(\tau)-\phi_\tau(0)}{\tau}=\frac{\phi(0)+t_kv_k-\psi(0)}{t_k}
=v_k\in B_{\|\cdot\|}(v,\varepsilon),
\]
thus concluding the proof of the claimed inclusion.
\\
We now prove the opposite inclusion, $\cD_{\cK_{S,L}}(\phi)\subseteq T_K(\phi(0))\cap B_{\|\cdot\|}(0,L)$.  Given any $v\in \cD_{\cK_{L,S}}(\phi)$ and any $k\in \N$, consider $\varepsilon=\frac{1}{k}$, $\tau\in (0,\varepsilon)$ and $\psi_\tau$ as in Definition~\ref{defn:COneFeasibleDirecitions}. Define $t_k=\tau$ and $v_k=\frac{\psi_\tau(\tau)-\psi_\tau(0)}{\tau}$. By definition of $\cK_{S,L}$, we have
\[
\|v_k\|=\frac{\|\psi_\tau(\tau)-\psi_\tau(0)\|}{\tau}\leq \frac{L\tau}{\tau}= L.
\] 
To conclude, we recall that $t_k\leq \frac{1}{k}\to 0$ and $v_k\to v$ and thus $\|v\|\leq L$. Hence $v\in T_K(\phi(0))\cap B_{\|\cdot\|}(0,L)$, and the proof is concluded.
\end{proof}
\begin{remark}
We note that in Lemma~\ref{lemma:ConesCharacterization} we have  proven, in particular, that the sets $\cD_{\cK_S}(\phi)$ and $\cD_{\cK_S,L}(\phi)$ are independent of the set $S$, which constrains the past values of the initial conditions. Indeed, $\cD_{\cK_S}(\phi)$ and $\cD_{\cK_S,L}(\phi)$ only depends on the ``arrival'' set $K$ and on the final position $\phi(0)$ (and on the Lipschitz constant $L>0$, in the case of $\cD_{\cK_S,L}(\phi)$). This property, which holds for sets of the form~\eqref{eq:SetContinuous} and~\eqref{eq:LocLipschitzSet}, is crucial in the subsequent analysis.
\end{remark}

In the next statement we specialize the main result concerning viability theory, namely Theorem 12.2.2~\cite{Aubin2009},
to the case in which  
the set $\cK$ is of the form~\eqref{eq:SetContinuous} or~\eqref{eq:LocLipschitzSet}.

\begin{theorem}\label{prop:ViabilityAubin}
Consider a set-valued map $F:\cC \toS \R^n$ which is upper semicontinuous with nonempty, convex and compact values and suppose that Assumption~\ref{assumpt:ExistenceofSolution} holds. 
Given compact and convex sets $K\subseteq S\subset \R^n$, and $L>0$, a set $\cK\subset \cC$ of the form~\eqref{eq:SetContinuous} or~\eqref{eq:LocLipschitzSet}  is forward invariant for~\eqref{eq:InfiniteDimensionalSetting1} (resp.\  viable) if and only if 
\[
\begin{aligned}
&\;F(\phi)\subseteq D_\cK(\phi)\,\;\forall \;\phi\in \cK,\\
\text{(resp.) } & F(\phi)\cap D_\cK(\phi)\neq \varnothing\;\,\forall\;\phi\in \cK.
\end{aligned}
\]
\end{theorem}
\begin{proof}
For the case $\cK$ of the form~\eqref{eq:SetContinuous}, we apply~\cite[Corollary 1.1]{Haddad81}, recalling by Lemma~\ref{lemma:ConesCharacterization} that $\cD_\cK(\phi)=T_K(\phi(0))$ for all $\phi\in \cK$.

In the case $\cK$ defined as in~\eqref{eq:LocLipschitzSet} we follow the argument of Theorem 12.2.2 in~\cite{Aubin2009},  after having observed that the closedness of $K$ in $\R^n$ implies that $\cK$  is closed with respect to the norm of uniform convergence. We point out that, since the solutions are globally defined (Assumption~\ref{assumpt:ExistenceofSolution}), the sub-linear growth assumption of $F$ required in~\cite{Aubin2009} is not needed here. 
\end{proof}

In spite of the fact that the sets of feasible directions are independent of $S$,  the characterization of invariant (resp.\ viable) sets, in fact, depends on the past, since the conditions of Theorem \ref{prop:ViabilityAubin}
must by verified on the $(S,L)$-depending set $\cK$ of initial conditions.  As a consequence, when the set $S\subseteq \R^n$ is fixed in~\eqref{eq:SetContinuous} and~\eqref{eq:LocLipschitzSet}, a notion of maximality of the set $K$ (of the traces $\phi(0)$ of the initial conditions $\phi\in\cK$)  cannot neglect  the \emph{past} behavior of the functions $\phi$. We thus give now an appropriate definition which distinguishes between initial conditions in $\cK_{S}$ and $\cK_{S,L}$.

\begin{defn}\label{defn:MaximalViableInvariance}
Consider a set $\cH\subseteq \cC$, a set $S\subseteq \R^n$ and an upper semicontinuous set-valued map $F:\cC \toS \R^n$ with nonempty, convex and compact values. A set $K\subseteq S$ is said to be
\begin{itemize}%[leftmargin=*]
\item the \emph{maximal forward invariant set with $\cH$-past in $S$ for~\eqref{eq:InfiniteDimensionalSetting1}} if $K$ is the largest set such that
\[
\cK:=\{\phi\in \cH\ \vert\ \phi(0)\in K,\ \phi(s)\in S\ \forall s\in [-h,0]\}
\]
is forward invariant (in the sense of Definition~\ref{defn:GeneralForwardInvatiantViability}).
\item the \emph{maximal viable set with $\cH$-past in $S$ for~\eqref{eq:InfiniteDimensionalSetting1}} if $K$ is the largest set such that
\[
\cK:=\{\phi\in \cH\ \vert\ \phi(0)\in K,\ \phi(s)\in S\ \forall s\in [-h,0]\}
\]
is viable (in the sense of Definition~\ref{defn:GeneralForwardInvatiantViability}).
    \end{itemize}
    When $\cH=\cC$ the $\cH$-prefix is dropped.
\end{defn}

\subsection{Viability for the delayed SIR model}
For the set-valued map $F_U:\cC([-h,0],\R^2) \toS \R^2$ of the delayed SIR model introduced in~\eqref{eq:DefnDiffInclSIR}, Theorem~\ref{prop:ViabilityAubin} specializes in the following statement.
 \begin{corollary}\label{cor:ViabilitySIR}
Given compact and convex sets $K\subseteq S\subset \R^2$ and $L>0$, a set $\cK\subset \cC([-h,0];\R^2)$ of the form~\eqref{eq:SetContinuous} or~\eqref{eq:LocLipschitzSet}  is forward invariant (resp.\ viable) for~\eqref{eq:InfiniteDimensionalSetting} if and only if 
\[
\begin{aligned}
&f(\phi,b) \in D_\cK(\phi) \ \forall \,\phi\in \cK\ \forall\, b\in [\beta_\star,\beta], \;\,\\
\text{(resp.) } &\forall\,\phi\in \cK,\ \exists\, b\in [\beta_\star,\beta]\,: \,f(\phi,b) \in D_\cK(\phi).
\end{aligned}
\]
 \end{corollary}

\section{Qualitative Analysis of Solutions}\label{sec:QualitativeAnalysis}
In this section we collect some properties of (solutions to)  the Cauchy problem~\eqref{eq:InfiniteDimensionalSetting} under arbitrary inputs~$b\in \cU$. 
 It is worth recalling that ~\eqref{eq:InfiniteDimensionalSetting} is nothing else than an equivalent formulation for~\eqref{eq:System}-\eqref{icphi}.

% \begin{lemma}
% Consider any $b\in L_\infty(\R_+,B)$ such that $b(t)\leq b\in [\beta_\star,\beta]$ for a.a. $t\in \R_+$ (essentialy uniformly bounded by $b$). Consider and initial condition $\phi\in \cC$ such that $s_{\phi}(0)\leq \frac{\gamma}{b}$, and $i_{\phi}(s)\leq i_{\phi}(0):=i_0$, for all $s\in [-h,0]$. Then, the solution of~\eqref{eq:System} satisfies 
% \[
% i_\phi(t)\leq i_0,\;\;\;\forall\;t\in \R_+.
% \]
% \end{lemma}
% \begin{proof}
% Computing, for any $t\in [0,h]$, we have
% \[
% i'(t)\leq bs(t)i(t-h)-\gamma i(t)\leq -\gamma i(t)+b\frac{\gamma}{b}i_0-\gamma i(t)\leq -\gamma(i(t)-i_0),
% \]
% integrating we have $i(t)-i(0)\leq e^{-\gamma t} i(0)-i_0=0$, and thus $i(t)\leq i_0$ for all $t\in [0,h]$. We can iterate the reasoning for $t\in [h, 2h]$:
% \[
% i'(t)\leq bs(t)i(t-h)-\gamma i(t)\leq -\gamma i(t)+b\frac{\gamma}{b}i_0-\gamma i(t)\leq -\gamma(i(t)-i_0),
% \]
% implies 
% \[
% i(t)\leq e^{-\gamma (t-h)}(i(h)-i_0)+i_0\leq i_0
% \]
% since $i(h)\leq i_0$. We can iterate on any interval of the form $[kh,(k+1)h]$ for $k\in \N$, concluding the proof.
% \end{proof}
By Lemma \ref{lemma:Well-PosednessLemma}, we have that for every 
initial condition $\phi\in\cT$ a solution to the Cauchy problem for~\eqref{eq:InfiniteDimensionalSetting} exists and its values belong to $T$. 
Our aim is now to study the asymptotic behavior of such solutions.

First, we note that the case in which $\phi=(s_\phi,i_\phi)\in \cT$ is such that $i_\phi(t)=0$ for every $t\in [-h,0]$  is not physically relevant. Nevertheless it is  trivial, since in this case $(s_\phi(t), i_\phi(t))\equiv (s_\phi(0),0)$ for all $t\in \R_+$. 
For this reason, in the rest of this section we assume that 
\begin{equation}\label{eq:InotZero}
\text{$i_\phi\not\equiv 0$ in $[-h,0]$.}
\end{equation}
Let us call $\cT_0$ the set of all $\phi\in \cT$ satisfying~\eqref{eq:InotZero}.

Our first result is inspired by~\cite{DiLiddo86}, in which the case of a constant $b(t)$ is treated; here we adapt the argument to the time-varying case.

\begin{lemma}\label{lemma:COnvergence}
Consider any control input $b\in \cU$, and any $\phi\in \cT_0$. Then, $s_\phi$ is non-increasing in $[0,+\infty)$, and it holds that
\[
\begin{aligned}
\lim_{t\to +\infty} i_\phi(t)=0,\;\;\;
\lim_{t\to +\infty} s_\phi(t)=s_{\phi,\infty}\in [0,\frac{\gamma}{\beta_\star}),
\end{aligned}
\]
where $s_{\phi,\infty}$  depends only on $b$, $\phi(0)$ and $\int_{-h}^0i_\phi(\tau)\;d\tau$. 

Moreover,
\begin{enumerate}%[leftmargin=*]
\item if $s_\phi(0)>0$, then
\begin{enumerate} \item $s_{\phi,\infty}>0$, \item $i_\phi(t)>0$ for all $t\geq h$, \item $s_\phi$ is strictly decreasing in $[2h,+\infty)$;
\end{enumerate}
 \item if the input $b$ is eventually essentially constant, i.e., if there exist $T>0$ and $b_\infty\in [\beta_\star, \beta]$ such that $b(t)= b_\infty$ for almost all $t\in [T,+\infty)$, then $s_{\phi,\infty}< \frac{\gamma}{\,b_\infty}$.
\end{enumerate}
\end{lemma}

\begin{proof}
For every $\phi\in \cT_0$, $s_\phi(\cdot):\R_+\to \R_+$ is non-increasing since, by the invariance of $\cT$ stated in Lemma~\ref{lemma:Well-PosednessLemma}, we have  $s'_\phi(t)\leq 0$ for all $t\in \R_+$. Moreover,  $s_\phi(\cdot)$ is bounded from below by $0$ and thus the limit
\[
\lim_{t\to +\infty} s_\phi(t)=:s_{\phi,\infty}
\]
exists and belongs to $[0,1]$. Define $N_\phi(t):=s_\phi(t)+i_\phi(t)$. We have that 
\begin{equation}\label{Npgi}
N'_\phi(t)=-\gamma i_\phi(t),\ t\in\R_+,
\end{equation}
and thus also $N_\phi(t)$ is a non-increasing function and admits a non-negative limit, by Lemma~\ref{lemma:Well-PosednessLemma}. Since $i_\phi(t)=N_\phi(t)-s_\phi(t)$, also $i_\phi$ admits a limit $i_{\phi,\infty}$, which is non-negative  again by  Lemma~\ref{lemma:Well-PosednessLemma}.

On the other hand, by integrating \eqref{Npgi} on a time interval $[T,+\infty)$ with $T\ge0$, and using that $i_{\phi,\infty}\ge0$,  we get
\begin{equation}\label{eq:SumInfinityle}
\int_{T}^\infty i_\phi(\tau)\;d\,\tau=\frac{N_\phi(T)-s_{\phi,\infty}-i_{\phi,\infty}}{\gamma}\le \frac{N_\phi(T)-s_{\phi,\infty}}{\gamma}.
\end{equation}
From the last inequality we immediately get that 
$i_{\phi,\infty}=\lim_{t\to\infty}i_\phi(t)=0$, and the first equality in \eqref{eq:SumInfinityle} becomes
\begin{equation}\label{eq:SumInfinity}
\int_{T}^\infty i_\phi(\tau)\;d\,\tau=\frac{N_\phi(T)-s_{\phi,\infty}}{\gamma}\ \mbox{ for 
 every }T\ge0.
\end{equation}
% Suppose by contradiction that there exists $\varepsilon>0$ such that $\lim_{t\to +\infty} i_\phi(t)\geq \varepsilon$; by \eqref{Npgi} then we would obtain that $\lim_{t\to \infty}N_\phi(t)=-\infty$. This leads to a contradiction and, thus, $\lim_{t\to \infty}i_{\phi}(t)=0$.

To conclude the proof of the first part of the statement, it remains to show that $s_{\phi, \infty}<\frac{\gamma}{\beta_\star}$.

We use here the notation $b(t)=\beta_\star+v(t)$ with $v\in L^\infty(\R_+;[0,\beta-\beta_\star])$. 
By integrating the first equation in~\eqref{eq:System} for $t\geq h$, and performing a simple change of variable, we obtain
\[
\begin{aligned}
s_{\phi}(t)&=s_\phi(0)\,\text{exp}\left ( -\int_0^t b(\tau) i_\phi(\tau-h)\;d\tau\right )\\&=s_\phi(0)\,\text{exp}\left (-\int_{-h}^0 b(\tau+h)i_\phi(\tau)\;d\tau-\beta_\star\int_{0}^{t-h} i_\phi(\tau)\,d\tau-\int_{0}^{t-h} v(\tau+h) i_\phi(\tau)\;d\tau\right ).
\end{aligned}
\]
 As $t\to +\infty$, we have
\begin{equation}\label{eq:Limitsolution}
s_{\phi,\infty}=s_\phi(0)\,\text{exp}\left (-\int_{-h}^0 b(\tau+h)i_\phi(\tau)\;d\tau-\beta_\star\int_{0}^{\infty} i_\phi(\tau)\,d\tau-\int_{0}^{\infty} v(\tau+h) i_\phi(\tau)\;d\tau \right ).
\end{equation}
% Since $N_\phi'(t)= -\gamma i_\phi(t)$, we have
% \[
% s_\phi(t)+i_\phi(t)=N_\phi(t)=N_\phi(0)-\gamma\int_{0}^t i_\phi(\tau)\;d\,\tau
% \]
% and taking the limit as $t \to +\infty$ we have
% \[
% s_{\phi,\infty}=N_\phi(0)-\gamma\int_{0}^\infty i_\phi(\tau)\;d\,\tau.
% \]
% Thus 
% \begin{equation}\label{eq:SumInfinity}
% \int_{0}^\infty i_\phi(\tau)\;d\,\tau=\frac{N_\phi(0)-s_{\phi,\infty}}{\gamma}.
% \end{equation}
By using the equation \eqref{eq:SumInfinity} with $T=0$ to substitute the integral $\int_{0}^{\infty} i_\phi(\tau)\,d\tau$
in~\eqref{eq:Limitsolution}, we obtain
\begin{equation}\label{eq:EquationEquilibrium}
s_{\phi,\infty}=s_\phi(0)\,\text{exp}\left (-\int_{-h}^0 b(\tau+h)i_\phi(\tau)\;d\tau-\int_{0}^{\infty} v(\tau+h) i_\phi(\tau)\;d\tau -\frac{\beta_\star}{\gamma}(N_\phi(0)-s_{\phi,\infty})\right ).
\end{equation}

Equation~\eqref{eq:EquationEquilibrium}  can be written in the equivalent form
\begin{equation}\label{eq:ExplicitDefEq}
s_{\phi,\infty} e^{-\frac{\beta_\star}{\gamma}s_{\phi,\infty}}=M(\phi,b)s_\phi(0)e^{-\frac{\beta_\star}{\gamma}s_\phi(0)}
\end{equation}
with 
\[M(\phi,b):=\text{exp}\left ( -\int_{-h}^0b(\tau+h) i_\phi(\tau)\;d\tau-\int_{0}^{\infty} v(\tau+h) i_\phi(\tau)\;d\tau-\frac{\beta_\star}{\gamma}(i_\phi(0))\right ).
\] 
Since $i_\phi(t)$ is not identically zero in $[-h,0]$ then $M(\phi,b)<1$. By this inequality and \eqref{eq:ExplicitDefEq}, then we get
\begin{equation}\label{eq:ExplicitDefEq<}
s_{\phi,\infty} e^{-\frac{\beta_\star}{\gamma}s_{\phi,\infty}}<s_\phi(0)e^{-\frac{\beta_\star}{\gamma}s_\phi(0)}
\end{equation}
 Let us now introduce the function $g(z):= ze^{-\frac{\beta_\star}{\gamma}z}$, which is increasing in $[0,\frac{\gamma}{\beta_\star}]$ and decreasing in $[\frac{\gamma}{\beta_\star}, +\infty) $. If $s_\phi(0)<\frac{\gamma}{\beta_\star}$ then the claim ($s_{\phi,\infty}<\frac{\gamma}{\beta_\star}$) follows by the monotonicity of $s_\phi$. Let us consider then the case $s_\phi(0)\geq \frac{\gamma}{\beta_\star}$ and assume by contradiction that $s_{\phi,\infty}\geq \frac{\gamma}{\beta_\star}$.
  Since $s_{\phi,\infty}\leq s_\phi(0)$, by the monotonicity of $g$ in the interval $[\frac{\gamma}{\beta_\star},+\infty)$ we have $g(s_{\phi,\infty})\geq g(s_\phi(0))$, so contradicting~\eqref{eq:ExplicitDefEq<}. This completes the proof of the first part of the statement. 
  
Let us now start proving the ``moreover'' part {\sl (1)} of the statement. 

If $s_\phi(0)>0$, from~\eqref{eq:ExplicitDefEq}, we have $s_{\phi,\infty}>0$, that is \emph{(1)(a)}.  

To prove \emph{(1)(b)},
 we first show that there exists a $\tau\in [0,h]$ such that $i_\phi(\tau)>0$. Since $i_\phi$ is continuous and not identically zero in $[-h,0]$, there exist $\delta\in (0,h)$ and $\varepsilon>0$ such that $i_\phi(s)>0$ for all $s\in (-h+\delta-\varepsilon,-h+\delta)$. Since $s_\phi(0)>0$, we have already proven that $s_\phi(t)\geq s_{\phi_\infty}>0$ for all $t\in \R_+$. Suppose by contradiction that $i_\phi(s)= 0$, for all $s\in [0,\delta]$. We have
\[
i_\phi(\delta)=\int_0^\delta b(s)s_\phi(s)i_\phi(s-h)\,ds \geq \beta_\star s_{\phi,\infty} \int_{0}^\delta i_\phi(s-h)\,ds\geq \beta_\star s_{\phi,\infty} \int_{-h+\delta-\varepsilon}^{-h+\delta} i_\phi(s)\,ds>0,
\]
leading to a contradiction.
We can thus fix $\tau\in [0,h]$, such that $i_\phi(\tau)>0$. Since $i_\phi'(t)\geq -\gamma i_\phi(t)$, by the comparison principle (see~\cite[Lemma 1.1]{Teschl12}) it holds that 
\begin{equation}\label{eq:Estimation}
i_\phi(t)\geq i_\phi(\tau)e^{-\gamma(t-\tau)},
\end{equation}
thus proving that $i_\phi(t)>0$ for all $t\geq h$, that is~\emph{(1)(b)}.

Assertion~\emph{1(c)} is a direct consequence of~\emph{(1)(a)} and~\emph{(1)(b)}, because $s_\phi(t)=-b(t)s_\phi(t)i_\phi(t-h)<0$ if $t\geq 2h$.

 Let us now prove~\emph{(2)}.  Suppose that there exists $T>h$ and $b_\infty\in [\beta_\star, \beta]$ such that $b(t)=b_\infty$, for almost all $t\in [T,+\infty)$. By integrating the first equation of~\eqref{eq:System}
 on the interval $[T,t]$ for a $t\geq T+h$, we obtain
\[
\begin{aligned}
s_{\phi}(t)&=s_\phi(T)\,\text{exp}\left ( -\int_T^t b(\tau) i_\phi(\tau-h)\;d\tau\right )=s_\phi(T)\,\text{exp}\left (-b_\infty\int_{T-h}^Ti_\phi(\tau)\;d\tau-b_\infty\int_{T}^{t-h} i_\phi(\tau)\,d\tau\right).
\end{aligned}
\]
By taking the limit as $t\to +\infty$, we have
\begin{equation}\label{eq:Limitsolution1}
s_{\phi,\infty}=s_\phi(T)\,\text{exp}\left (-b_\infty\int_{T-h}^T i_\phi(\tau)\;d\tau-b_\infty\int_{T}^{\infty} i_\phi(\tau)\,d\tau\right ).
\end{equation}
By using the equation \eqref{eq:SumInfinity} to substitute the integral $\int_{T}^{\infty} i_\phi(\tau)\,d\tau$
in~\eqref{eq:Limitsolution1},  and manipulating, we see  that $s_{\phi,\infty}$  satisfies the equation
\begin{equation}\label{eq:Limitsolution11}
s_{\phi,\infty} e^{-\frac{b_\infty}{\gamma}s_{\phi,\infty}}=M_T(\phi,b_\infty)s_\phi(T)e^{-\frac{b_\infty}{\gamma}s_\phi(T)}
\end{equation}
with 
\[M_T(\phi,b_\infty):=\text{exp}\left ( -b_\infty\int_{T-h}^T i_\phi(\tau)\;d\tau-\frac{b_\infty}{\gamma}(i_\phi(T))\right )<1,
\]
since by \emph{(1)(b)} if $i_\phi(t)$ is not identically zero in $[-h,0]$, then $i_\phi(T)>0$.
 Recalling that $s_{\phi,\infty}\leq s_\phi(T)$, and arguing as in the proof of the first part of the statement we see that~\eqref{eq:Limitsolution11} implies $s_{\phi,\infty}< \frac{\gamma}{b_\infty}$.
\end{proof}

In the delay-free case, i.e., when $h=0$ the value $s=\frac{\gamma}{\beta}$ is the \emph{herd immunity threshold}, see for example~\cite{AvrFre22,FG2023} and references therein. In this case,  if $s_\phi(0)\leq \frac{\gamma}{\beta}$ and $i_\phi(0)>0$,~\eqref{eq:System} immediately implies that $i_\phi:\R_+\to \R$ is strictly decreasing.
In the subsequent statement we show that $\frac{\gamma}{\beta}$ is an important threshold also in the delay case (when $h>0$): if at the current instant the susceptible population $s$ is smaller than this value, the number of infected people $i$ will not exceed again the maximum value attained in the last $h$ time units.

\begin{lemma}\label{lemma:DecreasingProperty}
Consider any control input $b\in\cU$. Given any $\phi\in \cT_0$ and any $\tau\in \R_+$, define $\wwi\,(\phi,\tau):=\displaystyle\max_{t\in [\tau-h,\tau]}i_\phi(t)$. Then, for any $\phi\in \cT_0$ such that $s_\phi(0)\leq \frac{\gamma}{\beta}$,  we have
\begin{equation}\label{eq:DecreasingWeak}
i_\phi(t+\tau)\leq\wwi\,(\phi,\tau),\;\;\forall\, \tau\in \R_+\ \forall \,t\in [0,h].
\end{equation}
\end{lemma}
\begin{proof}
Consider any $\tau\in \R_+$, and suppose by contradiction that there exists a $\bar \tau \in [\tau,\tau+h]$ such that $i_\phi(\bar \tau)>\wwi\,(\phi,\tau)$.  Consider 
\[
\tau_1=\max\{t\in [\tau,\bar \tau]\ :\ i_\phi(t)=\wwi\,(\phi,\tau)\},
\]
which exists by continuity and because $i_\phi(\tau)\leq \wwi\,(\phi,\tau)$ and $i_\phi(\bar \tau)>\wwi\,(\phi,\tau)$. Moreover,  $\tau_1<\bar \tau$. Then, by definition of $\tau_1$, it holds that $i_\phi(t)\geq \wwi\,(\phi,\tau)$ for all $t\in [\tau_1,\bar \tau]$. We have
\[
i_\phi(\bar \tau)=i_\phi(\tau_1)+\int_{\tau_1}^{\bar \tau} b(t)s_\phi(t)i_\phi(t-h)-\gamma i_\phi(t)\,dt \leq \wwi(\phi,\tau)+\gamma \int_{\tau_1}^{\bar \tau}(i_\phi(t-h)-\wwi\,(\phi,\tau))\,dt\leq \wwi\,(\phi,\tau),
\] 
where, in the last inequality, we used that $i_\phi(t-h)\leq \wwi\,(\phi,\tau)$ for all $t\in [\tau_1,\bar \tau]$ because $t-h\in [\tau-h,\tau]$, and that $s_\phi(t)\leq s_\phi(0)\leq \frac{\gamma}{\beta}$, for all $t\geq 0$. This leads to a contradiction which concludes the proof. 
\end{proof}

In Lemma~\ref{lemma:DecreasingProperty}, rephrasing, we have proved that, for any $\phi\in \cT_0$ such that $s_\phi(0)\leq\frac{\gamma}{\beta}$, the function $t\mapsto \|S(t)i_\phi\|_\infty=\max_{s\in [0,h]}i_\phi(t-s)$ is non-increasing.
This should be compared with the delay-free case (see~\cite{AvrFre22,FG2023}), in which, if the initial condition $(s_\phi(0),i_\phi(0))$ is such that $s_\phi(0)\leq \frac{\gamma}{\beta}$ and $i_\phi(0)>0$, we have that the function $t\mapsto i_{\phi}(t)$ is strictly decreasing. In this regard, Lemma~\ref{lemma:DecreasingProperty} provides a ``weak'' decreasing property of the infectious component $i_\phi$ of the state in the delay case. We want to stress the fact that, when $h>0$, \eqref{eq:DecreasingWeak} does \emph{not} imply that $i_\phi$ be monotonic (which, in fact, might be not).

\section{Viability analysis and control}\label{Sec:ViabilityAndControl}
 Given a parameter $0<i_M\leq 1$ (representing the maximal acceptable proportion of infected), we are going to analyze the cases in which the trajectories  
of~\eqref{eq:System}/\eqref{eq:InfiniteDimensionalSetting} 
 stay (can be forced to be in) the feasible set 
\begin{equation}\label{eq:FeasibleSet}
C:=\{x\in T\;\vert\;i\leq i_M\}.
\end{equation} 
\subsection{Arbitrary feasible initial conditions}\label{subsec:Arbitrary}

In this subsection we consider the case in which the controller has a limited knowledge on the past evolution of the epidemics: we suppose that it is only known that in the past $h$ time units, the value of $i_\phi$ did not exceed the safety threshold $i_M$. More formally, we will assume the following. 

\begin{assumption}\label{assum:Minimal}
%Given a threshold $i_M\leq 1$ and the set $C=\{x\in T\;\vert\;i\leq i_M\}$ we assume that 
  The initial condition $\phi\in \cC$ 
  % of~\eqref{eq:InfiniteDimensionalSetting} {\color{red}\sout{\eqref{eq:System}}} 
satisfies $\phi(t)\in C$ for all $t\in [-h,0]$. 
\end{assumption}

Now, we are going  to introduce some additional state-space curves which characterize the (maximal) forward invariant and viable sets with past in $C$ of~\eqref{eq:System} (see  Definition~\ref{defn:MaximalViableInvariance}). In particular, we characterize the boundaries of these regions by providing the ``\emph{worst case behavior}''  for the two cases $b\equiv \beta$ and $b\equiv \beta_\star$, respectively. We will thus consider (backward) solutions of~\eqref{eq:System} with an artificial value of $i(t-h)$, fixed equal to $i_M$, as formalized in the sequel.
 The backward solution for $b=\beta$ with initial point  $s(0)=\gamma/\beta$, $i(0)=i_M$, will be proven to define the boundary $\Gamma_\beta$ of the forward invariant set, see Figure  \ref{Figure:FirstPlot}. The boundary $\Gamma_{\beta_\star}$ of the viable set, which is also drawn in the figure, will be analogously obtained by taking the backward solution for $b=\beta_\star$ starting from  $s(0)=\gamma/\beta_\star$ and  $i(0)=i_M$.

\subsubsection{Construction of the curve $\Gamma_\beta$.} 

Let us define  $\Psi_\beta=(\ws_\beta,\wi_\beta):\R_+\to \R^2$ as the solution to the linear system
\begin{equation}\label{Eq:ForwardInvSyst}
\begin{cases}
s'(t)=\beta i_Ms(t),\\
i'(t)=-\beta i_Ms(t)+\gamma i(t),
\end{cases}
\end{equation}
with initial condition $s(0)=\frac{\gamma}{\beta}$ and $i(0)=i_M$. Since the system is linear, the unique solution can be explicitly computed to be
\begin{equation}\label{esbb}
\ws_\beta(t)=\frac{\gamma}{\beta}e^{\beta i_M t},\quad 
\wi_\beta(t)=i_Me^{\gamma t}
\begin{cases}
1-t,&\mbox{ if }\beta i_M=\gamma,\\
\frac{\beta i_M-\gamma\rm{e}^{(\beta i_M-\gamma)t}}{\beta i_M-\gamma},&\mbox{ if }{\beta i_M\neq\gamma},
\end{cases}
\quad t\in\R_+.
\end{equation}
Let us denote by $T_\beta>0$ the unique solution of the equation $\wi_\beta(t)=0$
that is the unique time at which $\Psi_\beta$ crosses the $s$-axis.
It turns out to be
\begin{equation}\label{Tbeta}
T_\beta=\begin{cases}
1,&\mbox{ if }\beta i_M=\gamma,\\
\frac{1}{\beta i_M-\gamma}
\ln\frac{\beta i_M}{\gamma},
&\mbox{ if }{\beta i_M\neq\gamma}.
\end{cases}
\end{equation}
% Let us denote by $T_\beta>0$ the time at which $\Psi_\beta$ crosses the $s$-axis, that is $\wi_\beta(T_\beta)=0$,
% that is 
% $$
% T_\beta
% $$
% \[
% T_\beta:=\sup\{\bar t\geq 0\;\vert\;\wi_{\beta}(t)>0\;\,\forall\,t\in [0,\bar t)\}.
% \]
% An explicit computation shows that $T_\beta$ is finite. 
Let us denote 
\begin{equation}\label{sbtb}
\widehat s_\beta:=\ws_\beta(T_\beta)=\frac{\gamma}{\beta}
\begin{cases}
\rm{e}^{\beta i_M},&\mbox{ if }\beta i_M=\gamma,\\
\Big(\frac{\beta i_M}{\gamma}\Big)^\frac{\beta i_M}{\beta i_M-\gamma},
&\mbox{ if }{\beta i_M\neq\gamma}.
\end{cases}
\end{equation}
Since $\ws_\beta$ is strictly increasing, we can define $\Gamma_{\beta}:[\frac{\gamma}{\beta},\widehat s_{\beta}]\to \R_+$ as the function representing the curve $\Psi_{\beta}$ in $[0,T_{\beta}]$ as a graph  $i(s)= \Gamma_\beta(s)$. 
 By dividing the equations in~\eqref{Eq:ForwardInvSyst}, we note that $\Gamma_\beta$ is the solution of the Cauchy problem
\begin{equation}\label{eq:DefintionGamma}
\begin{cases}
i'(s)=-1+\frac{\gamma i(s)}{\beta i_M s},\\
i(\frac{\gamma}{\beta})=i_M.
\end{cases}
\end{equation}  By  integrating, we have
\[
\Gamma_\beta(s)=\begin{cases}cs^{\omega}+\frac{s}{\omega-1},\;\;\;&\text{if }\omega\neq 1,\\
c_1s-s\log(s),\;\;\;&\text{if }\omega=1,
\end{cases}
\]
with  $\omega=\frac{\gamma}{\beta i_M}$, $c=\big(\frac{\gamma}{\beta}\big)^{-\omega}(i_M-\frac{\gamma}{\beta(\omega-1)})$ and $c_1=1+\log(\frac{\gamma}{\beta})$.
% \[
% \Gamma_\beta(s)=cs^{\omega}+\frac{s}{\omega-1},\;\;\;\text{for}\;s\in [\frac{\gamma}{\beta},\widehat s_\beta],
% \]
% where $\omega=\frac{\gamma}{\beta i_M}$ and $c=(\frac{\gamma}{\beta})^{-\omega}(i_M-\frac{\gamma}{\beta(\omega-1)})$ if $\omega\neq 1$ or 
% $
% \Gamma_\beta(s)=c_1s-s\log(s)$,
% with $c_1=1+\log(\frac{\gamma}{\beta})$, if $\omega=1$.
% In Lemma~\ref{lemma:ConcavityGammas} in Appendix, we 
By basic calculus arguments one can
show that $\Gamma_\beta$ is strictly decreasing and concave in $[\frac{\gamma}{\beta},\widehat s_\beta]$, and thus the  hypograph $\{(s,i)\in \R^2\,\vert\;s\in [\frac{\gamma}{\beta},\widehat s_\beta],\; i\leq \Gamma_\beta(s)\}$ is convex, see Figure~\ref{Figure:FirstPlot}.

\subsubsection{Construction of the curve $\Gamma_{\beta_\star}$.}

Similarly, we denote by $\Psi_{\beta_\star}=(\ws_{\beta_\star},\wi_{\beta_\star}):\R_+\to \R^2$ the unique solution to the system (given by replacing $\beta$ with $\beta_\star$ in \eqref{Eq:ForwardInvSyst}) 
\[
\begin{cases}
s'(t)=\beta_\star i_Ms(t),\\
i'(t)=-\beta_\star i_Ms(t)+\gamma i(t),
\end{cases}
\]
with initial condition $s(0)=\frac{\gamma}{\beta_\star}$ and
$i(0)=i_M$. The solution, the interception time $T_{\beta_\star}$ with the $s$-axis, and 
$\widehat s_{\beta_\star}:=\ws_{\beta_\star}(T_{\beta_\star})$,
can be explicitly written by replacing $\beta$ with $\beta_\star$ in \eqref{esbb}, \eqref{Tbeta} and \eqref{sbtb}, respectively.

% Introducing $T_{\beta_\star}:=\sup\{\bar t\geq 0\;\vert\;\wi_{\beta^ \star}(t)>0\;\,\forall\,t\in [0,\bar t)\}$ which is finite, we define $\widehat s_{\beta_\star}:=\ws_{\beta_\star}(T_{\beta_\star})$.  It can be seen that $\widehat s_{\beta_\star}\geq \widehat s_{\beta}$. 

As before, in the interval $[0,T_{\beta_\star}]$ we represent the curve $\Psi_{\beta_\star}$  by a graph $i(s)=\Gamma_{\beta_\star}(s)$ with $\Gamma_{\beta_\star}:[\frac{\gamma}{\beta_\star},\widehat s_{\beta_\star}]\to\R_+$ defined by
\begin{equation}\label{eq:CurvaGammaStar}
\Gamma_{\beta_\star}(s)=\begin{cases}c_\star s^{\omega_\star}+\frac{s}{\omega_\star-1},\;\;\;&\text{if }\omega_\star\neq 1,\\
c_{\star,1}s-s\log(s),\;\;\;&\text{if }\omega_\star=1,
\end{cases}
% \Gamma_{\beta_\star}(s)=c_\star s^{\omega_\star}+\frac{s}{\omega_\star-1},\;\;\;\text{for}\;s\in [\frac{\gamma}{\beta_\star},\widehat s_{\beta_\star}]
\end{equation}
with $\omega_\star=\frac{\gamma}{i_M\beta_\star}$, $c_\star=(\frac{\gamma}{\beta_\star})^{-\omega_\star}(i_M-\frac{\gamma}{\beta_\star(\omega_\star-1)})$ and $c_{\star,1}=1+\log(\frac{\gamma}{\beta_\star})$.
Also $\Gamma_{\beta_\star}$ is strictly decreasing and concave in its domain. In Figure~\ref{Figure:FirstPlot} we have depicted the curve $\Gamma_{\beta_\star}$ for arbitrarily chosen values of $\gamma$, $\beta_\star$ and~$i_M$.

The next theorem will characterize the maximal forward invariant and viable sets with past in $C=\{x\in T\;\vert\;i\leq i_M\}$, 
%the set $C$ defined in~\eqref{eq:FeasibleSet}, 
in terms of the following subsets of $C$:
\begin{equation}\label{RAB}
\begin{array}{ll}
R:= ([0, \frac{\gamma}{\beta}]\times [0,i_M])\cap T, & A:=R\, \cup\,\big \{(s,i)\in T\ \vert\ s\in [\frac{\gamma}{\beta},\widehat s_\beta],\, i\leq \Gamma_\beta(s)\big\},\\
R_\star:= ([0, \frac{\gamma}{\beta_\star}]\times [0,i_M])\cap T, & 
B:=R_\star \cup\big \{(s,i)\in T\ \vert\ s\in [\frac{\gamma}{\beta_\star},\widehat s_{\beta_\star}],\, i\leq \Gamma_{\beta_\star}(s)\big\}.
\end{array}
\end{equation}
% \[
% R:= ([0, \frac{\gamma}{\beta}]\times [0,i_M])\cap T,\;\;R_\star:= ([0, \frac{\gamma}{\beta_\star}]\times [0,i_M])\cap T,
% \]
% \[
% \begin{aligned}
% A&:=R\, \cup\,\left \{(s,i)\in T\;\vert\; s\in [\frac{\gamma}{\beta},\widehat s_\beta],\,i\leq \Gamma_\beta(s)\right \},\\
% B&:=R_\star \cup\left \{(s,i)\in T\;\vert\; s\in [\frac{\gamma}{\beta_\star},\widehat s_{\beta_\star}],\,i\leq \Gamma_{\beta_\star}(s)\right \}.
% \end{aligned}
% \] 

\begin{theorem}\label{thm:ViabilityDelay}

\begin{enumerate}[leftmargin=*]
\item The set $\cA\subset \cC$ defined by
\[
\cA:=\left\{\phi\in \cC\ \vert\ \phi(0)\in A,\ \phi(t)\in C\ \forall \,t\in [-h,0]\right\}
\]
is forward invariant, and $A$ 
is the maximal forward invariant set with past in $C$, in the sense of Definition~\ref{defn:MaximalViableInvariance}.
\item  The set $\cB\subset \cC$ defined by
\[
\cB:=\left\{\phi\in \cC\ \vert\ \phi(0)\in B,\  \phi(t)\in C\ \forall\, t\in [-h,0]\right\}
\]
is viable, and $B$ is the maximal viable set with past in $C$, in the sense of Definition~\ref{defn:MaximalViableInvariance}.
\end{enumerate}
\end{theorem}

\begin{figure}[!t]
\centering
  \includegraphics[scale=0.85]{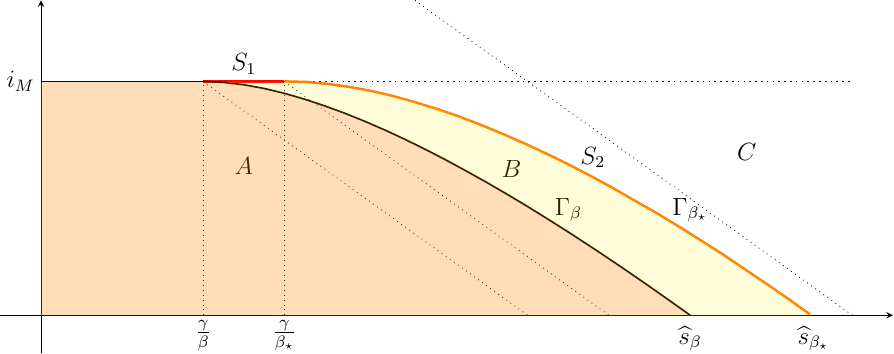}
\caption{Plot of the regions in Theorem~\ref{thm:ViabilityDelay} for the values $i_M=0.4$, $\gamma/\beta=0.2$, $\gamma/\beta_\star=0.3$. The orange region represents the maximal forward invariant set $A$ with past in $C$, while the whole colored corresponds to the set $B$, that is the maximal viable set with past in $C$. The red line $S_1$ and the orange curve $S_2$ represent the sets in which the control policy in~Theorem~\ref{thm:GreedyControlPolicy} is possibly discontinuous.} \label{Figure:FirstPlot}
\end{figure}

\begin{proof}
% The statement will be proven by applying Corollary~\ref{cor:ViabilitySIR}. 
Let us prove the forward invariance of $\cA$ by applying  Corollary~\ref{cor:ViabilitySIR} with $\cK=\cA$.
There, the invariance condition is written in terms of the set  $D_{\cA}$ of feasible directions which, in turn, is characterized in Lemma~\ref{lemma:ConesCharacterization} (with $\cK_S=\cA$, $S=C$ and $K=A$) as
$$
D_{\cA}(\phi)=T_A(\phi(0))\;\;\,\forall \phi\in \cA.
$$
Summarizing, we have to prove that 
\begin{equation} \label{icA}  
f(\phi,b) \in T_A(\phi(0)) \quad \forall \,\phi\in \cA\  \forall\, b\in [\beta_\star,\beta],
\end{equation}
where $f(\phi,b)$ is given  in~\eqref{eq:DefnDiffInclSIR}.

It is enough to consider the case $\phi(0)\in \partial A$, because if $\phi(0)\in \text{Int}(A)$ then $T_A(\phi(0))=\R^n$ and condition~\eqref{icA}  is straightforwardly satisfied.

The set $A$ has a piecewise $C^1$ boundary and can be written in the form \eqref{Kgi}. Thus, we can use the characterization of $T_A(\phi(0))$ given by \eqref{eq:COneGradient}. 
In checking \eqref{icA}, by continuity of $f$, we can restrict ourselves to consider only the $\phi(0)$ that belong 
to the (relative) interior of each piece of $\partial{A}$ and prove  that  the scalar product of $f(\phi,b)$ with a (outer) normal vector $v\ne0$ to $A$ in $\phi(0)$ turns out to be non-positive.

Let us then distinguish the following cases corresponding to the relative interior of different pieces of $\partial A$.  \begin{enumerate}[leftmargin=*]
\item $i_\phi(0)=0$.   Since 
$f(\phi,b)=(-bs_\phi(0)i_\phi(-h),\;bs_\phi(0)i_\phi(-h))$ 
and $v=(0,-1)$ is a normal vector to the halfplane $\{i\geq0\}$,
then we have
$
\inp{v}{f}=-bs_\phi(0)i_\phi(-h) \leq 0,\ \forall b\in [\beta_\star,\beta].
$
%f\in F_U(\phi)$. 
\item $s_\phi(0)=0$. We have that $f(\phi,b)=\{(0,-\gamma i_\phi(0))\}$. Since 
$v=(-1,0)$ is a normal vector to $A$ at $(s_\phi(0),i_\phi(0))$
then we have  
$
\inp{v}{f}=0$.
\item $s_\phi(0)<\frac{\gamma}{\beta}$ and $i_\phi(0)=i_M$. Since $v=(0,1)$, we have
\[
\inp{v}{f}=bs_\phi(0)i_\phi(-h)-\gamma i_\phi(0)=bs_\phi(0)i_\phi(-h)-\gamma i_M< (\frac{b}{\beta}-1)\gamma i_M\leq 0
\]
for every $b\in [\beta_\star,\beta]$.
\item $s_\phi(0)\in (\frac{\gamma}{\beta},\widehat s_\beta)$ and $i_\phi(0)=\Gamma_\beta(s_\phi(0))$. Let us denote, for simplicity, $\phi(0)=(s_0,i_0)$ and observe that any normal vector $v=(v_1,v_2)\neq 0$ to $A$ at $(s_0,i_0)$ satisfies $v_1\geq0$ and $v_2\geq0$. This vector $v$ is, by definition, perpendicular to the tangent vector $(1,\Gamma'_\beta(s_0))$. Since, by~\eqref{eq:DefintionGamma}, we have 
$\Gamma'_\beta(s_0)=-1+\frac{\gamma i_0}{\beta  s_0 i_M}$, then
$$
0=\inp{(v_1,v_2)}{(1,\Gamma'_\beta(s_0))}=v_1+\Gamma'_\beta(s_0)v_2=v_1-v_2+v_2\frac{\gamma i_0}{\beta s_0i_M}
$$
if and only if
$$
0=-(v_2-v_1)\beta s_0 i_M+v_2\gamma i_0,
$$
\noindent from which we also deduce that $v_2\geq v_1$.
% Thus, by Lemma~\ref{lemma:ConesCharacterization} we have to verify that
% \[
% \inp{v}{f}\leq 0,\;\;\forall f\in F_U(\phi).
% \]
%  Consider any $b\in [\beta_\star, \beta]$. By recalling the definition~\eqref{eq:DefnDiffInclSIR} of $F_U(\phi)$, 
Then, we have
\[
\inp{v}{f(\phi,b)}=(v_2-v_1) b s_0i_\phi(-h)-v_2\gamma i_0\leq  (v_2-v_1) \beta s_0i_M-v_2\gamma i_0=0.
\]
% \item 
% The case $\phi(0)=(\frac{\gamma}{\beta}, i_M)$ follows by the previous cases, via a continuity argument, recall~\eqref{eq:COneGradient}. 
\end{enumerate}
The proof of forward invariance of $\cA$ is thus completed.  
We now prove the maximality. Consider any set $\widetilde A\neq A$ such that $A\subset \widetilde A\subseteq C$, and suppose by contradiction that  $ \widetilde\cA:=\left\{\phi\in \cC\;\vert\;\phi(0)\in \widetilde A,\; \phi(s)\in C\;\forall s\in [-h,0]\right\}$ is forward invariant. 
We can now consider a particular $\phi_0\in \widetilde \cA$ satisfying $\phi_0(0)\in \widetilde A\setminus A$, i.e.,
\begin{equation}\label{eq:InequalityInitialCondition}
s_{\phi_0}(0)>\frac{\gamma}{\beta},\quad i_{\phi_0}(0)>\Gamma_\beta(s_{\phi_0}(0)),
{}\end{equation} 
and $i_{\phi_0}(t)=i_M$ for all $t\in [-h,-\varepsilon]$ for a  given  $\varepsilon\in (0,h)$.  Since $\widetilde \cA$ is supposed to be forward invariant, the solution $(s_{\phi_0}(t),i_{\phi_0}(t))$ of~\eqref{eq:System} corresponding to the control $b\equiv\beta$  belongs to $\widetilde A$ for every $t\in \R_+$. Moreover, in the interval $[0,h-\varepsilon]$, it coincides with the solution of the linear system (where $i_M$ replaces $i(t-h)$)
\begin{equation}\label{eq:FOrwardSystemAuxiliary}
\begin{cases}
s'(t)=-\beta i_Ms(t),\\
i'(t)=+\beta i_Ms(t)-\gamma i(t),\\
\end{cases}
\end{equation}
with initial condition $(s(0),i(0))=\phi_0(0)$. We note that the system~\eqref{eq:FOrwardSystemAuxiliary} is the time-inversion of system~\eqref{Eq:ForwardInvSyst}, i.e., it is defined by the same vector field with opposite sign. This implies that $i(s)=\Gamma_\beta(s)$ represents the graph of a  solution of~\eqref{Eq:ForwardInvSyst}, as well as~\eqref{eq:FOrwardSystemAuxiliary}. By uniqueness of solutions of~\eqref{eq:FOrwardSystemAuxiliary}, inequality~\eqref{eq:InequalityInitialCondition} implies  $i_{\phi_0}(t)>\Gamma_\beta(s_{\phi_0}(t))$ for all $t\in [0,h-\varepsilon]$. Thus $(s_{\phi_0}(t),i_{\phi_0}(t))\notin A$ for all $t\in [0,h-\varepsilon]$, while $(s_{\phi_0}(t),i_{\phi_0}(t))\in \widetilde A$,  by forward invariance of $\wt \cA$. We can now iterate the argument, considering a  $\phi_1\in \widetilde \cA$ such that $\phi_1(0)=(s_{\phi_0}(h-\varepsilon),i_{\phi_0}(h-\varepsilon))$, and $i_{\phi_1}(t)=i_M$ for $t\in [-h,-\varepsilon]$. The solution $(s_{\phi_1}(t),i_{\phi_1}(t))$ to the system~\eqref{eq:System} corresponding to the control $b\equiv \beta$ again coincides in $[0,h-\varepsilon]$, with the solution of~\eqref{eq:FOrwardSystemAuxiliary} with initial condition  $(s(0),i(0))=\phi_1(0)$. Proceeding similarly to define $\phi_k\in \widetilde \cA$,  we have that the solutions $(s_{\phi_k}(t),i_{\phi_k}(t))$ have the property that $i_{\phi_k}(t)>\Gamma_\beta( s_{\phi_k}(t))$ for all $t\in [0,h-\varepsilon]$ and all $k\in \N$. We now join all curves $\phi_k$ by considering $\bar \phi:\R_+\to \R^2$ defined by 
\[
\bar \phi(t)=(s_{\bar \phi}(t),i_{\bar \phi}(t))=\left(s_{\phi_k}(t-k(h-\varepsilon)),i_{\phi_k}(t-k(h-\varepsilon))\right)\;\;\;\text{if } t\in [k(h-\varepsilon),\,(k+1)(h-\varepsilon)],\;\;k\in \N.
\] 
By construction we have that $\bar \phi(t)\in \wt A\setminus A$ for all $t\in \R_+$, and $\bar \phi$ satisfies the Cauchy problem defined by the linear system~\eqref{eq:FOrwardSystemAuxiliary} with initial condition $\bar \phi(0)=\phi_0(0)$. 
The Cauchy problem for~\eqref{eq:FOrwardSystemAuxiliary} with initial condition  $(s(0),i(0))=\phi_0(0)$
can be explicitly solved. In particular, we have
$s_{\bar \phi}(t)=s(0)e^{-\beta i_Mt}$.
% and it can be easily seen that $\lim_{t\to +\infty}\bar \phi(t)=0$. Moreover, since $s_{\bar \phi}(t)=s(0)e^{-\beta i_Mt}$ is strictly decreasing and
Since 
$s(0)>\frac{\gamma}{\beta}$ (see~\eqref{eq:InequalityInitialCondition}),  
in $T=\frac{1}{\beta i_M}\log(\frac{s_0\beta}{\gamma})>0$ we have 
$s_{\bar \phi}(T)=\frac{\gamma}{\beta}$.
By the uniqueness of solutions of~\eqref{eq:FOrwardSystemAuxiliary}, we have therefore
\[
i_{\bar\phi}(T)>\Gamma_\beta(s_{\bar \phi}(T))=\Gamma_\beta(\frac{\gamma}{\beta})=i_M.
\] Thus, $\bar \phi(T)\notin C \supseteq \wt A$,  leading to a contradiction.

The proof of part~{\it(2)} of the statement proceeds in a similar way, that is, we prove the viability of $\cB$ by applying  Corollary~\ref{cor:ViabilitySIR} with $\cK=\cB$ (see \eqref{eq:SetContinuous})
and Lemma~\ref{lemma:ConesCharacterization} with $\cK_S=\cB$, $S=C$ and $K=B$.
In the current case, we have to prove that 
\begin{equation} \label{vcB} 
\forall\,\phi\in \cB,\ \exists\, b\in [\beta_\star,\beta]\,: \,f(\phi,b) \in T_B(\phi(0)),
\end{equation}
where $f(\phi,b)$ is given  in~\eqref{eq:DefnDiffInclSIR},
and it is enough to consider the case $\phi(0)\in \partial B$ (since otherwise $T_B(\phi(0))=\R^n$).

As before, since the set $B$ has a piecewise $C^1$ boundary  we can restrict ourselves to consider only the $\phi(0)$ that belong 
to the (relative) interior of each piece of $\partial{B}$ and prove  that  the scalar product of $f(\phi,b)$ with a normal vector $v\ne0$ to $B$ in $\phi(0)$ turns out to be non-positive.

%Let us then distinguish the following cases corresponding to the relative interior of different pieces of $\partial B$.  

We  explicitly develop the non-trivial cases only. 
%Consider any $\phi\in \cB$ such that $\phi(0)\in \partial B$. 
In the case $0<s_\phi(0)<\frac{\gamma}{\beta_\star}$ and $i_\phi(0)=i_M$, by computing the scalar product of 
%directions $f\in F_U(\phi)$ 
$f(\phi,b)$ with the normal vector $v=(0,1)$, we obtain 
\[
\inp{v}{f}=bs_\phi(0)i_\phi(-h)-\gamma i_\phi(0)=bs_\phi(0)i_\phi(-h)-\gamma i_M\leq (b\frac{\gamma}{\beta_\star}-\gamma) i_M,
\]
which is non-positive if $b\leq \beta_\star$, and thus $f(\phi,\beta_\star)\in T_B(\phi(0))$. Now suppose that $\phi(0)=(s_0,i_0)$ is such that $s_0\in (\frac{\gamma}{\beta_\star}, \widehat s_{\beta_\star})$ and $\Gamma_{\beta_\star}(s_0)=i_0$. Again, a normal vector to $B$ at $(s_0,i_0)$ is of the form $w=(w_1,w_2)\neq 0$, with $w_1,w_2\geq 0$, and, by definition of $\Gamma_{\beta_\star}$ it satisfies
\begin{equation}\label{eq:DefTangentB}
(w_2-w_1)\beta_\star s_0i_M-w_2\gamma i_0=0,
\end{equation}
which also implies $w_2\geq w_1$.
We have
\begin{equation}\label{eq:DefTangentBBB}
\inp{w}{f(\phi,\beta_\star)}= (w_2-w_1)\beta_\star s_0i_\phi(-h)-w_2\gamma i_0\leq (w_2-w_1)\beta_\star s_0i_M-w_2\gamma i_0=0,
\end{equation}
and \eqref{vcB} is satisfied.
This proves that $\cB$ is viable. The maximality follows by an argument analogous to the one used in proving~(1).
\end{proof}
Besides the state equation~\eqref{eq:System} and the state constraint~\eqref{ICU},  we  consider now a \emph{cost functional} of the form 
\begin{equation}\label{eq:CostFunctionsal}
J(u):=\int_0^\infty G(u(t))\;dt,
\end{equation}
where $u(t):=\beta-b(t)$ and $G:\R_+\to \R_+$ is a convex and strictly increasing function, satisfying $G(0)=0$.

\begin{problem}[Optimal Control Problem $\cP_\phi$.]\label{prob:OptimalControlPb}
Given the initial data $\phi\in \cC$, minimize over all admissible controls $u\in\cU$ 
\begin{itemize}
\item the cost functional \eqref{eq:CostFunctionsal},
\item under the ICU constraint \eqref{ICU}  on the trajectory of ~\eqref{eq:System}.
\end{itemize}
In the sequel we refer to this formulation as problem $\cP_\phi$.
\end{problem}

We first prove that a solution to the optimal control problem $\cP_\phi$ exists, for initial conditions in $\cC$. 

\begin{theorem}\label{thm:ExistenceOptimalControl}
 For any initial condition $\phi\in \cC$, the optimal control problem $\cP_\phi$ admits a solution.
\end{theorem}

\begin{proof} 
Let us denote for simplicity $I=(0,+\infty)$.
To prove the existence of an optimal solution we  observe that it is equivalent to prove the existence of a minimizer
of the functional $J^\infty:L^\infty(I;[0,\beta-\beta_\star])\times W^{1,\infty}(I;\R^2)\to[0,+\infty]$ defined by
\begin{equation}\label{minF}
J^\infty(u,s,i):=J(u)+\chi_{\Lambda}(u,s,i)+\chi_{i\le i_M}(i),
\end{equation}
where $\Lambda$ is the set of admissible pairs, that is all control-state vectors $(u,s,i)$ that satisfy the initial value problem for the state equation \eqref{eq:System} with initial condition $\phi\in \cC$, while $\chi_\Lambda $ denotes the indicator function of $\Lambda$ that takes the value $0$ on $\Lambda$ and $+\infty$ otherwise; similarly, the function $\chi_{i\le i_M}(i)$ is $0$ if $i(t)\le i_M$ for every $t\in I$, and $+\infty$ otherwise. 

On the domain of $J^\infty$ we consider the topology given by the product of the weak* topologies of the spaces
 $L^\infty$ and $W^{1,\infty}$, and aim to prove sequential lower semicontinuity and coercivity 
of the functional $J^\infty$ with respect to this topology. By the 
Direct Method of the Calculus of Variations {\color{black}(see, for instance, Buttazzo \cite[Sec.\ 1.2]{Buttazzo89})}, these properties imply the existence of a solution to the minimum problem. They are direct consequences of the  
fact that the space of controls is weakly* compact, that the assumptions on the integrand imply that the cost functional $J$ is weakly* lower semicontinuous (see, for instance, \cite[Theorem 5.1]{FG2023})
 and the fact that the sets $\Lambda$ {\color{black} and $\{i\le i_M\}$ are} closed with respect to the weak* convergence. 
{\color{black} The claimed closedness of such sets follows by the application of Rellich compactness theorem, which ensures that weakly* converging sequences in $W^{1,\infty}(I)$ are, up to subsequences, uniformly converging on every bounded subinterval of $I$ (see for instance \cite[Theorem 8.8 and Remark 10]{Brezis2011}).} To prove it in details, let us consider a sequence $(u_n,s_n,i_n)\in\Lambda$  weakly* converging in $L^\infty(I)\times W^{1,\infty}(I, \R^2)$ to $(u,s,i)$ and satisfying $i_n\le i_M$ for every $n\in\N$. Then, we easily get that 
$i_n(\cdot-h)=i_\phi 1_{[0,h]}+i_n1_{(h,+\infty)}$ weakly* converges to $i(\cdot-h)=i_\phi 1_{[0,h]}+i1_{(h,+\infty)}$ in $L^\infty(I)$.
Then we can pass to the limit in the state equations and, by uniqueness of the limit, we obtain that $(u,s,i)\in\Lambda$. Moreover, $i\le i_M$ by the local uniform convergence.   
\end{proof}

\begin{remark}
In proving the existence of an optimal control (Theorem \ref{thm:ExistenceOptimalControl}), we have chosen to work with the single-valued  function  formulation of the control problem (Problem \ref{prob:OptimalControlPb}). It is worth noting that the same problem can be also stated in terms of the functional differential inclusion \eqref{eq:InfiniteDimensionalSetting}. In the latter  case,  the existence follows by the compactness of the set of trajectories,  which can be proven by using, for instance, Theorem 3 (Chapter 4, Section 7) of  Aubin and Cellina \cite{AC_Diff_Inc}) (we acknowledge an anonymous referee for this remark). 
Nevertheless, to check the assumption of the mentioned Theorem 3, the semicontinuity of the cost functional and the closure of the ICU constraint,  
 we would be led to work with 
  topologies equivalent to those used in the proof of Theorem \ref{thm:ExistenceOptimalControl} and to the same computations.
\end{remark}

The performed viability analysis provides a route for designing a  state-dependent control policy (also known as state-feedback control) in order to minimize/bound the cost~\eqref{eq:CostFunctionsal} and to fulfill the state constraint~\eqref{ICU}. Indeed,  in Theorem~\ref{thm:ViabilityDelay} we have proven that, for any initial condition $\phi\in \cB$, there exists at least a control action for which the corresponding solution satisfies $S(t)x_\phi\in \cB$ for all $t\in \R_+$.  Rephrasing, we have proven that the so-called \emph{regulation map} (see~\cite[Definition 6.1.2]{Aubin2009}) 
\[
U_\cB(\phi):=\{b\in [\beta_\star, \beta]\;\vert\;f(\phi,b)\in D_\cB(\phi)\},
\]
in non-empty for every $\phi\in \cB$. We recall that, by Lemma~\ref{lemma:ConesCharacterization}, we have  $\cD_\cB(\phi)=T_B(\phi(0))$, which is a closed set (see  Definition~\ref{DefnBoul} and characterization~\eqref{eq:CaracConeConvex}). This implies  that $U_\cB(\phi)$ is compact, for every $\phi\in \cB$, since $f(\phi,\cdot):\R\to \R^2$ is continuous for every $\phi\in \cB$.
% Indeed, in  Theorem~\ref{thm:ViabilityDelay}, we proved that the set $\cB$ is viable, i.e., by Theorem~\ref{prop:ViabilityAubin}, $F_U(\phi)\cap D_\cB(\phi)\neq \varnothing$. Recalling the definition of $F_U(\phi)$ in~\eqref{eq:DefnDiffInclSIR}, this implies that there exists $b(\phi)\in [\beta_\star, \beta]$ such that $f(\phi,b(\phi))\in \cD_\cB(\phi)$.

 The so-called {\em greedy control strategy} consists in selecting, for every $\phi\in\cB$, 
 a  
 %(in general, sub-optimal) 
 control $\widetilde b(\phi)\in U_\cB(\phi)$  by (locally) minimizing the running cost $G$ in~\eqref{eq:CostFunctionsal} among all controls that keep the solution inside the viable set. Precisely, 
\begin{equation}\label{eq:GreedyAsMinimizing}
\widetilde b(\phi):=\arg\min_{b\in U_\cB(\phi)}G(\beta-b)=\max U_\cB(\phi), 
\end{equation}
where the second equality follows by the fact that the function $b\mapsto G(\beta-b)$ is strictly decreasing.

In the subsequent statement, we explicitly develop the expression in~\eqref{eq:GreedyAsMinimizing}.

\begin{lemma}\label{lemma:EquivalentDefinitionControl}
Let  $\Gamma_{\beta_\star}:[\frac{\gamma}{\beta_\star},\widehat s_{\beta_\star}]\to [0,i_M]$ be the curve defined in~\eqref{eq:CurvaGammaStar} and $B\subset C$ be the set defined in~\eqref{RAB}. Let us introduce the sets
\[
\begin{aligned}
S_1&:=\text{\emph{co}}\big\{(\frac{\gamma}{\beta},i_M),\,(\frac{\gamma}{\beta_\star},i_M)\big
\}\subset \partial B,\\
S_2&:=\big\{(s,i)\in T\ \vert\ s\in [\frac{\gamma}{\beta_\star}, \widehat s_{\beta_\star}] \text{ and } i =\Gamma_{\beta_\star}(s)\big\}\subset \partial B.
\end{aligned}
\]
The greedy control policy $\bt:\cB\to [\beta_\star,\beta]$ defined  in~\eqref{eq:GreedyAsMinimizing} turns out to be 
\begin{equation}\label{eq:ControlPolicy}
\bt(\phi)=\begin{cases}
\beta,\;\;\;&\text{if }\phi(0)\in B\setminus (S_1\cup S_2),\\
\min\big\{\beta,\frac{\gamma i_M}{s_\phi(0)i_\phi(-h)}\big \},\;\;\;&\text{if }\phi(0)\in S_1,\\
\min \big\{\beta,\;\beta_\star\frac{i_\phi(0)}{i_\phi(-h)}\big\},\;\;\;&\text{if }\phi(0)\in S_2,
\end{cases} 
\end{equation}
 with the convention $1/0=+\infty$.
\end{lemma}
\begin{proof}
The expression~\eqref{eq:ControlPolicy} is obtained  by explicitly solving the maximum problem in~\eqref{eq:GreedyAsMinimizing}.
We proceed by cases.
\begin{itemize}[leftmargin=*]
\item Let us suppose that $\phi\in \cC$ is such that $\phi(0)\in B\setminus (S_1\cup S_2)$. We first note that, by definition of $S_1$ and $S_2$, we have  $B\setminus (S_1\cup S_2)\subset A\cup \text{int}(B)$ (for a graphical illustration, see Figure~\ref{Figure:FirstPlot}). If $\phi(0)\in \text{int}(B)$ then $\cD_\cB(\phi)=T_B(\phi(0))=\R^2$ and thus, trivially,    $U_\cB(\phi)=[\beta_\star,\beta]$. If $\phi(0)\in A$, by the viability analysis provided in Theorem~\ref{thm:ViabilityDelay}, we have that, for every $b\in [\beta_\star,\beta]$,  $f(\phi,b)\in T_A(\phi(0))=D_{\cA}(\phi)\subseteq D_{\cB}(\phi)$; this implies that $U_\cB(\phi)=[\beta_\star,\beta]$. In both cases,  $\wt b(\phi)=\max [\beta_\star,\beta]=\beta$.
\item If $\phi\in \cC$ is such that $\phi(0)\in S_1$, we have that $\cD_\cB(\phi)=T_B(\phi(0))=\{(v_1,v_2)\in\R^2\;\vert\;v_2\leq 0\}$. By the expression  \eqref{eq:DefnDelSistema} of $f$,  and since $i_\phi(0)=i_M$, then we get  
\[
U_\cB(\phi)=\{b\in [\beta_\star, \beta]\;\vert\; bs_\phi(0)i_\phi(-h)-\gamma i_M\leq 0\},
\]
implying that $\bt(\phi)=\max U_\cB(\phi)=\min\{\beta, \frac{\gamma i_M}{s_\phi(0)i_\phi(-h)}\}$.
\item 
Finally, suppose $\phi\in \cC$ is such that $\phi(0)\in S_2$. As stated  in equation~\eqref{eq:DefTangentB} in the  proof of Theorem~\ref{thm:GreedyControlPolicy}, we have that
\[
\cD_\cB(\phi)=T_B(\phi(0))=\{f\in \R^2\;\vert\;\inp{w}{f}\leq 0\}
\]
where $w=(w_1,w_2)\in \R^2$ is a non-zero normal vector to $B$ (or, equivalently, to the curve $\Gamma_{\beta_\star}$) at $\phi(0)$, thus satisfying $w_2>w_1>0$ and $(w_2-w_1)\beta_\star s_\phi(0)i_M-w_2\gamma i_\phi(0)=0$, as in~\eqref{eq:DefTangentB}.
 Recalling the definition of $f$ in~\eqref{eq:DefnDelSistema}, we thus have 
\[
U_\cB(\phi)=\{b\in [\beta_\star, \beta]\;\vert\; \inp{w}{f(\phi,b)}\leq 0\}=\{b\in [\beta_\star, \beta]\;\vert\; (w_2-w_1)bs_\phi(0)i_\phi(-h)-w_2\gamma i_\phi(0)\leq 0\}.
\]
Using the fact that $(w_2-w_1)\beta_\star s_\phi(0)i_M=w_2\gamma i_\phi(0)$,  we have
\[
\bt(\phi)=\max U_\cB(\phi)=\min\{\beta,\, \frac{w_2\gamma i_\phi(0)}{(w_2-w_1)s_\phi(0)i_\phi(-h)}\}=\min\{\beta,\, \beta_\star\frac{i_\phi(0)}{i_\phi(-h)}\},
\]
as required.
\end{itemize}
\end{proof}

In the following statement we provide some important properties of the greedy control strategy and of the resulting controlled (also known as \emph{closed-loop}) solutions.

\begin{theorem}[greedy control policy and closed-loop solutions]\label{thm:GreedyControlPolicy}
Given the function $\wt F:\cB\to \R^2$ defined  by
$\wt F(\phi)=f(\phi, \bt(\phi))$ (with $f$ defined in~\eqref{eq:DefnDelSistema}), there exists a solution $x_\phi:[-h,+\infty)\to \R^2$ to 
% {\color{red} Attenzione! Con $x_\phi$ abbiamo gi\'a denotato le soluzioni di \eqref{eq:InfiniteDimensionalSetting} o di \eqref{eq:System} e relativa condizione iniziale, mentre la seguente sembra non essere la stessa cosa, altrimenti non ci sarebbe bisogno di riscriverla. Occorre quindi precisare la relazione tra le equazioni.} 
\begin{equation}\label{eq:ClosedLoop}
\begin{cases}
x'(t)=\wt F(S(t)x),\\
S(0)x=\phi\in \cB,
\end{cases}
\end{equation} 
where the operator $S(t)$ has been introduced in \eqref{S(t)}.

Given any $\phi\in \cB$, we have that
\begin{enumerate}
    \item $x_\phi(t)\in B$ for all $t\in \R_+$;
    \item  defining $\bt_\phi:\R_+\to [\beta_\star,\beta]$ by
\begin{equation}\label{eq:ControlPolicyOpenLoop}
\bt_\phi(t):=\bt(S(t)x_\phi),
\end{equation}
with $\bt$ defined as in~\eqref{eq:ControlPolicy},
it holds that 
\begin{enumerate}
\item $x_\phi$ is the unique solution to the Cauchy problem  \eqref{eq:System}-\eqref{icphi} with $b=\bt_\phi$;
\item 
$\bt_\phi$ is eventually constant equal to $\beta$;
\item the corresponding cost~\eqref{eq:CostFunctionsal} is finite.
\end{enumerate}
\end{enumerate}
Moreover, if $\phi\in \cB\cap \cT_0$ (i.e., avoiding the trivial case of $i_\phi$ identically zero in $[-h,0]$), we also have 
\begin{enumerate}
\setcounter{enumi}{2}
\item there exists a $T=T(\phi)\geq0$ such that $x_\phi(t)\in R$ (see \eqref{RAB}) for all $t\geq T$.
\end{enumerate}
\end{theorem}

\begin{remark}\label{gcsc}Let us make the following remarks.
\begin{enumerate} [leftmargin=*]  
\item The control policy $\widetilde b_\phi$ introduced in~\eqref{eq:ControlPolicyOpenLoop} can be more explicitly written as follows:
\begin{equation}\label{eq:ControlPolicyExplicitFeedback}
\bt_\phi(t)=\begin{cases}
\beta,\;\;\;&\text{if }(s_\phi(t), i_\phi(t))\in B\setminus (S_1\cup S_2),\\
\min\big\{\beta,\frac{\gamma i_M}{s_\phi(t)i_\phi(t-h)}\big \},\;\;\;&\text{if }(s_\phi(t), i_\phi(t))\in S_1,\\
\min \big\{\beta,\;\beta_\star\frac{i_\phi(t)}{i_\phi(t-h)}\big\},\;\;\;&\text{if }(s_\phi(t), i_\phi(t))\in S_2.
\end{cases} 
\end{equation}
Thus, the controller only needs to know the actual state of the epidemic (i.e., $(s_\phi(t),i_\phi(t))$) and the infected population at time $t-h$, (i.e., $i_\phi(t-h)$) to successfully implement the greedy control strategy.

\item The control $\bt(\phi)$ is well-defined for $\phi$ such that $\phi(0)\in S_1\cap S_2$. Indeed, if $\phi(0)\in S_1\cap S_2=\{(\frac{\gamma}{\beta_\star},i_M)\}$ we have $\frac{\gamma i_M}{s_\phi(0)i_\phi(-h)}=\beta_\star\frac{i_M}{i_\phi(-h)}$.
\item The case in which $\bt(\phi)=\beta_\star$  only occurs when $\phi(0)\in S_2$ and $i_\phi(-h)=i_M$. Hence, the case in which $\bt_\phi(t)=\beta_\star$ for a.e.\ $t$ in an interval $J$ only occurs when $(s_\phi(t),i_\phi(t))\in S_2$ and $i_\phi(t-h)=i_M$ for all $t\in J$. This scenario may happen only in the time interval $[0,h]$ and for a suitable initial condition. Then, in general, one cannot expect that a constant control regime $\bt_\phi=\beta_\star$ occurs, except than in some very particular cases.
\end{enumerate}
\end{remark}

\begin{proof} 
As a preliminary step to the proof of the existence of a solution to the Cauchy problem \eqref{eq:ClosedLoop}, we note  that 
\begin{equation}\label{eq:EqualityLemmaForInv}
 \wt F(\phi)
 %=f(\phi,b(\phi))
 \in \cD_\cB(\phi)\;\,\forall\,\phi\in \cB.
 \end{equation} 

%This indeed follows
%by Lemma~\ref{lemma:EquivalentDefinitionControl}, 
Indeed, for any $\phi\in\cB$
% since, by definition of $U_\cB(\phi)$, we have that if 
we have
$\bt(\phi)=\max U_\cB(\phi)\in U_\cB(\phi)$.  Thus, $\wt F(\phi)=f(\phi,\bt(\phi))\in \cD_\cB(\phi)$ by definition of $U_\cB(\phi)$.

 Under this condition, the  existence of solutions $x_\phi$ to~\eqref{eq:ClosedLoop} is proved  in Lemma~\ref{lemma:WellPosednessSoltuions} in Appendix, together with the property  $S(t)x_\phi\in \cB$, for all $t\in \R_+$, which implies $S(t)x_\phi(0)=x_\phi(t)\in B$, for all $t\in \R_+$. 
 Then, part \emph{(1)} of the statement is proved.
 
To prove {\em(2)(a)}, it is enough to observe that
$$
\wt F(S(t)x_\phi)=
f(S(t)x_\phi,\bt_\phi(t)) 
=(-\bt_\phi(t)s_\phi(t)i_\phi(t-h)\, ,\,\bt_\phi(t)s_\phi(t)i_\phi(t-h)-\gamma i_\phi(t)), 
$$
which means that the solution $x_\phi$ to \eqref{eq:ClosedLoop}  
is also a solution to \eqref{eq:System}-\eqref{icphi} with $b=\bt_\phi$, and such Cauchy problem has a unique solution by Lemma \ref{lemma:Well-PosednessLemma}. Of course, this 
implies (a posteriori) that 
also \eqref{eq:ClosedLoop} has a unique solution.

Let us now prove~\emph{(2)(b)}. In Lemma~\ref{lemma:COnvergence} we have proven that 
$\lim_{t\to +\infty}i_\phi(t)=0$ and $s_{\phi,\infty}=\lim_{t\to +\infty}s_\phi(t)\in [0,\frac{\gamma}{\beta_\star})$. Then, there exists a $T_1(\phi)\geq 0$ such that $x_\phi(t)\notin S_1\cup S_2$ for all $t\geq T_1(\phi)$. By~\eqref{eq:ControlPolicy} and \eqref{eq:ControlPolicyOpenLoop}, this implies that $b_{\phi}(t)=b(S(t)x_\phi)=\beta$ for all $t\geq T_1(\phi)$, since $S(t)x_\phi(0)=x_\phi(t)\notin S_1\cup S_2$. This proves~\emph{(2)(b)}.

\emph{(2)(c)} holds since
\[
\int_0^{+\infty}G(\beta-\bt_\phi(t))\,dt=\int_0^{T_1(\phi)}G(\beta-\bt_\phi(t))\,dt\leq \int_0^{T_1(\phi)}G(\beta-\beta_\star)\,dt=T_1(\phi)G(\beta-\beta_\star).
\] 
To prove part \emph{(3)}, suppose $\phi\in \cB\cap \cT_0$. Since we have proven that $b_\phi$ is eventually equal to $\beta$, by part \emph{(2)} of Lemma~\ref{lemma:COnvergence} we have that  $s_{\phi,\infty}\in [0,\frac{\gamma}{\beta})$, which in turns implies that there exists a $T=T(\phi)\geq 0$ such that $s_\phi(t)\leq \frac{\gamma}{\beta}$ (and thus $x_\phi(t)\in R$) for all $t\geq T$, and the proposition is completely proved.
\end{proof}
% The fact that any solution eventually reaches the set $A$ follows by Lemma~\ref{lemma:COnvergence}. We prove that, for any $\phi\in \cB$ there exists a $T_\phi>0$ for which $b(S(t)x_\phi)=\beta$ for all $t\geq T_\phi$. Indeed, for times $t\in \R_+$ at which $b(S(t)x_\phi)\neq\beta$ we have $2$ possibilities:
% either $
% x_\phi(t)\in S_2$,
% and $s_\varphi(t)=-\beta_\star\frac{i_M}{i_\phi(-h)}s_\phi(t)i_\phi(t-h)\leq - \gamma i_M$, or $x_\phi(t)\in S_1$ and $s_\varphi(t)=-\frac{\gamma i_M}{s_\phi(t)i_\phi(t-h)}s_\phi(t)i_\phi(t-h)=-\gamma i_M$.
% }
% Thus, if $x_\phi(t)\in S_1\cup S_2$ and $b(S(t)x_\phi)\neq \beta$, $s'_\phi(t)\leq -\gamma i_M$. Since for all $S_1\cup S_2\subset \{(s,i)^\top\in T\;\vert\;s\geq \frac{\gamma}{\beta}\}$.
A first direct consequence of the existence of the ``greedy'' feedback control policy defined in~\eqref{eq:ControlPolicy} is stated below. 
\begin{corollary}
For any prescribed initial condition $\phi\in \cB$, the optimal control problem $\cP_\phi$  admits a solution with a finite cost.
\end{corollary}
\begin{proof}
The proof directly follows by Theorem~\ref{thm:ExistenceOptimalControl} and assertion~\emph{(3)} of Theorem~\ref{thm:GreedyControlPolicy}. 
\end{proof}

\begin{remark}
We note that 
% the results in Theorem~\ref{thm:ViabilityDelay} and the induced control policy described in Theorem~\ref{thm:GreedyControlPolicy} do \emph{not} depend on the delay parameter $h>0$.   For that reason, 
the invariance/viability regions in Theorem~\ref{thm:ViabilityDelay} 
 are independent of $h$. In particular, they do not converge, as $h\to 0$, to the regions obtained in the delay-free case (see~\cite{Fre20,AvrFre22,FG2023}), recalled also in the subsequent Theorem \ref{lemma:ViabilityDelayFree}. In the next subsection we strenghten  Assumption~\ref{assum:Minimal} by bounding the velocity of the past evolution of the epidemic. This allows us to obtain a viability analysis and a control policy which do depend on the parameter $h>0$, and that converge, in a sense that will be clarified, to the solutions provided for the delay-free case.
\end{remark}

\subsection{Lipschitz continuous initial conditions}\label{Subsec:locallyLispchitz}

In the previous subsection we supposed to have limited information on the past evolution of the epidemic, only considering Assumption~\ref{assum:Minimal}. We hereafter assume that, in the past $h$ time units, the epidemic dynamic not only was under the warning level (i.e. $i_\phi(t)\leq i_M$ for all $t\in[-h,0]$), but also was evolving with a limited  ``speed'' compatible with the epidemic model. More formally we assume the following. 
\begin{assumption}\label{assum:SEcondAss}
Given a threshold $0<i_M\leq 1$ consider any $L\in \R$ such that 
\begin{equation}\label{eq:BoundLipschitzConstant}
L\geq i_M\max\{\beta,\gamma\}>0.
\end{equation} 
We assume that the initial condition $\phi\in \cC$ of~\eqref{eq:System} satisfies
\[
 \phi(t)\in C\;\;\; \forall \;t\in [-h,0],
\]
with $C=\{(s,i)\in T\;\vert\;i\leq i_M\}$, and
\[
\phi\in \cC_{L,|\cdot|_{\max}}=\{\varphi\in \cC\;\vert\;|\varphi(t_1)-\varphi(t_2)|_{\max}\leq L|t_1-t_2|\;\;\forall t_1,t_2\in [-h,0]\},
\]
where $|x|_{\max}:=\max\{|x_1|,|x_2|\}$, for any $x=(x_1,x_2) \in \R^2$.
\end{assumption}

\begin{remark}
    The choice of the infinity/max norm is motivated by the fact that the dynamics in~\eqref{eq:System} are only affected by the delayed value of the $i$-component ($i(-h)$), while the delayed $s$-component ($s(-h)$) does not play any role. The $\max$ norm, which, intuitively, is only affected by the ``worst case''  of any component (considered separately), is thus a natural choice. Nevertheless, 
the subsequent analysis can  be adapted to the choice of any other  norm in $\R^2$.
\end{remark}

\begin{remark}
The lower bound~\eqref{eq:BoundLipschitzConstant} on the Lipschitz constant $L>0$ is motivated by noting that, for any $x=(s,i)\in C$ such that  $i\leq i_M$ we have
\begin{equation*}
|-\beta s i|\leq \beta i_M\,\;\text{ and }\,\;
|\beta si-\gamma i|\leq i_M\max\{\beta ,\gamma \}.
\end{equation*}
Thus,  
\begin{equation}\label{eq:InequalityLipschitzConstant}
|f(\phi,b)|_{\max}\leq L\;\;\;\forall \,\phi\in \cC,\;\forall \,b\in U,
\end{equation}
and, hence,
$L\geq i_M\max\{\beta,\gamma\}$
represents a uniform upper bound to the speed modulus of the solution to~\eqref{eq:System} with initial condition $\phi\in \cC$ such that $\phi(s)\in C$, for all $s\in [-h,0]$.
Intuitively, in Assumption~\ref{assum:SEcondAss}, we are supposing that the epidemic, in the past uncontrolled interval of time $[-h,0]$, has evolved at a bounded speed, and this bound is assumed to be not smaller than the one holding forward in time, according with the model. 
\end{remark}

%\begin{remark}\label{rem:LocLipAuxiliaryCurves}
In order to retrace the analysis performed in Subsection~\ref{subsec:Arbitrary}, we need to introduce auxiliary (non-delayed) systems that mimic the ``worst-case'' behaviour of the delay system~ \eqref{eq:System}. Since the considered initial conditions are supposed to satisfy  Assumption~\ref{assum:SEcondAss}, we first define a scalar function that models the maximal gap between $i(0)$ and $i(-h)$. Namely,  for any $L,h>0$, we consider the Lipschitz continuous function $\psi_{L,h}:\R\to \R$ (see Figure~\ref{Figure:PlotPsi}) defined by
\begin{equation}\label{tfLh}
\psi_{L,h}(i):=\begin{cases}
-i_M,\;\;\;&\text{if }  i\leq -i_M-Lh,\\
i+Lh,\;\;\;&\text{if }  -i_M-Lh< i\leq i_M-Lh,\\
i_M,\;\;\;&\text{if } i_M-Lh< i.
\end{cases}
\end{equation}
%\end{remark}

\begin{figure}[h]%[!t]
\centering
  \includegraphics[scale=1]{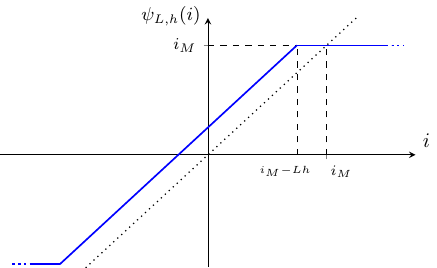}
\caption{The truncating function $\psi_{L,h}$ defined in~\eqref{tfLh}.}
\label{Figure:PlotPsi}
\end{figure}

 As done in Subsection \ref{subsec:Arbitrary},  we are now going to  introduce  additional state-space curves which characterize the (maximal) forward invariant and viable sets with past in $C$ of~\eqref{eq:System} (see  Definition~\ref{defn:MaximalViableInvariance}). Actually, we characterize the boundaries of these regions by providing the ``\emph{worst case behavior}''  for the two cases $b\equiv \beta$ and $b\equiv \beta_\star$, respectively. We will thus consider (backward) solutions of~\eqref{eq:System} with an artificial value of $i(t-h)$, fixed equal to $\psi_{L,h}(i(t))$, as formalized in the sequel. It is worth noting that if $Lh>i_M$ than we have $\psi_{L,h}(i)=i_M$ for any $i>0$ and we recover exactly what we done in Subsection \ref{subsec:Arbitrary}.  
 The backward solution for $b=\beta$ with initial point  $s(0)=\gamma/\beta$, $i(0)=i_M$, will be proved to define the boundary $\Gamma_{\beta,L,h}$ of the forward invariant set, in orange in Figure  \ref{Figure:SecondPlot}. An analogous construction will be done for $b=\beta_\star$ to obtain the boundary $\Gamma_{\beta_\star,L,h}$ of the maximal viable sets.

\subsubsection{Construction of the curve $\Gamma_{\beta,L,h}$.} 
Let us consider the non-linear differential equation
\begin{equation}\label{eq:Non-LinearBeta}
\begin{cases}
s'(t)=\beta s(t) \psi_{L,h}(i(t)),\\
i'(t)=-\beta s(t) \psi_{L,h}(i(t))+\gamma i(t).
\end{cases}
\end{equation}
The right-hand side in~\eqref{eq:Non-LinearBeta} is locally Lipschitz, and since $|\psi_{L,h}(i)|\leq i_M$ for all $i\in \R$, we also have that
\begin{equation}\label{Eq:subLinearBoundSIR}
|(\beta s \psi_{L,h}(i)\, ,\,-\beta s\psi_{L,h}(i)+\gamma i)|_{\max}\leq A|(s,i)|_{\max}, \;\;\forall (s,i)\in \R^2,
\end{equation}
for a suitable $A>0$.
This  implies that, for any prescribed initial condition, the solution to~\eqref{eq:Non-LinearBeta} is unique and globally defined, see~\cite[Theorem 2.17]{Teschl12}.
We thus denote by $\Psi_{\beta,L,h}=(\ws_{\beta,L,h}, \,\wi_{\beta,L,h}):\R_+\to \R^2$ the solution to~\eqref{eq:Non-LinearBeta} corresponding to the initial condition $(s(0),i(0))=(\frac{\gamma}{\beta}, i_M)$. 
 We define
\[
T_{\beta,L,h}:=\sup\{\bar t\geq 0\;\vert\;\wi_{\beta,L,h}(t)>0\;\,\forall\,t\in [0,\bar t)\}.
\]
 Let us denote $\widehat s_{\beta,L,h}:=\lim_{t\to T_{\beta,L,h}^-}\ws_{\beta,L,h}(t)$.
Since the solutions to~\eqref{eq:Non-LinearBeta} satisfy $s(t)=s_0 \,e^{\beta  \int_0^t\psi_{L,h}(i(\tau))\,d\tau}$ and $\psi_{L,h}(i)>0$ for any $i\in \R_+$, we have that $\ws_{\beta,L,h}$ is strictly positive and strictly increasing in $[0,T_{\beta,L,h})$.

Thus, we can define $\Gamma_{\beta,L,h}:[\frac{\gamma}{\beta},\widehat s_{\beta,L,h})\to \R_+$ as the function representing the curve $\Psi_{\beta,L,h}$ in $[0,T_{\beta,L,h})$ as a graph  $i(s)=\Gamma_{\beta,L,h}(s)$. 
 In Lemma~\ref{lemma:ConvexityAuxiliaryLocallyLipschitz} in Appendix we show that $\Gamma_{\beta,L,h}$ is strictly decreasing and concave in $[\frac{\gamma}{\beta},\widehat s_{\beta,L,h})$.
 Since  $\lim_{s\to \widehat s_{\beta,L,h}^{\,-}}\Gamma_{\beta,L,h}(s)\geq 0$,  we have that  $\widehat s_{\beta,L,h}$ is finite. 
 In turn, this implies that $T_{\beta,L,h}<+\infty$. Indeed, suppose by contradiction that $T_{\beta, L,h}=+\infty$; then $\ws_{\beta,L,h}'(t)=\beta \ws_{\beta,L,h}(t)\psi_{L,h}(\wi_{\beta,L,h}(t))\geq \gamma Lh$ for all $t\in (0,T_{\beta, L,h})=(0,+\infty)$, in contradiction with $\widehat s_{\beta,L,h}=\lim_{t\to T_{\beta,L,h}^-}\ws_{\beta,L,h}(t)<+\infty$.  
Moreover, recalling that $\ws_{\beta,L,h}(t)$ is continuous and strictly increasing, $T_{\beta,L,h}<+\infty$   implies that
 \[
\lim_{s\to \widehat s_{\beta,L,h}^-}\Gamma_{\beta,L,h}(s)=\lim_{t\to T_{\beta,L,h}^-}\Gamma_{\beta,L,h}(\ws_{\beta,L,h}(t))=\lim_{t\to T_{\beta,L,h}^-}\wi_{\beta,L,h}(t)=0,
 \]
 where the last equality follows by definition of $T_{\beta,L,h}$.
 We can thus extend  $\Gamma_{\beta,L,h}$ on the closed interval $[\frac{\gamma}{\beta},\widehat s_{\beta,L,h}]$ by setting $\Gamma_{\beta,L,h}(\widehat s_{\beta,L,h})=0$. See Figure~\ref{Figure:SecondPlot}
 for a graphical representation of such curve.
 
\subsubsection{Construction of the curve $\Gamma_{\beta_\star,L,h}$.}  Similarly, denote by $\Psi_{\beta_\star,L,h}=(\ws_{\beta_\star,L,h}, \,\wi_{\beta_\star,L,h}):\R_+\to \R^2$  the solution to the differential equation
\begin{equation}\label{eq:Non-LinearBetaStar}
\begin{cases}
s'(t)=\beta_\star s(t) \psi_{L,h}(i(t))\\
i'(t)=-\beta_\star s(t) \psi_{L,h}(i(t))+\gamma i(t)
\end{cases}
\end{equation}
with initial condition $(s(0),i(0))=(\frac{\gamma}{\beta_\star}, i_M)$. We consider
$
T_{\beta_\star,L,h}:=\sup\{\bar t\geq 0\;\vert\;\wi_{\beta^ \star,L,h}(t)>0\;\,\forall\,t\in [0,\bar t)\}
$, define $\widehat s_{\beta_\star,L,h}:=\lim_{t\to T_{\beta_\star,L,h}^- }\ws_{\beta_\star,L,h}(t)$, and represent the  curve $\Psi_{\beta_\star,L,h}:[0,T_{\beta_\star,L,h})\to \R^2$ as a graph $i(s)=\Gamma_{\beta_\star,L,h}(s)$, with $\Gamma_{\beta_\star,L,h}:[\frac{\gamma}{\beta_\star},\widehat s_{\beta_\star,L,h})\to \R_+$, see Figure~\ref{Figure:SecondPlot}. In Lemma~\ref{lemma:ConvexityAuxiliaryLocallyLipschitz}, it is proved that  $\Gamma_{\beta_\star,L,h}$ is concave and strictly decreasing, also implying that $\widehat s_{\beta_\star,L,h}$ is finite. Arguing as before,
we can thus extend  $\Gamma_{\beta_\star,L,h}$ by continuity on $[\frac{\gamma}{\beta_\star},\widehat s_{\beta_\star,L,h}]$ by setting $\Gamma_{\beta_\star,L,h}(\widehat s_{\beta_\star,L,h})=0$.

 We can now provide the viability result.
 
\begin{theorem}\label{lemma:ViabilityDelayLipschitz}
Consider the sets $R, R_\star\subset C$  defined in~\eqref{RAB}, and
 define the sets $A_{L,h}, B_{L,h}\subset C$ by
\begin{equation}\label{eq:DefinitionAElleAcca}
\begin{aligned}
A_{L,h}&:=R \cup\Big \{(s,i)\in T\;\vert\; s\in [\frac{\gamma}{\beta},\widehat s_{\beta,L,h}],\;\,i\leq \Gamma_{\beta,L,h}(s)\Big \},\\
B_{L,h}&:=R_\star \cup\Big \{(s,i)\in T\;\vert\; s\in [\frac{\gamma}{\beta_\star},\widehat s_{\beta_\star,L,h}],\;\,i\leq \Gamma_{\beta_\star,L,h}(s)\Big\}.
\end{aligned}
\end{equation}
The following propositions hold.
\begin{enumerate}%[leftmargin=*]
\item $A_{L,h}$ is the maximal forward invariant set with $\cC_{L,|\cdot|_{\max}}$-past in $C$.
\item   $B_{L,h}$ is the maximal viable set with $\cC_{L,|\cdot|_{\max}}$-past in $C$.
\end{enumerate}
\end{theorem}
\begin{proof}
The idea behind the proof is essentially the same of the proof of Theorem~\ref{thm:ViabilityDelay}, and it relies on 
Lemma~\ref{lemma:ConesCharacterization} and Corollary~\ref{cor:ViabilitySIR}.
We start by proving that $\cA_{L,h}=\left\{\phi\in \cC_{L,|\cdot|_{\max}}\;\vert\;\phi(0)\in A_{L,h},\;\; \phi(s)\in C\;\,\forall\, s\in [-h,0]\right\}$ is forward invariant. As always, it is enough to consider $\phi\in \cA_{L,h}$ such that $\phi(0)\in \partial A_{L,h}$ and prove that 
$F_U(\phi)\subseteq T_{A_{L,h}}(\phi(0))\cap B_{|\cdot|_{\max}}(0,L)$. 
By~\eqref{eq:InequalityLipschitzConstant}, it sufficies to prove that 
\begin{equation}\label{FUTB}
F_U(\phi)\subseteq T_{A_{L,h}}(\phi(0)).
\end{equation}
Let us consider two cases.
\begin{enumerate}[leftmargin=*]
\item Given the set $D= \{(s,i)\in T\;\vert\;s\leq \frac{\gamma}{\beta}\}$, we note that  $A\cap D=A_{L,h}\cap D$. Hence, in $D$ the analysis performed for $A$ in Theorem~\ref{thm:ViabilityDelay} applies also to $A_{L,h}$. Therefore,  
 the inclusion \eqref{FUTB} holds true in the case $\phi(0)\in \partial A_{L,h}\cap D$.
\item In the remaining case,  in which 
$\phi(0)=(s_0,i_0)$ with  $s_0\in (\frac{\gamma}{\beta},\widehat s_{\beta,L,h}]$ and $i_0=\Gamma_{\beta,L,h}(s_0)$,
 we denote the unique normal vector (up to positive scalar multiplication) to $A_{L,h}$ at $(s_0,i_0)$ by $v=(v_1,v_2)\in \R^2\setminus\{0\}$. It is,  by definition, perpendicular to the tangent vector $(1,\Gamma'_{\beta,L,h}(s_0))$. Since $\Gamma'_{\beta,L,h}(s_0)<0$ (by Lemma~\ref{lemma:ConvexityAuxiliaryLocallyLipschitz}), we have  that $v_1\geq 0,v_2\geq 0$, and (by normality)  
\begin{equation}\label{eq:AuxiliaryViabilityLipschitz}
0=v_1\beta s_0 \psi_{L,h}(i_0)-v_2\beta s_0\psi_{L,h}(i_0)+v_2\gamma i_0=-(v_2-v_1)\beta s_0 \psi_{L,h}(i_0)+v_2\gamma i_0,
\end{equation}
also proving that $v_2\geq v_1$. 
We note that, by Assumption~\ref{assum:SEcondAss}, we have $i_\phi(-h)\leq \min\{i_M,\,i_0+Lh\}= \psi_{L,h}(i_0)$. Considering any $b\in [\beta_\star, \beta]$ and using~\eqref{eq:AuxiliaryViabilityLipschitz}, we obtain
\[
\inp{v}{f(\phi,b)} =(v_2-v_1) b s_0i_\phi(-h)-v_2\gamma i_0\leq  (v_2-v_1) \beta s_0\psi_{L,h}(i_0)-v_2\gamma i_0=0,
\]
which proves \eqref{FUTB} also in this case.
\end{enumerate}
This concludes  the proof of  forward invariance of $\cA_{L,h}$.
The proofs of maximality of $A_{L,h}$ and assertion~\emph{(2)} follow by arguments similar to the ones provided in the proof of Theorem~\ref{thm:ViabilityDelay}, and are left to the reader. 
\end{proof}
Differently from the viability analysis performed in Theorem~\ref{thm:ViabilityDelay}, the regions defined in~Theorem~\ref{lemma:ViabilityDelayLipschitz} do depend on the delay parameter $h>0$. We will show that they converge, in a sense we are going to clarify, to the invariance/viability regions for the delay-free case, as $h\to 0$. To this aim, let us summarize here the main viability and invariance results obtained in~\cite[Section 2]{AvrFre22}  for the \emph{delay-free system} of ordinary differential equations
\begin{equation}\label{eq:DelayFreeCase}
\begin{cases}
s'(t)=-b(t)s(t)i(t),\\
i'(t)=b(t)s(t)i(t)-\gamma i(t).
\end{cases}
\end{equation}
% In~\cite[Section 2]{AvrFre22} the main viability results for system~\eqref{eq:DelayFreeCase} are provided, we report them here in a condensed form. 
\begin{theorem}[Theorem 2.3,~\cite{AvrFre22}]\label{lemma:ViabilityDelayFree}
The following propositions hold true for the delay-free system~\eqref{eq:DelayFreeCase}.
\begin{enumerate}%[leftmargin=*]
    \item
The maximal forward invariant set contained in $C$ (the all-control zone $\cA$ in~\cite{AvrFre22}) , is given by
\[
 A_0:=\{x=(s,i)\in C\;\vert\;s\leq \Gamma_{\beta, 0}(s)\},
\]
where 
\begin{equation}\label{eq:GammaBar}
\Gamma_{\beta, 0}(s)=\begin{cases}
i_M,\;\;\;\;&\text{if }0\leq s\leq \frac{\gamma}{\beta},\\
\frac{\gamma}{\beta}+i_M-s+\frac{\gamma}{\beta}\log(\frac{\beta}{\gamma}s),\;\;\;&\text{if }s\geq \frac{\gamma}{\beta}.
\end{cases}
\end{equation}
\item The maximal viable set contained in $C$ (the feasible set $\cB$ in~\cite{AvrFre22}), is given by
\[
 B_0:=\{x=(s,i)\in C\;\vert\;s\leq \Gamma_{\beta_\star, 0}(s)\},
\]
where 
\[
\Gamma_{\beta_\star, 0}(s)=\begin{cases}
i_M,\;\;\;\;&\text{if }0\leq s\leq \frac{\gamma}{\beta_\star},\\
\frac{\gamma}{\beta_\star}+i_M-s+\frac{\gamma}{\beta_\star}\log(\frac{\beta_\star}{\gamma}s),\;\;\;&\text{if } s\geq \frac{\gamma}{\beta_\star}.
\end{cases}
\]
    \end{enumerate}
\end{theorem}
\begin{remark}
We have chosen to denote the maximal sets by $A_0$ and $B_0$ instead than $\cA$ and $\cB$ as in~\cite{AvrFre22} to stress the fact that these sets corresponds to the delay $h=0$.
\end{remark}

\begin{figure}[h!]
\centering
\includegraphics[width=\textwidth]{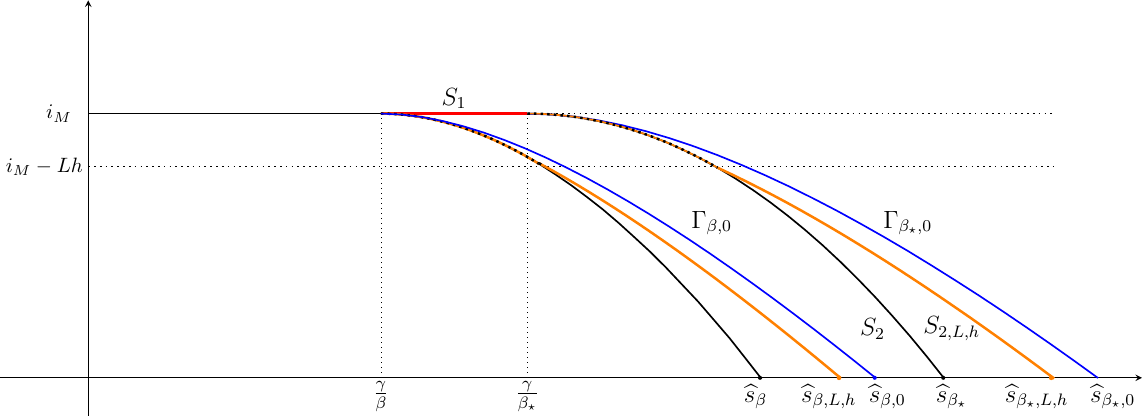}
        \caption{The figure shows  qualitative graphs of the  curves introduced in Theorem~\ref{lemma:ViabilityDelayLipschitz},~Lemma~\ref{lemma:ConvergenceSetsDelayFree} and Proposition~\ref{prop:GreedyLocallyLispchitz}. In orange we plot the graphs of the functions $\Gamma_{\beta,L,h}$ and $\Gamma_{\beta_\star,L,h}$. In black the curves $\Gamma_\beta$ and $\Gamma_{\beta_\star}$ corresponding to the case studied in Subsection~\ref{subsec:Arbitrary}, while in blue the curves $\Gamma_{\beta,0}$ and $\Gamma_{\beta_\star,0}$ corresponding to the (invariant and viable) sets in the delay-free case. }\label{Figure:SecondPlot}
\end{figure}

We now study and characterize the dependence of the sets $A_{L,h}$ and $B_{L,h}$ on the parameter $h>0$, 
and their relations with the sets $A,B$ defined in~Subsection~\ref{subsec:Arbitrary} for an arbitary continuous initial condition, and the sets $ A_0, B_0$ corresponding to the delay-free case. To these aims, we introduce the Cauchy problems
\begin{subequations}
\begin{equation}\label{eq:SystemS0}
(\cS_0):\;\;	\begin{cases}x'(t)=g(x(t)),\\
	x(0)=(\frac{\gamma}{\beta}, i_M),
	\end{cases}
 \end{equation}
 \begin{equation}\label{eq:SystemSLH}
	( \cS_{L,h}):\;\begin{cases}
	x'(t)=g_{L,h}(x(t)),
	\\
		x(0)=(\frac{\gamma}{\beta}, i_M),
\end{cases}
\end{equation}
\end{subequations}
with $x=(s,i)$, $g(x):=(\beta si, \, -\beta si+\gamma i)$ and $g_{L,h}(x):=(\beta s\psi_{L,h}(i), \, -\beta s\psi_{L,h}(i)+\gamma i)$.
We note that the system in 
\eqref{eq:SystemSLH} has been already considered in \eqref{eq:Non-LinearBeta}, and the function $\Gamma_{\beta,L,h}$,  introduced immediately after,   is the graph, in the $(s,i)$-plane, of the solution to  $\cS_{L,h}$.
Similarly, for $s\ge\frac{\gamma}{\beta}$, the function  $\Gamma_{\beta, 0}$ appearing in \eqref{eq:GammaBar} is the graph of the solution to the delay-free problem \eqref{eq:DelayFreeCase} with $b(t)=\beta$, as well as to its time-reversed version~$\cS_0$.

As a preliminary step, we study the asymptotic behaviour of the solution to $\cS_0$.

\begin{lemma}\label{lemma:BehaviourXbarra}
The Cauchy problem $\cS_0$ admits a unique solution 
$x(t)=(s(t),i(t))$, $t\in \R_+$. Moreover, it  satisfies the following properties:
\begin{enumerate}%[leftmargin=*]
\item $s$ and $i$ are strictly positive;
\item $s$ is strictly increasing, and $i$ is non increasing;
\item $\lim_{t\to +\infty} s(t)=:\widehat s_{\beta,0}<+\infty$ and $\lim_{t\to +\infty} i(t)=0$;
\item $\widehat s_{\beta,0}$  is equal to the unique element of $[\frac{\gamma}{\beta}, +\infty)$ such that $\Gamma_{\beta, 0}(\widehat s_{\beta,0})=0$.
\end{enumerate}
\end{lemma}

\begin{proof}
Since $g$ is a smooth function, then it is locally Lipschitz and there exists a unique solution defined on a maximal interval $[0,\tau)$, $\tau>0$. 

\emph{(1)} By integration we have $s(t)=\frac{\gamma}{\beta} e^{\int_0^t\beta i(\xi)\,d\xi}>0$
%for all $t\in \R_+$, 
and $i(t)=i_M e^{\int_0^t(-\beta s(\xi)+\gamma)\,d\xi }>0$,  for all $t\in [0,\tau)$.

\emph{(2)} Since $s'(t)=\beta s(t) i(t)>0$ for all $t\in [0,\tau)$, we have that $s$ is strictly increasing. Similarly, $i'(t)=-\beta s(t)i(t)+\gamma i(t)\leq -\beta \frac{\gamma}{\beta}i(t)+\gamma i(t)\leq 0$, proving that $i$ is non-increasing.

As a consequence, we have
$$
\big|g\big(s(t),i(t)\big)\big|\le \beta i_M|s(t)|+\gamma i_M
$$
for every $t\in[0,\tau)$, which implies that $\tau=\infty$ (otherwise, the solution could be extended on a right neighborhood of $\tau$) and the solution exists on $[0,+\infty)$ as claimed in the first part of the statement. 

\emph{(3)} By summing the equations, integrating in $[0,t]$ and substituting the function $i(t)$ under the integral with its expression obtained by the first equation, we get
\begin{equation}\label{eq:Toolexpression}
i(t)=\frac{\gamma}{\beta} +i_M-s(t)+\frac{\gamma}{\beta} \log(\frac{\beta}{\gamma}s(t))\quad\forall\,t\in\R_+.
\end{equation}
 The limits in \emph{(3)}  
 exist by monotonicity.
If, by contradiction, $\widehat s_{\beta,0}=+\infty$ then the previous equation would imply $\lim_{t\to +\infty} i(t)=-\infty$, leading to a contradiction with \emph{(1)}.
Assuming, by contradiction, that $\lim_{t\to +\infty}i(t)=:\widehat \imath_{\beta,0} >0$, we have that there exists $\bar t>0$ such that 
\[
s'(t)\geq \beta \frac{\gamma}{\beta}\frac{\widehat \imath_{\beta,0}}{2}=\frac{\gamma \widehat \imath_{\beta,0}}{2}>0,\;\;\forall t\geq \bar t,
\]
contradicting $\widehat s_{\beta,0}<+\infty$.

\emph{(4)} By taking the limit as $t\to +\infty$ in~\eqref{eq:Toolexpression}, we get
$
0=\frac{\gamma}{\beta} +i_M-\widehat s_{\beta,0}+\frac{\gamma}{\beta} \log(\frac{\beta}{\gamma}\widehat s_{\beta,0})$,
that is,  $\Gamma_{\beta, 0}(\widehat s_{\beta,0})=0$ (see~\eqref{eq:GammaBar}). The uniqueness simply follows by the strict monotonicity of $\Gamma_{\beta, 0}$ in the interval $[\frac{\gamma}{\beta},\infty)$.
\end{proof}

\begin{remark}
Also the Cauchy problem $\cS_{L,h}$ admits a unique solution 
$x_{L,h}(t)=(s_{L,h}(t),i_{L,h}(t))$, $t\in \R_+$. 
Moreover,  $s_{L,h}$ is strictly positive but (it is easy to prove that) $i_{L,h}$ must take also negative values. 
\end{remark}

We are now able to prove the aforementioned convergence result. 
\begin{theorem}\label{lemma:ConvergenceSetsDelayFree}
Let us consider $h>0$, $h'>0$, $L>0$ and the sets defined in Theorem~\ref{thm:ViabilityDelay}, Theorem~\ref{lemma:ViabilityDelayLipschitz} and Lemma~\ref{lemma:ViabilityDelayFree}. The following propositions hold.
\begin{enumerate}%[leftmargin=*]
\item If $Lh\geq i_M$, then $A=A_{L,h}$ and $B=B_{L,h}$.
\item $A\subseteq A_{L,h}$ and $B\subseteq B_{L,h}$.
\item  If $h\geq h'$, then $A_{L,h}\subseteq A_{L,h'}$ and $B_{L,h}\subseteq B_{L,h'}$.
\item $\overline{\bigcup_{h>0}A_{L,h}}=A_0$,
 and $\overline{\bigcup_{h> 0}B_{L,h}}=B_0$.
\end{enumerate}
\end{theorem}
\begin{proof}
Proposition~\emph{(1)} follows by the fact that, if $Lh\geq i_M$, then $\psi_{L,h}(i)=i_M$ for all $i\geq 0$ and thus the curves defining $A,A_{L,h}$ and $B,B_{L,h}$ coincide, i.e.,  $\Gamma_\beta\equiv\Gamma_{\beta,L,h}$ and $\Gamma_{\beta_\star}\equiv\Gamma_{\beta_\star,L,h}$ in their respective domains.

About~\emph{(2)}, we recall from Subsection~\ref{subsec:Arbitrary} that $\Gamma_\beta$ is the restriction to  $ [\frac{\gamma}{\beta}, \widehat s_{\beta}]$ of the solution to the scalar differential equation $i'(s)=-1+ \frac{\gamma i(s)}{\beta s i_M}$, with initial condition $i(\frac{\gamma}{\beta}) =i_M$, see~\eqref{eq:DefintionGamma}. Similarly,  by dividing the equations in~\eqref{eq:Non-LinearBeta}, $\Gamma_{\beta,L,h}$ turns out to be the solution to
\begin{equation}\label{eq:SolutionInTheMiddle}
i'(s)=-1+ \frac{\gamma i(s)}{\beta s \psi_{L,h}(i(s))}
\end{equation} with the same initial condition. Since $\psi_{L,h}(i)\leq i_M$ for all $i\in \R$, by the Comparison Lemma (see for example~\cite[Lemma 1.2]{Teschl12}) we have
\[
\Gamma_\beta(s)\leq \Gamma_{\beta,L,h}(s) \;\;\forall s\in [\frac{\gamma}{\beta}, \widehat s_{\beta}],
\]
proving that $A\subseteq A_{L,h}$. The proof of the inclusion $B\subseteq B_{L,h}$ follows the same argument. 

Proposition \emph{(3)} follows again by a similar comparison argument, noting that, if $h'\leq h$, then $\psi_{L,h'}(i)\leq \psi_{L,h}(i)$ for all $i\in \R$.

To prove~\emph{(4)},  we first show that $\overline{\bigcup_{h>0}A_{L,h}}\subseteq A_0$. Let us start by recalling that $\Gamma_{\beta, 0}(s)$, for  $s\in[\frac{\gamma}{\beta}, \widehat s_{\beta, 0}]$,  is
the graph of the solution to the delay-free problem \eqref{eq:DelayFreeCase} with $b\equiv\beta$. Hence,
it is the solution in $[\frac{\gamma}{\beta}, \widehat s_{\beta, 0}]$ to the scalar differential equation
\[
i'(s)=-1+\frac{\gamma}{\beta s}
\]
with initial condition $i(\frac{\gamma}{\beta})=i_M$. On the other hand, we have already seen that  $\Gamma_{\beta, L,h}$, in $[\frac{\gamma}{\beta}, \widehat s_{\beta,L,h}]$, is the solution to~\eqref{eq:SolutionInTheMiddle}
with the same initial condition $i(\frac{\gamma}{\beta})=i_M$. Since $\frac{i}{\psi_{L,h}(i)}\leq 1$ for all $i\in[0, i_M]$ and all $h\geq 0$, we have, again by a comparison argument, that $\Gamma_{\beta, L,h}(s)\leq\Gamma_{\beta,0}(s)$ for all $h>0$ and for all $s\in [\frac{\gamma}{\beta},\widehat s_{\beta,L,h}]$. This in particular yields $A_{L,h}\subset A_0$ for all $h>0$, implying $\overline{\bigcup_{h>0}A_{L,h}}\subseteq A_0$.

Since $A_0=\overline{\text{int}(A_0)}$,  to prove the converse inclusion it is enough to show that \begin{equation}\label{intA0s}
\text{int}(A_0)\subseteq \overline{\bigcup_{h>0}A_{L,h}}.
\end{equation}
To this aim, it is useful to recall that, as $s\ge\frac{\gamma}{\beta}$,  the curves $\Gamma_{\beta, 0}$ and $\Gamma_{\beta,L,h}$, defining the (boundaries of) the sets $A_0$ and $A_{L,h}$, are the graphs, in the $(s,i)$-plane, of the solutions to the problems $(\cS_0)$ and $(\cS_{L,h})$, respectively.
Thus, as a preliminary step, we prove that the solutions to~$(\cS_{L,h})$ defined in~\eqref{eq:SystemSLH} pointwise converge, as $h \to 0$, to the solution to~$(\cS_0)$ defined in~\eqref{eq:SystemS0}, by verifying the hypotheses of Lemma~\ref{lemma:AppendixGronwall}. Let us denote by $\bar x=(\bar s,\bar \imath):\R_+\to \R^2$ the solution to $(\cS_0)$  and by $x_{L,h}=(s_{L,h},i_{L,h}):\R_+\to \R^2$ the solution to $(\cS_{L,h})$.
It can be proved that the vector fields $g_{L,h}:\R^2\to \R^2$ are uniformly locally Lipschitz, i.e., for any compact set $K\subset \R^n$ there exists a $M_K\geq 0$ such that
\begin{equation}\label{eq:EquicontinuityGlh}
	|g_{L,h}(x_1)-g_{L,h}(x_2)|_{\max}\leq M_K|x_1-x_2|_{\max},\;\;\forall h>0,\;\forall x_1,x_2\in K,
\end{equation}
by using the inequality $|\psi_{L,h}(i_2)-\psi_{L,h}(i_1)|\leq \min\{2i_M, |i_2-i_1|\}$  for all $i_1,i_2\in \R$ and all $h>0$.
% Indeed, fixing $x_1=(s_1,i_1)\in\R^2$ and $x_2=(s_2,i_2)\in \R^2$, we compute
% \[
% \begin{aligned}
% \|g_{L,h}(x_2)-g_{L,h}(x_1)\|_\infty &=\max\left \{\beta|s_2\psi_{L,h}(i_2)-s_1\psi_{L,h}(i_1)|,\,\beta  | s_2\psi_{L,h}(i_2)-\beta s_1\psi_{L,h}(i_1)|+\gamma|i_2-i_1|\right\}\\&\leq 2\beta i_M|s_2-s_1|+\gamma|i_2-i_1|\leq (2\beta i_M+\gamma) \|x_2-x_1\|_\infty,	
% \end{aligned}
% \]
% where we have 
%We have already proved that, for both problems $(\cS_0)$ and $(\cS_{L,h})$, the $s$-component of the solution is strictly increasing. 
By Lemma~\ref{lemma:BehaviourXbarra}, there exists a compact set $Q\subset [\frac{\gamma}{\beta}, +\infty)\times [0,i_M]$ such that that $\bar x(t)\in Q$, for all $t\in \R_+$. 
We also note that $\psi_{L,h}$ converges uniformly to the identity $\text{Id}(i)=i$ in $[0,i_M]$ as $h$ goes to $0$. 
Thus, $g_{L,h}$ is converging uniformly to $g$ in $Q$. Recalling the sublinear bound in~\eqref{Eq:subLinearBoundSIR}, all the hypotheses of Lemma~\ref{lemma:AppendixGronwall} are verified. We can thus conclude that $\displaystyle \lim_{h\to 0}x_{L,h}(t)=\bar x(t)$, for all $t\in \R_+$.

Let us now consider any point $z=(s_z,i_z)\in \text{int}(A_0)$. If  $z\in A$, then \eqref{intA0s}  trivially follows by point \emph{(2)}. Let us thus suppose that $z\in \text{int}(A_0)\setminus A$, i.e.,  $\frac{\gamma}{\beta}<s_z<\widehat s_{\beta,0}$ and $\,\max\{0,\Gamma_\beta(s_z)\}<i_z<\Gamma_{\beta, 0}(s_z)$.
By Lemma~\ref{lemma:BehaviourXbarra}, since $s_z<\widehat s_{\beta,0}$ there exists $T_z>0$ such that $\bar s(T_z)>s_z$. By Lemma~\ref{lemma:BehaviourXbarra} we also have $\bar \imath(T_z)>0$. By pointwise convergence of $x_{L,h}$ to $\bar x$, there exists $\overline h>0$ small enough such that $s_{L,h}(T_z)>s_z$ for all $h\leq \overline h$. Thus, by continuity of $s_{L,h}$,  for any $h\leq \overline h$ we can choose $\tau_h\in(0,T_z)$  such that $s_{L,h}(\tau_h)=s_z$.
In $(0, \overline h)$, the function 
$
h\mapsto \tau_h
$
is bounded from below by $0$ and bounded from above by $T_z$, and  thus there exists a sequence $h_n\to 0^+$ such that $\lim\limits_{n\to +\infty}\tau_{h_n}=\bar \tau$ for some $\bar \tau \in [0,T_z]$.
By Lemma~\ref{lemma:AppendixGronwall} we have that, up to a subsequence,  $x_{L,h_n}\to \overline x$ uniformly in $[0,T_z]$, as $n\to +\infty$. We thus have:
\[
\lim_{n \to \infty}\Gamma_{\beta, L,h_n}(s_z)=
\lim_{n \to \infty}\Gamma_{\beta, L,h_n}(s_{ L,h_n}(\tau_{h_n}))=
\lim_{n\to +\infty}i_{ L,h_n}(\tau_{h_n})= \wi(\bar \tau)=\Gamma_{\beta,0}(s_z).%{\color{red}\in \partial A_0 \mbox{ NO}} 
\]

Since $i_z< \Gamma_{\beta,0}(s_z)$ this implies that there exists  $h_z>0$ such that 
\[
i_z\leq \Gamma_{\beta, L,h_z}(s_z)< \Gamma_{\beta,0}(s_z),
\]
implying that $z\in A_{L,h_z}\subset \overline{\bigcup_{h>0}A_{L,h}}$.
We have thus proved that $\text{int}(A_0)\subseteq \overline{\bigcup_{h>0}A_{L,h}}$, as claimed.

The argument for $B_{L,h}$ and $B_0$ is similar.
 \end{proof}

In Figure~\ref{Figure:SecondPlot} the results of Lemma~\ref{lemma:ConvergenceSetsDelayFree} can be visualized in a particular case.

As in the case studied in Subsection~\ref{subsec:Arbitrary}, Theorem~\ref{lemma:ViabilityDelayLipschitz} provides a tool in designing a feedback control policy  (the so-called \emph{greedy control strategy}).
To be concise, we develop here only the technical details and send back 
the unaccustomed reader to  Subsection~\ref{subsec:Arbitrary} for an introduction to the greedy control strategy.

Given the set of functions
\[
\cB_{L,h}:=\big\{\phi\in \cC_{L,|\cdot|_{\max}}\;\vert\;\phi(0)\in B_{L,h},\;\;\phi(s)\in C\;\,\forall \,s\in [-h,0]\big\},
\] we  consider, for every $\phi\in\cB_{L,h}$, the {greedy control strategy} 
%$\bt(\phi)\in U_{\cB_{L,h}}(\phi):=\{b\in [\beta_\star, \beta]\;\vert\;f(\phi,b)\in D_{\cB_{L,h}}(\phi)\} $ obtained by minimizing the running cost $G$ in~\eqref{eq:CostFunctionsal} among all controls that keep the solution inside the viable set, i.e.,
\begin{equation}\label{eq:GreedyAsMinimizing1}
\bt_{L,h}(\phi):=\arg\min_{b\in U_{\cB_{L,h}}(\phi)}G(\beta-b)=\max U_{\cB_{L,h}}(\phi),
\end{equation}
where 
$$
U_{\cB_{L,h}}(\phi):=\{b\in [\beta_\star, \beta]\;\vert\;f(\phi,b)\in D_{\cB_{L,h}}(\phi)\}. 
$$
By computing the maximum in~\eqref{eq:GreedyAsMinimizing1}, as done in Lemma~\ref{lemma:EquivalentDefinitionControl} for the continuous initial conditions case, one obtain the following explicit expression 
\begin{equation}\label{eq:ControlPolicy2}
\bt_{L,h}(\phi)=\begin{cases}
\beta,\;\;\;&\text{if }\phi(0)\in B_{L,h}\setminus (S_1\cup S_{2,L,h}),\\
\min\big \{\beta,\frac{\gamma i_M}{s_\phi(0)i_\phi(-h)}\big \},\;\;\;&\text{if }\phi(0)\in S_1,\\
\min \big\{\beta,\beta_\star\frac{\psi_{L,h}(i_\phi(0))}{i_\phi(-h)} \big\},\;\;\;&\text{if }\phi(0)\in S_{2,L,h},
\end{cases} 
\end{equation}
with the convention $1/0=+\infty$. In the expression above, the set $S_1$
is the same defined in Lemma \ref{lemma:EquivalentDefinitionControl}, while
\[
\begin{aligned}
S_{2,L,h}:=\big\{(s,i)\in T\;\;\vert\;\;s\in [\frac{\gamma}{\beta_\star}, \widehat s_{\beta_\star,L,h}] \text{ and } i =\Gamma_{\beta_\star,L,h}(s)\big\},
\end{aligned}
\]
where   $\Gamma_{\beta_\star,L,h}:[\frac{\gamma}{\beta_\star},\widehat s_{\beta_\star,L,h}]\to [0,i_M]$ represents the solution of \eqref{eq:Non-LinearBetaStar} in the plane $(s,i)$.

The properties of the greedy control policy and  the resulting controlled (\emph{closed-loop}) solutions are summarized in the following statement.

\begin{theorem}[greedy  control policy for Lipschitz continuous initial conditions]\label{prop:GreedyLocallyLispchitz} 
Given the function $F_{L,h}:\cB_{L,h}\to \R^2$ defined  by
$F_{L,h}(\phi)=f(\phi, \bt_{L,h}(\phi))$ (see~\eqref{eq:DefnDelSistema}), there exists a solution $x_\phi:[-h,+\infty)\to \R^2$ to the Cauchy problem
\begin{equation}\label{eq:ClosedLoopLipschitz}
\begin{cases}
x'(t)=F_{L,h}(S(t)x),\\
S(0)x=\phi\in \cB_{L,h}.
\end{cases}
\end{equation} 
Given any $\phi\in \cB_{L,h}$, we have that
\begin{enumerate}
\item $x_\phi(t)\in B_{L,h}$ for all $t\in \R_+$. 
\item   Defining $\bt_{L,h,\phi}:\R_+\to [\beta_\star,\beta]$ by
\begin{equation}\label{eq:ControlPolicyOpenLoopLipschitz}
\bt_{L,h,\phi}(t):=\bt_{L,h}(S(t)x_\phi),
\end{equation}
it holds that 
\begin{enumerate}
\item $x_\phi$ is the unique solution to the Cauchy problem  \eqref{eq:System}-\eqref{icphi} with $b=\bt_{L,h,\phi}$;
\item 
$\bt_{L,h,\phi}$ is eventually constant equal to $\beta$;
\item  the corresponding cost~\eqref{eq:CostFunctionsal} is finite.
\end{enumerate}
\end{enumerate}
Moreover, if $\phi\in \cB_{L,h}\cap \cT_0$ (i.e. avoiding the trivial case of $i_\phi$ identically zero in $[-h,0]$), then
\begin{enumerate}%[leftmargin=*]
\setcounter{enumi}{2}
\item there exists a $T=T(\phi)\geq0$ such that $x_\phi(t)\in R$ for all $t\geq T$.
\end{enumerate}
\end{theorem} 
\begin{remark}\label{gcsc1}
The control policy $\bt_{L,h,\phi}:\R_+\to [\beta_\star,\beta]$ in~\eqref{eq:ControlPolicyOpenLoopLipschitz} can be rewritten as follows:
\begin{equation}\label{eq:ControlPolicyExplicitFeedback1}
\bt_{L,h,\phi}(t)=\begin{cases}
\beta,\;\;\;&\text{if }(s_\phi(t), i_\phi(t))\in B_{L,h}\setminus (S_1\cup S_{2,L,h}),\\
\min\big\{\beta,\frac{\gamma i_M}{s_\phi(t)i_\phi(t-h)}\big \},\;\;\;&\text{if }(s_\phi(t), i_\phi(t))\in S_1,\\
\min \big\{\beta,\;\beta_\star\frac{\psi_{L,h}(i_\phi(t))}{i_\phi(t-h)}\big\},\;\;\;&\text{if }(s_\phi(t), i_\phi(t))\in S_2.
\end{cases} 
\end{equation} 
\end{remark}

\begin{proof}  
The proof relies on arguments similar to the ones of~Theorem~\ref{thm:GreedyControlPolicy}, using the viability analysis performed in Theorem~\ref{lemma:ViabilityDelayLipschitz}. 
A preliminary step to the proof of existence of a solution to the Cauchy problem \eqref{eq:ClosedLoopLipschitz}, consists in noting  that 
\begin{equation}\label{eq:EqualityLemmaForInvFLh}
 F_{L,h}(\phi)=f(\phi,\bt_{L,h}(\phi))
 \in \cD_{\cB_{L,h}}(\phi)=T_{B_{L,h}}(\phi(0))\cap B_{|\cdot|_{\max}}(0,L)\;\;\;\;\forall\,\phi \in \cB_{L,h}.
 \end{equation}
 This is a direct consequence of the fact that $\bt_{L,h}(\phi)\in U_{\cB_{L,h}}(\phi)$ for every $\phi\in \cB_{L,h}$ which implies~\eqref{eq:EqualityLemmaForInvFLh} by definition of $U_{\cB_{L,h}}(\phi)$.

Under this condition, the  existence of solutions $x_\phi$ to~\eqref{eq:ClosedLoopLipschitz} is proved as  in Lemma~\ref{lemma:WellPosednessSoltuions} in Appendix, together with the property  $S(t)x_\phi\in \cB_{L,h}$, for all $t\in \R_+$, which implies $S(t)x_\phi(0)=x_\phi(t)\in B_{L,h}$, for all $t\in \R_+$. 
 Then, part \emph{(1)} of the statement is proved.

The proof of the remaining part of the statement proceeds exactly as done in the proof of Theorem~\ref{thm:GreedyControlPolicy}, just under  the necessary notational adaptation.  
\end{proof}

A direct consequence of Proposition~\ref{prop:GreedyLocallyLispchitz} in terms of the optimal control problem $\cP_\phi$ defined in Problem~\ref{prob:OptimalControlPb} is stated below.

\begin{corollary}
Given any $h>0$ and any $L>0$, the optimal control problem $\cP_\phi$  admits a solution with a finite cost, for any initial condition $\phi\in \cB_{L,h}$.
\end{corollary}

\begin{figure}[h!]%[t!]%[H]
\begin{center}
\includegraphics[width=0.49\textwidth]{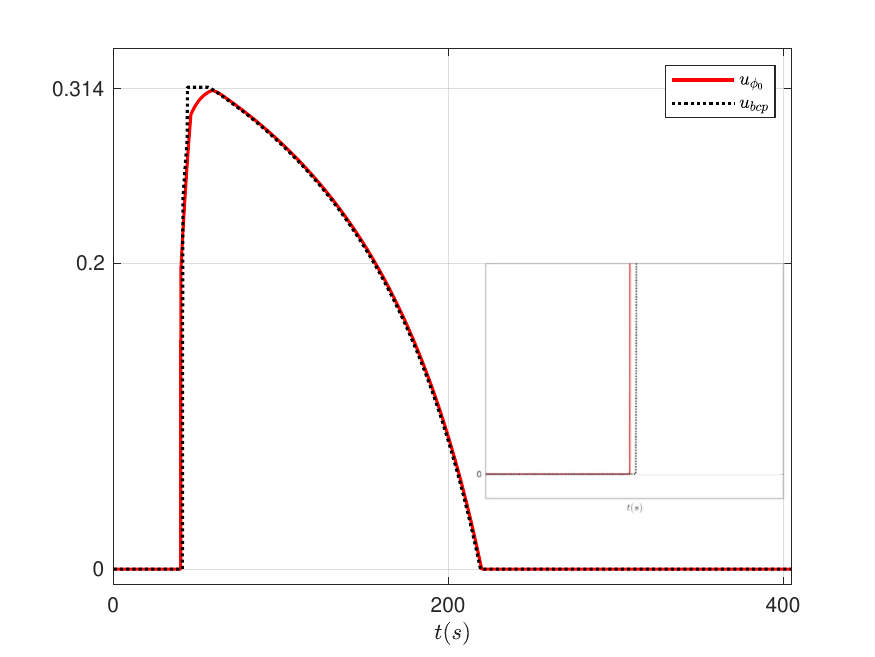}
\includegraphics[width=0.49\textwidth]{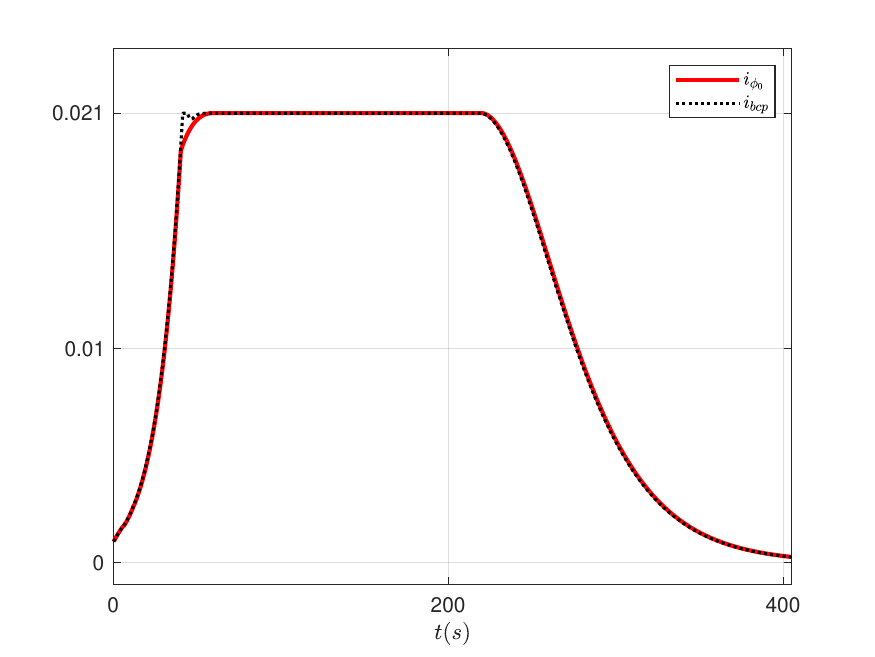}
\includegraphics[width=0.49\textwidth]{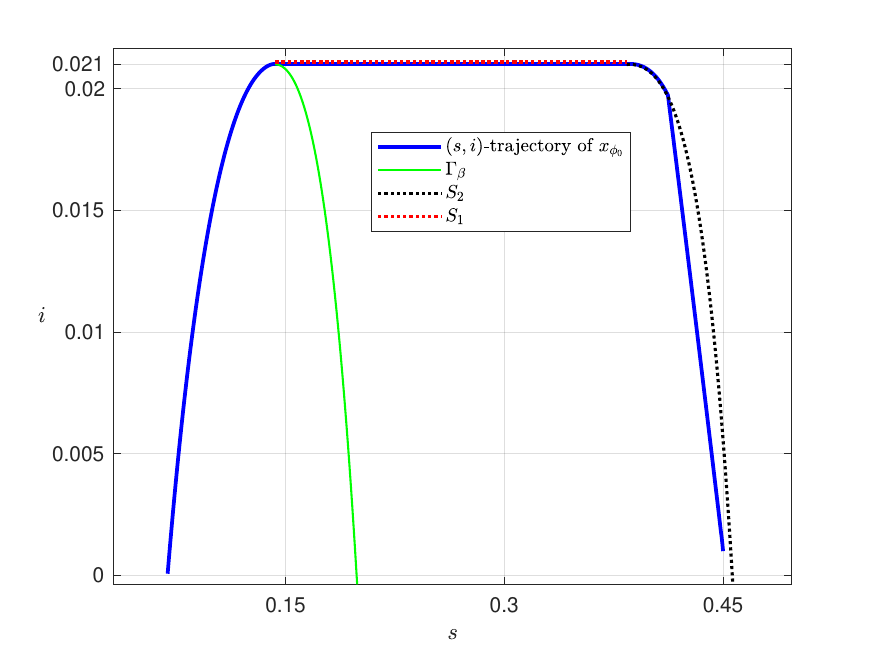}
\includegraphics[width=0.49\textwidth]{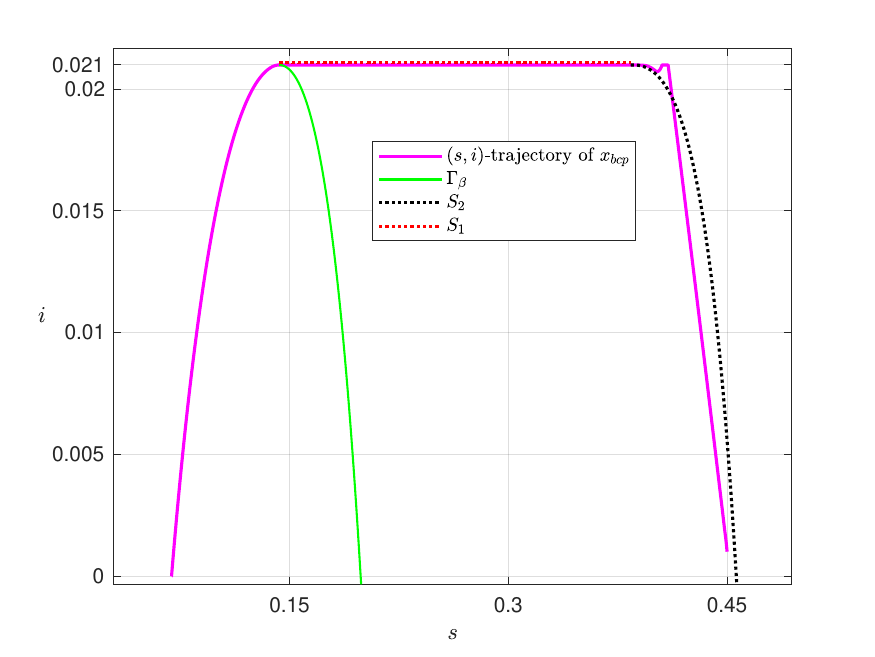}
% \end{center}
% \caption{Viability zones for the Delta variant (autumn 2021) in Italy}
% \label{i_opt}
% \end{figure}

% \begin{figure}[h]
% \begin{center}

\end{center}
\caption{Solutions for the constant initial condition $\phi_0(t)\equiv x_0=(0.45, 0.001)$. The dotted black lines in the pictures above represent  the optimal control $u_{{bcp}}(t)$ and the value of $i_{{bcp}}(t)$,  respectively, as computed by \textsc{Bocop}. The red lines represent the greedy control $u_{\phi_0}(t)=\beta-b_{\phi_0}(t)$ described in Theorem~\ref{thm:GreedyControlPolicy}, and the $i$-component of the solution, respectively, computed with \textsc{Matlab}.
In the pictures below, on the left, we plotted the trajectory of the solution $x_{\phi_0}$, corresponding to the greedy control $u_{\phi_0}$, with the curves $S_1$ and $S_2$ in Theorem~\ref{thm:GreedyControlPolicy} plotted in dotted red and black lines, respectively, together with the curve $\Gamma_\beta$, representing the boundary of the set $A$ defined in~Theorem~\ref{thm:ViabilityDelay}, plotted in green. On the right, the same plot for the solution $x_{{bcp}}$ corresponding to the control $u_{{bcp}}$.}
\label{i_greed}
\end{figure}

\section{Examples and numerical simulations}\label{Sec:Simul}
 %\begin{figure}[h]
%\begin{center}
%\includegraphics[width=1.0\textwidth]{ita_covid_2021_noic2_clean_s0=045}
%\end{center}
%\caption{Viability zones for the Delta variant (autumn 2021) in Italy}
%\label{ita_aut}
%\end{figure}

In the sequel, with the aid of numerical examples, we illustrate the greedy control schemes introduced in the previous section (Theorem~\ref{thm:GreedyControlPolicy} and Theorem~\ref{prop:GreedyLocallyLispchitz}).

In all  examples, the epidemic parameters are chosen  according to  the state of knowledge of  the COVID-19 epidemic in Italy in autumn 2021 (see \cite{FG2023}, Example 4.3). Consequently,  we specialize~\eqref{eq:System} by choosing
\[
\gamma=0.0714,\;\; \beta=0.5,\;\;\beta_\star=0.185,\;\; i_M=0.021. 
\] The reader interested in understanding how such  parameters have been chosen is refereed to \cite{FG2023}, Example~4.3. An appropriate latency period would be between $4$ and $6$ (\cite{Lauer_etal_2020,TaoMa2022,Wu_etal2022}). We consider here a delay $h=6$, which is large enough to appreciate the effect of the delay in numerical computations. 

\begin{example}[constant initial condition]\label{example:Numerical}\em 

We start by considering a constant initial condition $\phi_0(t)\equiv x_0:=(s_0,i_0):=(0.45,0.001)$ for all $t\in [-h,0]$. This means that,  at time $0$, we have a small fraction of infected population, and almost half of the remaining individuals are susceptible. 
Moreover, this constant initial condition  case models an isolated-in-time spreading event at time $-h$;  then, nothing happens ($s$ and $i$ stay constant) till time $0$ after which the infection can be transmitted.

Since $h\max\{\beta,\gamma\}\geq 1$, for every $L>0$ satisfying~\eqref{eq:BoundLipschitzConstant} we have $Lh\geq i_M$. 
By Theorem~\ref{lemma:ConvergenceSetsDelayFree} \emph{(1)}, this implies that the control strategy  described in Subsection~\ref{Subsec:locallyLispchitz} is equivalent to the one introduced in~Subsection~\ref{subsec:Arbitrary}. Roughly speaking, the delay parameter $h=6$ is large enough to prevent any gain in imposing the Lipschitz condition in Assumption~\ref{assum:SEcondAss}. We thus consider the ``greedy'' control policy $\bt_{\phi_0}$ described in Theorem~\ref{thm:GreedyControlPolicy}.
This is possible, since it can be verified that
$\phi_0$ belongs to viable set $\cB$ (see Theorem~\ref{thm:ViabilityDelay}).

The numerical computation of the greedy control $u_{\phi_0}(t):=\beta-\bt_{\phi_0}(t)$ and the corresponding solution 
$x_{\phi_0}=(s_{\phi_0},i_{\phi_0})$
are  performed in \textsc{Matlab} and 
plotted in red in Figure~\ref{i_greed}.
In the same figure, we compare the performance of this control strategy with the ``optimal" one numerically obtained by the algorithmic optimal control toolbox \textsc{Bocop} (\cite{Bocop}), and denoted by $u_{bcp}$. 
In the most recent versions,   \textsc{Bocop}  handles certain classes of delay problems, including the one considered here,  for a constant initial condition. As a template for the numerical simulation we used Example 9.1 in \cite{BocopExamples2019}.
We have chosen  the Gauss's II method and a discretization detail of 10 time steps per day. It can be observed that  the two control actions appear to be close  one to the other. On the other hand, it can be noted that $u_{\phi_0}$ starts its action earlier with respect to $u_{bcp}$, as illustrated in zoomed detail inside the small bottom-right box of Figure~\ref{i_greed}. Moreover, differently from the optimal one $u_{bcp}$, the control strategy 
$u_{\phi_0}$ does not provide 
any time interval at the maximum regime $\beta_\star$ (corresponding to a lockdown policy), according to what already expected in Remark ~\ref{gcsc}~{\it(2)}.

The suboptimality of  $u_{\phi_0}$ is also highlighted by the numerical computation of the costs. Considering for simplicity $G\equiv \text{Id}$ in~\eqref{eq:CostFunctionsal}, it  reads
\begin{equation}\label{eq:CostExample}
J(u_{\phi_0})=\int_0^{+\infty} u_{\phi_0}(t)\,dt=\int_0^{+\infty} \beta-b_{\phi_0}(t)\,dt
,\end{equation}
and \textsc{Matlab} returns the value  $J(u_{\phi_0})=38.766$,  while \textsc{Bocop} gives  $J(u_{{bcp}})=38.5263$. 

It is worth noting also that, while \textsc{Bocop} provides an approximated optimal solution in some minutes, the computation time of the suboptimal greedy strategy with   
\textsc{Matlab} is of the order of a couple of seconds. 
% The optimal cost with respect to the control $u_{\text{bcp}}$ (as approximated by~\textsc{Bocop}) has  value 
% \begin{itemize}
%     \item $J(u_{\text{bcp}})=38.5215$ in around $2250$ seconds with Lobatto's IIIC method,
%     \item $J(u_{\text{bcp}})=38.5263$ in around $350$ seconds with Gauss's II method. {\color{red}Forse potremmo lasciare solo questo}
% \end{itemize}

\begin{figure}[h!]%[t!]%[H]
\begin{center}
\includegraphics[width=0.49\textwidth]{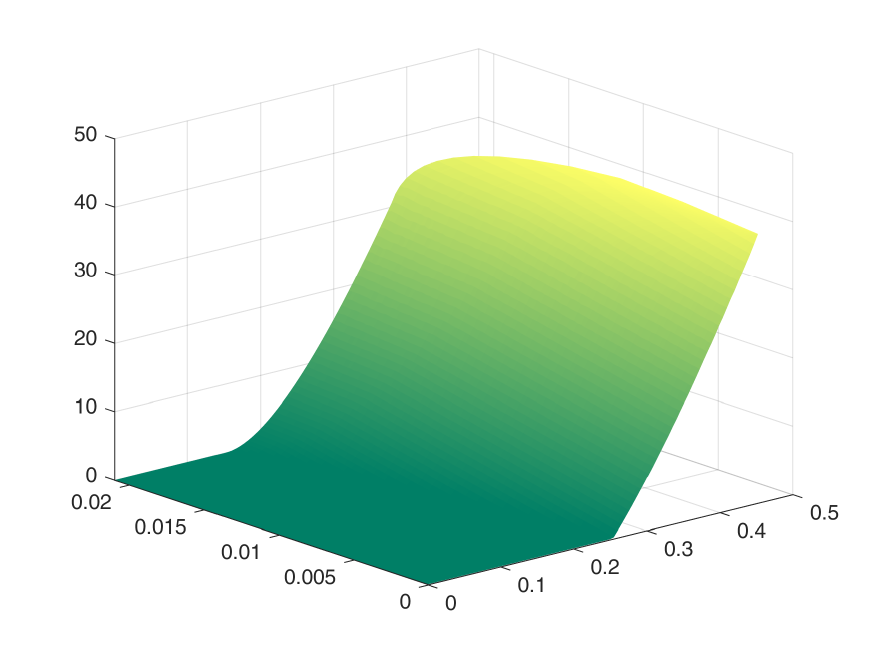}
\includegraphics[width=0.49\textwidth]{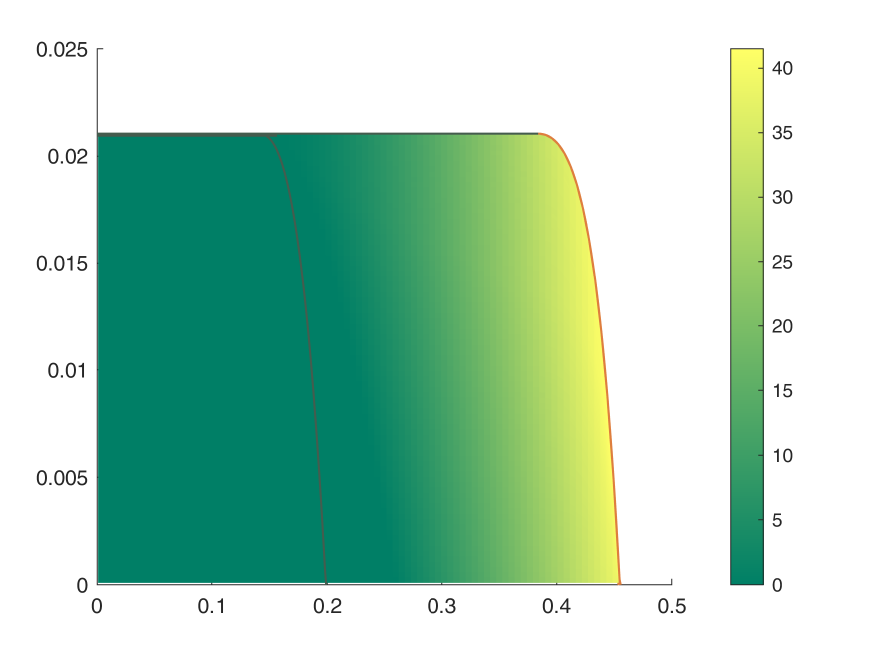}
\end{center}
\caption{
The value of the  cost $J$ for different constant initial conditions in $B$. 
On the right, the viable zone $B$ of the $(s,i)$-plane is represented  with different colors for different values of $J$.
The color scale goes from dark green, corresponding to a null cost, to light yellow corresponding to the highest values. The curve in black, crossing the set $B$, represents the relative boundary ($\Gamma_\beta$) of the set $A$. 
On the left, the 3D graph of the map $x\mapsto J(x)$.}
\label{Fig:ObjectivePlot}
\end{figure}

In order to further analyze the effectiveness and performance of the greedy control strategy, 
% introduced in Theorem~\ref{thm:GreedyControlPolicy}, 
in Figure~\ref{Fig:ObjectivePlot} we plot the function $x\mapsto \bar J(x):= J(u_{\phi_{x}})$ by considering constant initial conditions $\phi_{x}(t)\equiv x$ in $[-h,0]$, for $x\in B$,  and where $u_{\phi_x}=\beta-b_{\phi_x}$ with $b_{\phi_x}$ defined in~\eqref{eq:ControlPolicyOpenLoop}. 
% The resulting plot, varying the point $x$ in $B$, is depicted in Figure~\ref{Fig:ObjectivePlot}. 
As expected, the value of $\bar J(x)$ is $0$ when $x\in A$, while $J$  increases in value as the initial condition approaches the set $S_1\cup S_2\subset \partial B$ (using the notation in Theorem~\ref{thm:ViabilityDelay}, see also Figure~\ref{Figure:FirstPlot}).
% , in which our control policy is ``active''. 
\end{example}

We then apply the greedy control policy described in Theorem~\ref{thm:GreedyControlPolicy} to different initial conditions (reaching the same point $(s_0,i_0)=(0.45,0.001)\in T$ at time $0$). This allows us to evaluate the dependence of the control, the corresponding solution and the corresponding cost, on the past/delayed behaviour of the initial condition.

\begin{example}[exponentially decaying initial condition]\em 
We now consider the situation in which at time $t=-h$ there is an initial proportion of exposed population that is not infectious, but it is recovering with rate equal to $\gamma$. We thus consider $\phi_1:[-h,0]\to \R^2$ defined by
\begin{equation}\label{eq:ExpFisIntialCond}
\phi_1(t)=(s_0, i_0e^{-\gamma t}),\quad t\in [-h,0].
\end{equation}
The constancy of the $s$-component of the initial condition can be assumed without loss of generality, because the dynamics in system~\eqref{eq:System} does not depend on the delayed value of $s$.
It can be numerically verified that $\phi_1\in \cB$, and thus the control introduced in Theorem~\ref{thm:GreedyControlPolicy} is feasible also for this initial condition.
The control $u_{\phi_1}$ and the corresponding $i$-component of the solution, $i_{\phi_1}$,  are depicted in Figure~\ref{Fig:Phi1}.
\begin{figure}[h!]%[t!]%[H]
\begin{center}
\includegraphics[width=0.49\textwidth]{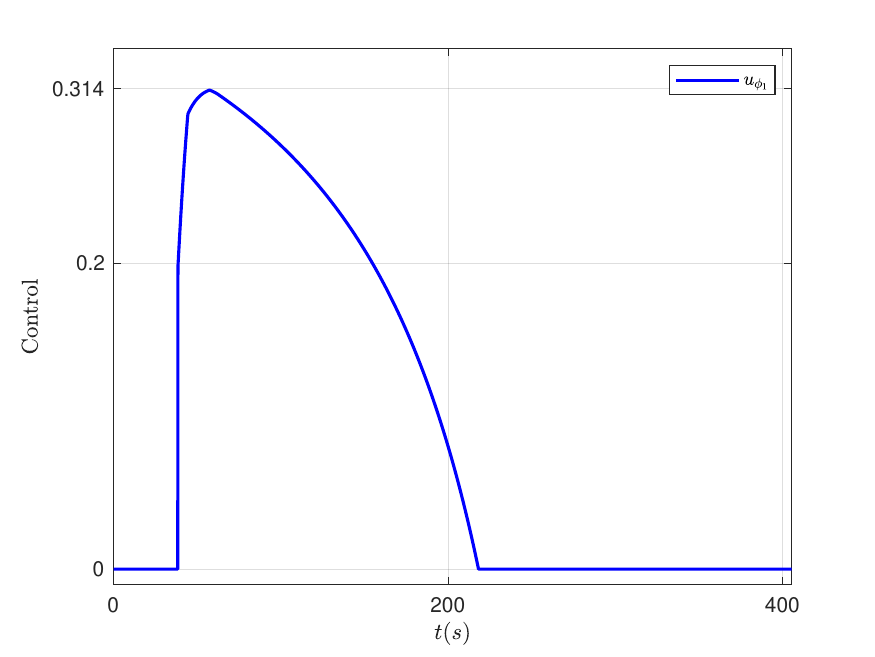}
\includegraphics[width=0.49\textwidth]{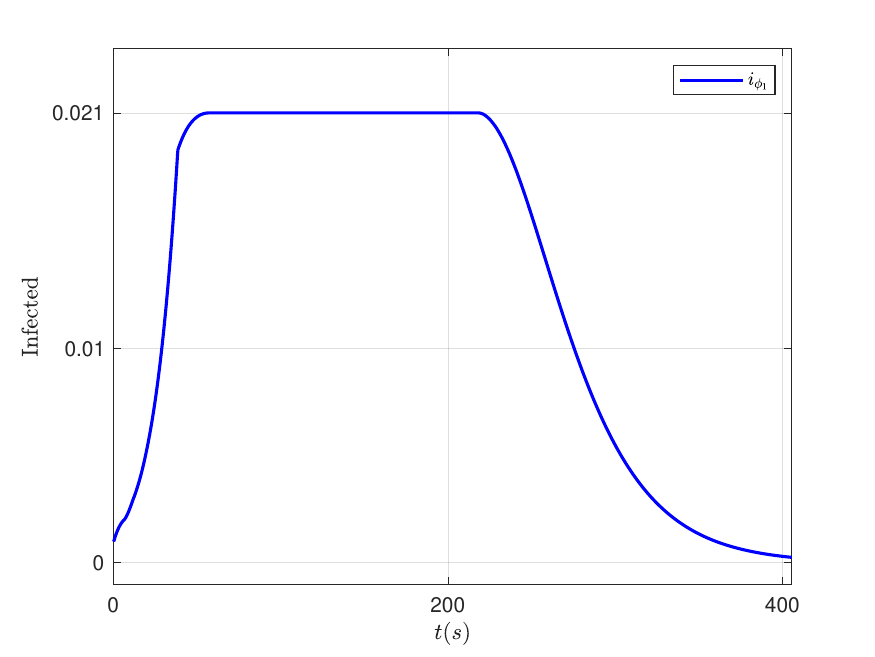}
\end{center}
\caption{The control $u_{\phi_1}$  obtained, via Theorem~\ref{thm:GreedyControlPolicy}, for the initial condition $\phi_1$ defined in~\eqref{eq:ExpFisIntialCond}, and the corresponding infected population $i_{\phi_1}$. }
\label{Fig:Phi1}
\end{figure}
Compared to the constant initial condition case depicted in Figure~\ref{i_greed}, we note that we have a slightly ``perturbed behaviour'' of  the $i$-component in the first time steps, due to the  greater value 
of the initial condition in $[-h,0]$. 
Besides this slight discrepance in the very short term,  the proposed feedback control $u_{\phi_1}$ and the corresponding behaviour of the solution are qualitatively equivalent to the ones obtained with the constant initial condition $\phi_0$.
Also the cost  $J(\phi_1)=38.799$ is again close to the cost $J(\phi_0)$ corresponding to the constant initial condition.
\end{example}

\begin{example}\em 
We finally consider a third initial condition motivated by mathematical interest. We take $\phi_2:[-h,0]\to \R^2$ defined by 
\begin{equation}\label{eq:ExpIntialCond}
\phi_2(t)=(s_0, i_0+\frac{i_M-i_0}{1-e^{-5h}}(1-e^{5t})),\quad t\in [-h,0],
\end{equation}
i.e., we consider an exponentially decreasing $i$-component, starting at $i=i_M$ at $t=-h$ and reaching $i=i_0$ at $t=0$.
The function $i_{\phi_2}:[-h,0]\to \R$ is illustrated in Figure~\ref{Fig:InitialExponential}, and it models an initial condition close to the ``worst case'' used in the proof of Theorem~\ref{thm:GreedyControlPolicy}: it remains close to the maximal admissible $i_M$ before rapidly decreasing to the initial value $i_0$. 
\begin{figure}[h!]%[b!]%[H]
\begin{center}
\includegraphics[width=0.45\textwidth]{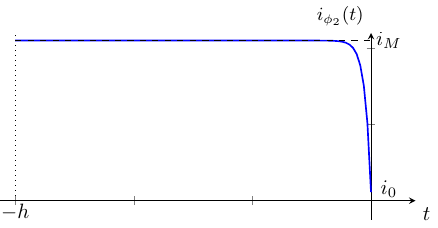}
\end{center}
\caption{The $i$-component of $\phi_2$ defined in~\eqref{eq:ExpIntialCond}, in the interval $[-h,0]$.  }
\label{Fig:InitialExponential}
\end{figure}

The resulting solution and control are depicted in Figure~\ref{Fig:Exponential}.
\begin{figure}[h!]%[t!]%[H]
\begin{center}
\includegraphics[width=0.49\textwidth]{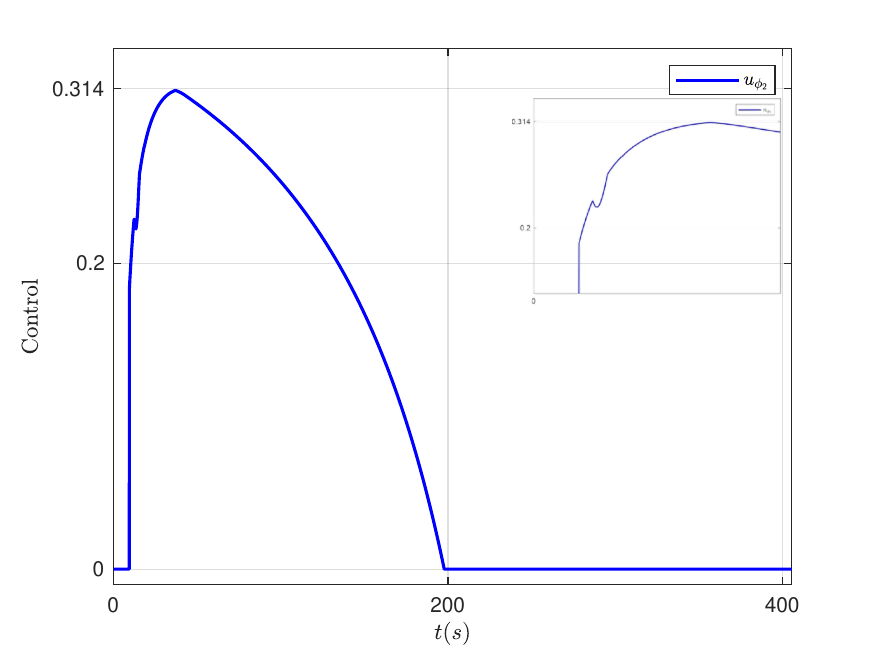}
\includegraphics[width=0.49\textwidth]{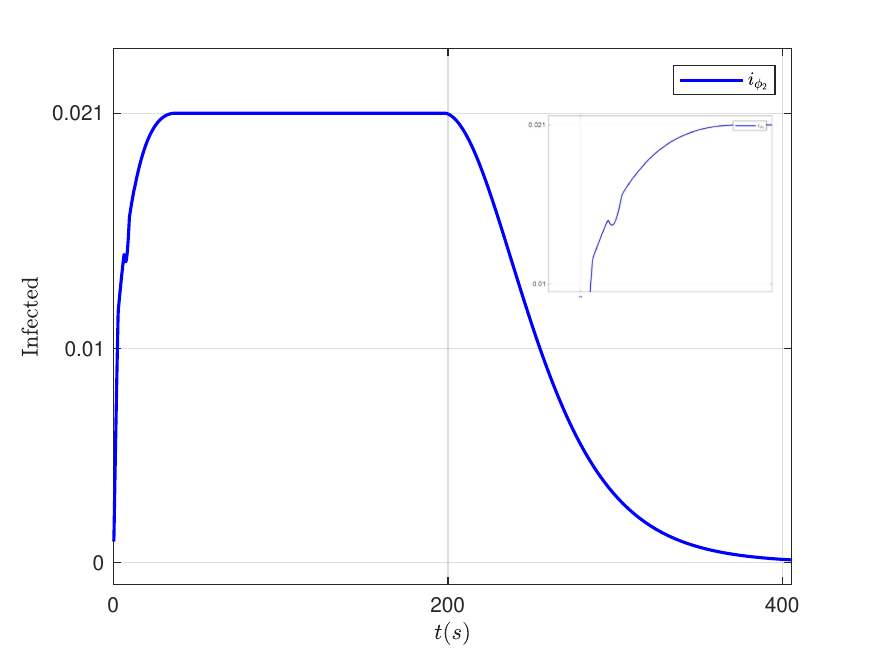}
\end{center}
\caption{The control $u_{\phi_2}$ obtained, via Theorem~\ref{thm:GreedyControlPolicy}, for the initial condition $\phi_2$ defined in~\eqref{eq:ExpIntialCond}, and the corresponding infected population $i_{\phi_2}$. In the upper-right squares, we highlight the non-monotonic behaviour of $u_{\phi_2}$ and  $i_{\phi_2}$ in the 
initial part of the epidemic horizon.}
\label{Fig:Exponential}
\end{figure}
With respect to the previous cases of Figure~\ref{i_greed} and Figure~\ref{Fig:Phi1}, we note that we have an initial non-monotonic evolution of the $i$-component and of the control $u_{\phi_2}$. We also note that the control action starts earlier, since the solution reaches earlier the curve $S_2$ defined in~Theorem~\ref{thm:GreedyControlPolicy} (as can be numerically verified). This discrepancy is also reflected by the value $J(\phi_2)=41.004$, which is  larger than the costs corresponding to the initial conditions $\phi_0$ and $\phi_1$ previously computed.

% \begin{figure}[t!]%[H]
% \begin{center}
% \includegraphics[width=0.39\textwidth]{exp_i.png}
% \includegraphics[width=0.39\textwidth]{exp_u.png}

% \includegraphics[width=0.39\textwidth]{const_i.png}
% \includegraphics[width=0.39\textwidth]{const_u.png}

% \includegraphics[width=0.39\textwidth]{linear_i.png}
% \includegraphics[width=0.39\textwidth]{linear_u.png}

% \includegraphics[width=0.39\textwidth]{super_exp_i.png}
% \includegraphics[width=0.39\textwidth]{super_exp_u.png}

% \includegraphics[width=0.39\textwidth]{super_exp_d=2_i.png}
% \includegraphics[width=0.39\textwidth]{super_exp_d=2_u.png}
% \end{center}

% \caption{Same $\varphi(0)$ and from the smaller initial condition to the bigger. Line 1: exp from below.  Line 2: constant.  Line 3: linear.  Line 4: exp from above, with $d=1$. Line 5: exp from above, with $d=2$.}
% \label{init}
% \end{figure}
 \end{example}

\section{Conclusions}\label{Sec:Conclu}

This paper is concerned with  a  viability analysis and control synthesis of a delayed SIR model under an ICU state-constraint, where a constant delay represents an incubation/latency time.  As a by-product of the considered functional viability tools, we provided  feasible control actions driving the solutions to a safe set, according to a cost functional to be minimized.
Two scenarios were examined: in the first
the initial conditions are simply continuous, while in the other they satisfy a suitable  Lipschitz continuity assumption. In the latter case, we studied the dependence of the obtained results on the delay parameter. 
The theoretical developments have been illustrated via a numerical example inspired by the recent COVID-19 epidemic.
%As open route for future research, we plan to study the applicability of the proposed viability techniques to more general (epidemiological) delayed models,  and to characterize the sub-optimality degree of the proposed control policies.

    As a consequence of the viability analysis,  we obtained  that the forward invariant and viable zones in the delayed case are smaller than the corresponding ones 
for the undelayed problem,  and their amplitude is non increasing as a function of the delay. This suggests, as expected, that the controller should be conservative and anticipate his/her action with respect to what he/she would have done in the absence of a delay. 
Qualitatively, on the other hand, the  control action is  similar to the case without delay, and the greedy strategy turns out to be rather effective. In practice, under viable initial conditions, it consists in putting in action the algoritms 
\begin{itemize}
    \item \eqref{eq:ControlPolicyExplicitFeedback}, if no sufficient  information  is available on the very first part of the epidemic (for instance in case of a new unknown epidemic), or
    \item  \eqref{eq:ControlPolicyExplicitFeedback1}, if the epidemic behavior is sufficiently known since the beginning.
\end{itemize}
In particular, supposing that the initial state belong to the set $B\setminus(S_1\cup S_2)$ (with $S_2$ replaced by $S_{2,L,h}$ if we are opting for the second strategy),   the controller should
%\begin{enumerate}
% \item {\em do nothing} (i.e., $b=\beta$) if the initial epidemic state is in the forward invariant set (in this case $i(t)\le i_M$ for every $t\in \R_+$, meaning that the capacity of the health-care system is never exceeded);
%\item 
do nothing till the epidemic trajectory reaches the boundary regions $S_1$ (i.e., $i=i_M$)  or $S_2$, and only then
\begin{enumerate}
    \item if $S_1$ is reached, then approximate the saturation regime  $i=i_M$ by taking $b(t)=\min\big\{\beta,\frac{\gamma i_M}{s_\phi(t)i_\phi(t-h)}\big \}$  until reaching the immunity threshold $s=\frac{\gamma}{\beta}$,
    \item if $S_2$ is reached, then keep the trajectory close to the boundary by taking $b(t)=\min \big\{\beta,\beta_\star\frac{\psi_{L,h}(i(t))}{i(t-h)} \big\}$ (with $\psi_{L,h}(i)=i$ if we are opting for the first strategy) until the set $S_1$ is reached, and then proceed as in the previous step. 
\end{enumerate}
%\end{enumerate}
Of course if, instead, the initial state belongs to $S_1$ (resp.\ $S_2$ or $S_{2,L,h}$) then it is enough to implement the control action (1) (resp.\ (2)) since the beginning.
% \item 
% {\em do nothing until time $\tau_1$ at which $i(\tau_1)=i_M$ (i.e., the boundary $S_1$ is reached), then approximate the saturation regime  $i=i_M$ by taking $b=\min{\beta,...}$  until reaching the immunity threshold $\frac{\gamma}{\beta}$} if the initial conditions allow for this kind of control.
% It is worth noting that in the absence of a delay, that is if $h=0$, both the strategies above reduces to the optimal policy characterized in \cite{AvrFre22,AFG2024_corr,FGLX22}, which, in turn, consists in
% \begin{itemize}
% \item {\em do nothing} if the initial epidemic state is in the forward invariant set (i.e., it allows for an evolution in which $i(t)\le i_M$ for every $t\in \R_+$, meaning that the capacity of the health-care system is never exceeded); 

% \item {\em do nothing until time $\tau_1$ at which $i(\tau_1)=i_M$, then preserve the saturation $i=i_M$ until reaching the immunity threshold $\frac{\gamma}{\beta}$} if the initial conditions allow for this kind of control.
% \item otherwise, when the solution reaches the boundary of the considered viable zone, {\em actuate a lock-down (whose level depends on present and past values of the infected population) in order to reach $i=i_M$, and then preserve saturation} on this limiting level until reaching the immunity threshold.
% \end{itemize} 
This control scheme is provided in a state-feedback form and thus the controller only needs to observe the evolution of the epidemic in the previous $h$ time units to efficiently implement the strategy.

\appendix
 \section{Technical Proofs}
 In this Appendix we collect some technical results.

 \begin{lemma}\label{lemma:WellPosednessSoltuions}
Assume that $ \wt F(\phi)\in \cD_\cB(\phi)$ for any $\phi\in \cB$.  Then, there exists a solution $x_\phi$ to the Cauchy problem~\eqref{eq:ClosedLoop}. Moreover, any maximal solution is global, i.e., $\text{dom}(x_\phi)=[-h,+\infty)$, for all $\phi\in \cB$, and it holds that $S(t)x_\phi\in \cB,\;\;\forall t\in\R_+$. A similar result holds when $\cB$ is replaced by $\cB_{L,h}$, and $\wt F$ is replaced by $F_{L,h}$.
 \end{lemma}
\begin{proof}
To prove the existence of solutions, we consider first a convex regularization of the (possibly discontinuous) map $\wt F:\cB\to \R^2$, by considering the set-valued map $\wt G:\cB\toS \R^2$, defined by
\begin{equation}\label{wtGdef}
\wt G(\phi)=\begin{cases}
\{\wt F(\phi)\}=\{f(\phi,\beta)\},&\text{ if } \phi(0)\in B\setminus (S_1\cup S_2),\\
\text{co}\left \{f(\phi,\beta), \wt F(\phi)\right\},&\text{ if } \phi(0)\in S_1\cup S_2.
\end{cases}
\end{equation}
Note that if $\phi(0)\in S_1$, we can also write $\wt G(\phi)=\text{co}\{f(\phi,\beta), f(\phi, \min\{\beta,\frac{\gamma i_M}{s_\phi(0)i_\phi(-h)} \}) \}$, while, if  $\phi(0)\in S_2$ we have $\wt G(\phi)=\text{co}\{f(\phi,\beta), f(\phi, \min\{\beta,\beta_\star\frac{i_M}{i_\phi(-h)} \}) \}$. The map $\wt G:\cB\toS \R^2$ is obtained by considering the  Krasovskii regularization (see~\cite[Definition 2.2]{Hajek} for the definition for finite dimensional maps) of the  $\wt F:\cB \to \R^2$ in $\cB$, defined by
\[
\wt G(\phi)=\bigcap_{\varepsilon >0} \text{co} \{\wt F(\varphi)\;\;\vert\;\varphi\in \mathds{B}_\infty(\phi,\varepsilon)\cap \cB\},
\]
where $\mathds{B}_\infty(\phi,\varepsilon):=\{\varphi\in \cC\;\vert\;\sup_{t\in [-h,0]}|\varphi(t)-\phi(t)|_{\max}\leq \varepsilon\}$.
By definition, $\wt G:\cB \to \R^2$ has non-empty, compact and convex values, and $\wt F(\phi)\in \wt G(\phi)$ for all $\phi\in \cB$. We now prove that $\wt G$ is also upper semicontinuous, i.e., 
\[
\forall\, \phi\in \cB,\;\forall \;\varepsilon>0\;\exists \,\delta>0 \text{ such that } \varphi\in \mathds{B}_\infty(\phi,\delta)\cap\mathcal{B}\;\Rightarrow\;\wt G(\varphi)\subseteq \wt G(\phi)+B_{|\cdot|_{\max}}(0,\varepsilon).
\]
Let us take $\phi\in \cB$ and proceed by cases.
\begin{enumerate}[leftmargin=*]
    \item
If $\phi(0)\in B \setminus (S_1\cup S_2)$, the conclusion easily follows  by observing that $\wt G(\phi)=\{f(\phi,\beta)\}$, the map $f(\cdot,\beta):\cB\to \R^2$ is continuous with respect to $\|\cdot\|_\infty$ in $\cB$ and  the set $B \setminus (S_1\cup S_2)$ is  relatively open in $B$.
\item If $\phi(0)\in S_1\setminus S_2$, we distinguish the following two sub-cases. 
\begin{enumerate}
    \item 
If $i_\phi(-h)=0$, {there exists $\delta>0$ such that, if $\|\varphi-\phi\|_\infty\leq \delta$} then
\[
{\beta\leq \frac{\gamma i_M}{(\frac{\gamma}{\beta_\star}+\delta)\delta }\leq}\frac{\gamma i_M}{s_\varphi(0)i_\varphi(-h)}, 
\]
and thus $b(\varphi)=\beta$, implying $\wt G(\varphi)=\{f(\varphi,\beta)\}$, and concluding the proof in this case.
\item In the case $i_\phi(-h)>0$, we note that for any $\varepsilon'>0$ there exists $\delta'>0$ such that $\|\phi-\varphi\|_\infty\leq \delta'$ implies 
\[
\Big|\frac{\gamma i_M}{s_\phi(0)i_\phi(-h)}-\frac{\gamma i_M}{s_{\varphi}(0)i_{\varphi}(-h)}\Big|\leq \varepsilon'.
\]
Given $\varepsilon>0$, by continuity of $f:\cC\times U\to \R^2$,  of the $\min$ operator, and by the previous inequality,  there exists a $\delta>0$ such that $\|\phi-\varphi\|_\infty\leq \delta$ implies 
\[
\begin{aligned}
\wt G(\varphi)&= \text{co}\{f(\varphi,\beta), f(\varphi, \min\{\beta,\frac{\gamma i_M}{s_{\varphi}(0)i_{\varphi}(-h)} \}) \}
\\&\subseteq \text{co}\left\{f(\phi,\beta)+B_{|\cdot|_{\max}}(0,\varepsilon), f(\phi, \min\{\beta,\frac{\gamma i_M}{s_\phi(0)i_\phi(-h)} \}+B_{|\cdot|_{\max}}(0,\varepsilon) \right\}
\\&=\text{co}\{f(\phi,\beta), f(\phi, \min\{\beta,\frac{\gamma i_M}{s_\phi(0)i_\phi(-h)} \}) \}+B_{|\cdot|_{\max}}(0,\varepsilon)=\wt G(\phi)+B_{|\cdot|_{\max}}(0,\varepsilon),
\end{aligned}
\]
where we used the fact that, for any $z_1,z_2\in \R^n$ and any convex set $K\subset \R^n$ it holds that 
\[
\text{co}\{z_1+K,z_2+K\}= \text{co}\{z_1,z_2\}+K.
\]
We have thus proved the upper semicontinuity in this case.
\end{enumerate}
\item
Let $\phi(0)\in S_2\setminus S_1$. The case $i_\phi(-h)=0$ can be treated as in the previous case.
If $i_\phi(-h)>0$, the claim follows by continuity of the function $i\mapsto \beta_\star\frac{i_M}{i}$ when $i\neq 0$, by adapting the argument of the previous case.
\item 
For $\phi(0)\in S_1\cap S_2=\{(\frac{\gamma}{\beta_\star},i_M)\}$, it suffices to recall that in this case we have 
\[
\frac{\gamma i_M}{s_\phi(0)i_\phi(-h)}=\beta_\star\frac{i_M}{i_\phi(-h)},
\]
and the same continuity argument can thus be applied.
\end{enumerate}

Summarizing, the set-valued map $\wt G:\cC\to \R^n$ is upper semicontinuous with non-empty, compact and convex values.  By assumption, $\wt F(\phi)\in \cD_\cB(\phi)$ for all $\phi\in \cB$.
Since $\wt F(\phi)\in \wt G(\phi)$ for all $\phi\in \cB$, this implies
\[
\wt G(\phi)\cap D_\cB(\phi)\neq \varnothing \quad \forall \,\phi\in \cB.
\]
We can thus apply~\cite[Theorem 1.1]{Haddad81} (which is the local version of Theorem~\ref{prop:ViabilityAubin}) proving that, for any $\phi\in \cB$ there exist a $\tau>0$ and  a function  $x_\phi:[-h,\tau)\to \R^2$   satisfying 
\begin{equation}\label{eq:EnlargedClosedLoop}
\begin{aligned}
&x'_\phi(t)\in \wt G(S(t)x_\phi) \ \text{ for a.e.\ }t\in [0,\tau), \\
&S(0)x_\phi=\phi,\\
&S(t)x_\phi\in \cB\ \  \forall\, t\in [0,\tau).\\
\end{aligned}
\end{equation}
The fact that $\tau=+\infty$, i.e. that maximal solutions are defined on $[-h,+\infty)$, follows again by viability analysis, since $B$ is compact and thus solutions cannot explode in finite time, see~\cite[Page 12, Proof of Theorem 1.1]{Haddad81}.

To conclude the proof we show that, for any $\phi\in \cB$,  the only viable direction in $\wt G(\phi)$ is given by the vector $\widetilde F(\phi)$, that is 
\begin{equation}\label{ftgtdb}
\{\wt F(\phi)\}=\wt G(\phi)\cap D_\cB(\phi)\ \forall \,\phi\in \cB, 
\end{equation}
thus proving that any solution to~\eqref{eq:EnlargedClosedLoop} is a solution to~\eqref{eq:ClosedLoop}.

By definition of $\wt G$ (see \eqref{wtGdef} and the two lines after) the claim is trivially true if, either,
\begin{itemize}
\item $\phi(0)\in B\setminus (S_1\cup S_2)$, or
\item $\phi(0)\in S_1$ and $\beta\leq \frac{\gamma i_M}{s_\phi(0)i_\phi(-h)}$, or
\item $\phi(0)\in S_2$ and $\beta\leq \beta_\star\frac{i_M}{i_\phi(-h)}$.
\end{itemize}
 We then proceed  by analyzing the following two remaining cases, only. 
 To this aim it is useful to note that, by Lemma \ref{lemma:ConesCharacterization}, the claimed proposition \eqref{ftgtdb} is equivalent to
\begin{equation}\label{ftgtTb}
\{\wt F(\phi)\}=\wt G(\phi)\cap T_\cB(\phi(0))\ \forall \,\phi\in \cB.
\end{equation}
\begin{enumerate}[leftmargin=*]
    \item Let $\phi(0)\in S_1$ and $\beta> \frac{\gamma i_M}{s_\phi(0)i_\phi(-h)}$. Considering the vector $v=(0,1)$, normal to $B$ at $\phi(0)$, and recalling \eqref{eq:DefnDelSistema}, we have 
\[
%\inp{v}{f(\phi,b(\phi))}=
\inp{v}{f(\phi, \frac{\gamma i_M}{s_\phi(0)i_\phi(-h)})}
=0,
\]
and 
\[
\inp{v}{f(\phi,\beta)}= \beta s_\phi(0)i_\phi(-h)-\gamma i_\phi(0)>0.
\]
Hence, by convexity, $\inp{v}{g}>0$ for all $g\in \wt G(\phi)$, $g\neq \widetilde F(\phi)$. This proves \eqref{ftgtTb}, and hence  \eqref{ftgtdb}, in the current case.
\item If $\phi(0)\in S_2$ and $\beta> \beta_\star\frac{i_M}{i_\phi(-h)}$, we can argue as in the first part of the proof of Theorem~\ref{thm:GreedyControlPolicy}, case (c).  
There, we proved that, denoted by $w\in \R^2$ the unique normal vector (modulo positive scalar multiplication) to $B$ at $\phi(0)$, it holds
\[
\inp{w}{f(\phi,b(\phi))}=\inp{w}{f(\phi, \beta_\star\frac{i_M}{i_\phi(-h)})}=0,
\]
and $\inp{w}{f(\phi,\beta)}>0$. Now, by convexity, we have $\inp{w}{f}>0$ for all $f\in \wt G(\phi)$, $f\neq \widetilde F(\phi)$, concluding the proof.
\end{enumerate}
\end{proof}

\begin{lemma}\label{lemma:ConvexityAuxiliaryLocallyLipschitz}
For any value of $\gamma>0$, $\beta>\beta_\star>0$, $0<i_M\leq 1$ and for any $h>0$, the functions $\Gamma_{\beta,L,h}: [\frac{\gamma}{\beta},\widehat s_{\beta,L,h})\to \R_+$ and $\Gamma_{\beta_\star,L,h}: [\frac{\gamma}{\beta_\star},\widehat s_{\beta_\star,L,h})\to \R_+$ defined in Subsection~\ref{Subsec:locallyLispchitz} are strictly decreasing and concave.  Moreover, the sets $A_{L,h}$ and $B_{L,h}$ defined in Theorem~\ref{lemma:ViabilityDelayLipschitz} are convex.
\end{lemma}
\begin{proof}
 We recall that $\Gamma_{\beta,L,h}:[\frac{\gamma}{\beta},\widehat s_{\beta,L,h})\to \R$ is defined as the graph in the plane $(s,i)$ of the solution of~\eqref{eq:Non-LinearBeta} with initial condition $(s(0),i(0))=(\frac{\gamma}{\beta},i_M)$. Since $s\psi_{L,h}(i)>0$ for all $s\in [\frac{\gamma}{\beta},\widehat s_{\beta,L,h})$ and $i\in [0,+\infty)$, we can divide the equations in~\eqref{eq:Non-LinearBeta} by this quantity,  obtaining that $\Gamma_{\beta,L,h}$ is solution to the differential equation
\[
\frac{d}{ds}\,i(s)=-1+\frac{\gamma i(s)}{\beta s\psi_{L,h}(i(s))},
\]
with initial condition $i(\frac{\gamma}{\beta})=i_M$, in the interval $[\frac{\gamma}{\beta},\widehat s_{\beta,L,h})$. For simplicity, in what follows we write $\Gamma_{\beta,L,h}(s):=i(s)$, for all $s\in  [\frac{\gamma}{\beta},\widehat s_{\beta,L,h})$.

First we note that $\frac{di}{ds}(\frac{\gamma}{\beta})=0$ and
\[
\frac{d^2i}{ds^2}(\frac{\gamma}{\beta})=-\frac{1}{\gamma}<0.
\]
This implies that  $i$ is strictly decreasing in a right-neighborhood $N$ of $s(0)=\frac{\gamma}{\beta}$ and, hence,  $i(s)=\Gamma_{\beta,L,h}(s)<i_M$ for all $s\in N$.
On the other hand, as long as $i\leq i_M$ and  $s>\frac{\gamma}{\beta}$, we have
\[
\frac{di}{ds}(s)\leq-1+\frac{\gamma}{\beta s}<0,
\] 
since $\frac{i}{\psi_{L,h}(i)}\leq 1$ for all $i\leq i_M$. This implies  that 
$i$ is strictly decreasing in $[\frac{\gamma}{\beta}, \widehat s_{\beta,L,h})$, and by definition of $\widehat s_{\beta,L,h}$ we also have that $i(s)>0$ for all $s\in [\frac{\gamma}{\beta}, \widehat s_{\beta,L,h})$. We now show that its derivative is also decreasing, proving concavity. By definition of $\psi_{L,h}(i(s))$, we have
\[
\frac{i(s)}{ \psi_{L,h}(i(s))}=\begin{cases}
\frac{i(s)}{i_M},\;\;\;&\text{if }i_M-Lh\leq i(s)\leq  i_M, \\[1ex]
\frac{i(s)}{i(s)+Lh},\;\;\;&\text{if } 0\leq i(s)\leq i_M-Lh.
\end{cases}
\]

We note that both the functions $z\mapsto \frac{z}{i_M}$ and $z\mapsto \frac{z}{z+Lh}$ are strictly increasing in $[0,+\infty)$; since we have proved that $i(s)$ is strictly decreasing in $[\frac{\gamma}{\beta}, \widehat s_{\beta,L,h})$, so is  $s\mapsto \frac{i(s)}{\psi_{L,h}(i(s))}$. Then the function
\[
s\mapsto \frac{di}{ds}(s)= -1+\frac{\gamma i(s)}{\beta s\psi_{L,h}(i(s))},
\]
is also decreasing, concluding the proof of concavity of $\Gamma_{\beta,L,h}$ in $[\frac{\gamma}{\beta}, \widehat s_{\beta,L,h})$.

% Since $\Gamma_{\beta,L,h}$ is strictly decreasing and concave in $[\frac{\gamma}{\beta},\widehat s_{\beta,L,h})$, and  $\lim_{s\to \widehat s_{\beta,L,h}^{\,-}}\Gamma_{\beta,L,h}(s)\geq 0$ then we have that  $\widehat s_{\beta,L,h}$ is finite.
% We can thus extend by continuity $\Gamma_{\beta,L,h}$ on $[\frac{\gamma}{\beta},\widehat s_{\beta,L,h}]$ by setting $\Gamma_{\beta,L,h}(\widehat s_{\beta,L,h})=0$ {\color{red} e se il limite non \'e $0$? In ogni caso non sta scritto nell'enunciato.}.

The set $A_{L,h}$ defined in~\eqref{eq:DefinitionAElleAcca}  is then convex, since defined as the intersection of the convex set $T$ with the hypo-graph of a concave function.

The proof for $\Gamma_{\beta_\star,L,h}: [\frac{\gamma}{\beta_\star},\widehat s_{\beta_\star,L,h})\to \R_+$ is similar.
\end{proof}

\begin{lemma}\label{lemma:AppendixGronwall}
Given an interval $I=[0,\tau)$ (with, possibly, $\tau=+\infty$), consider a continuous function $g:\R^n \to \R^n$ and  $x_0\in \R^n$. Suppose there exists $x:I\to \R^n$, solution to
\[
 x'(t)=g(x(t)),\;\;\;\;x(0)=x_0,
\] 
 and that there exists a compact set $Q\subset \R^n$ such that $x(t)\in Q$ for all $t\in I$. 
For a given $a>0$, for any $h\in (0,a)$  consider  $g_h:\R^n\to \R^n$  such that:
\begin{itemize}[leftmargin=*]
    \item $g_h$ is uniformly converging to $g$ in $Q$ as $h\to 0$,
    \item $g_h$ are uniformly locally Lipschitz, i.e.,  for any compact set $K\subset \R^n$ there exists $M_K\geq 0$ such that 
\begin{equation}\label{eq:UniformGlobalLispchitz}
|g_h(x_1)-g_h(x_2)|\leq M_K|x_1-x_2|,\;\;\;\forall h\in (0,a),\;\forall x_1,x_2\in K.
\end{equation}
%(for a certain norm $|\cdot|$ of $\R^n$),
\item  $g_h$ are uniformly sublinear in norm, i.e., there exist $A>0, B>0$ such that
\begin{equation}\label{eq:SublineaRBound}
|g_h(x)|\leq A|x|+B,\;\;\;\forall h\in (0,a),\;\forall \,x\in \R^n.
\end{equation}
\end{itemize}
Let us denote by $y_h:I\to \R^n$ the solution to
\[
y'(t)=g_h(y(t)),\;\;\;y(0)=x_0.
\]
Then $g$ is Lipschitz and sublinear in norm in $Q$, and  we have 
\[
\lim_{h\to 0^+}|x(t)-y_h(t)|=0\;\;\;\forall \;t\in I.
\]
Moreover, for every compact interval $J\subset I$, there exists a sequence $h_n\to 0$ such that $y_{h_n}$ converges uniformly to $x$ on $J$.
\end{lemma}

\begin{proof}
Let us fix $t\in I$. By~\eqref{eq:SublineaRBound}, using  Grönwall's  inequality (see~\cite[Lemma 2.7]{Teschl12}) we have that 
\begin{equation}\label{eq:UnifBounded}
|y_h(t)|\leq \bar R:=|x_0|e^{At}+\frac{B}{A}(e^{At}-1), \,\;\forall h\in (0,a).
\end{equation}
Given the compact set $K=Q \cup B_{\|\cdot\|}(0,\bar R)$, consider $M_K>0$ as in~\eqref{eq:UniformGlobalLispchitz}. 
Now, for all $h\in (0,a)$, we have
\[
\begin{aligned}
	|y_h(t)-x(t)|&\leq \int_0^t |g_h(y_h(\theta))-g(x(\theta))|\,d\theta\\&\leq  \int_0^t |g_h(y_h(\theta))-g_h(x(\theta))|\,d\theta+  \int_0^t |g_h(x(\theta))-g(x(\theta))|\,d\theta\\& \leq M_K \int_0^t |y_h(\theta)-x(\theta)|\,d\theta+  \int_0^t |g_h(x(\theta))-g(x(\theta))|\,d\theta
\end{aligned}
\]
By uniform convergence, for every $\varepsilon>0$ there exists  $\overline h\in (0,a)$ such that, for any $h\leq \overline h$, it holds that (recall that $t$ is fixed)
\[
	|y_h(t)-x(t)|\leq M_K \int_0^t |y_h(\theta)-x(\theta)|\,d\theta+\varepsilon.
\]
Applying again the Gr\"onwall's Inequality (\cite[Lemma 2.7]{Teschl12}) this implies that 
\[
	|y_h(t)-x(t)|\leq\varepsilon e^{M_Kt}.
\]
We have, thus, proved that
\[
	\lim_{h\to 0}|y_h(t)-x(t)|=0, 
\]
as required. 
Let us now take a compact subinterval $J\subset I$. By~\eqref{eq:UnifBounded}, the functions $y_h$ are uniformly bounded on~$J$, i.e.,  there exists $C>0$ such that $|y_h(t)|\leq C$ for all $h\in (0,a)$ and  for all $t\in J$. Moreover, by~\eqref{eq:SublineaRBound} we have $|g_h(t)|\leq AC+B$ for all  for all $h\in (0,a)$ for all $t\in J$. In other words, the derivatives of the functions $y_h$ are uniformly bounded in $J$ and  this implies that the functions $y_h$ are equicontinuous on $J$. By Ascoli-Arzelà's  theorem, there exists a subsequence $y_{h_n}$ that uniformly converges on $J$. By the pointwise convergence proven above, we then have $y_{h_n}\to x$ uniformly on $J$, concluding the proof.
\end{proof}

\noindent
{\bf Acknowledgements.}  
DB is member of GNCS-INdAM. LF is member of GNAMPA--INdAM.

\

 \bibliography{bib_SIR_delays} 
 \bibliographystyle{abbrv}
 
\end{document}